\documentclass[reqno]{amsart}

\usepackage{amssymb}
\usepackage{bbm}
\usepackage{enumitem} 
\usepackage{nicefrac}
\usepackage[perpage,symbol*]{footmisc}
\usepackage{hyperref}

\DefineFNsymbols{daggers}{\dag\ddag{\dag\dag}{\ddag\ddag}}
\setfnsymbol{daggers}

\newcommand*{\mailto}[1]{\href{mailto: #1}{\nolinkurl{#1}}}
\newcommand{\arxiv}[1]{\href{http://arxiv.org/abs/#1}{arXiv: #1}}

\newcommand{\msc}[1]{\href{http://www.ams.org/msc/msc2020.html?t=&s=#1}{#1}}

\newtheorem{theorem}{Theorem}[section]
\newtheorem{lemma}[theorem]{Lemma}
\newtheorem{proposition}[theorem]{Proposition}
\newtheorem{corollary}[theorem]{Corollary}
\newtheorem{mainthm}{Theorem}[section]
\newtheorem{maincor}{Corollary}[section]

\theoremstyle{definition}
\newtheorem{definition}[theorem]{Definition}

\newtheorem{example}[theorem]{Example}
\newtheorem{mainexa}{Example}[section]

\theoremstyle{remark}
\newtheorem{remark}[theorem]{Remark}

\newcommand{\R}{{\mathbb R}}
\newcommand{\N}{{\mathbb N}}
\newcommand{\Z}{{\mathbb Z}}
\newcommand{\C}{{\mathbb C}}

\newcommand{\spr}[2]{\langle #1 , #2 \rangle}

\newcommand{\Qr}{\mathsf{q}}

\newcommand{\Wr}{\mathsf{w}}

\newcommand{\T}{\mathrm{T}}

\newcommand{\KIO}{\mathrm{K}}

\newcommand{\rI}{\mathrm{I}}

\newcommand{\E}{\mathrm{e}}
\newcommand{\I}{\mathrm{i}}
\newcommand{\sgn}{\mathrm{sgn}}
\newcommand{\supp}{\mathrm{supp}}
\newcommand{\tr}{\mathrm{tr}}
\newcommand{\im}{\mathrm{Im}}
\newcommand{\re}{\mathrm{Re}}
\newcommand{\loc}{{\mathrm{loc}}}
\newcommand{\cc}{{\mathrm{c}}}
\newcommand{\ac}{{\mathrm{ac}}}

\newcommand{\dis}{{\mathrm{dis}}}
\newcommand{\sing}{{\mathrm{s}}}

\newcommand{\cQ}{\mathcal{Q}}
\newcommand{\cE}{E}
\newcommand{\cF}{F}
\newcommand{\cR}{\mathcal{R}}

\newcommand{\gx}{\mathfrak{x}}
\newcommand{\gh}{\mathfrak{h}}

\newcommand{\ran}[1]{\mathrm{ran}(#1)}
\renewcommand{\ker}[1]{\mathrm{ker}(#1)}

\newcommand{\OO}{\mathcal{O}}
\newcommand{\oo}{o}

\newcommand{\ledot}{\,\cdot\,}
\newcommand{\redot}{\cdot\,}

\newcommand{\NL}{(0-)}
\newcommand{\NLz}{(z,0-)}

\newcommand{\qd}{{[1]}}

\newcommand{\Hast}{\dot{H}^1[0,L)}
\newcommand{\Hasto}{\dot{H}^1_0[0,L)}
\newcommand{\cH}{\mathcal{H}}
\newcommand{\dip}{\upsilon}
\newcommand{\eps}{\varepsilon}

\newcommand{\id}{{\mathbbm 1}}

\newcommand{\Peakons}{\mathcal{P}}
\newcommand{\CHdom}{\mathcal{D}}

\makeatletter
\newcommand{\dlmf}[1]{%
\cite[%
 \def\nextitem{\def\nextitem{, }}%
 \@for \el:=#1\do{\nextitem\expandafter\dlmf@eq@href\el...\end}%
]{dlmf}%
}
\def\dlmf@eq@href#1.#2.#3.#4\end{%
  \href{http://dlmf.nist.gov/#1.#2.E#3}{(#1.#2.#3)}}
\makeatother  

\numberwithin{equation}{section}

\begin{document}

\title[Trace formulas and inverse spectral theory]{Trace formulas and inverse spectral theory\\ for generalized indefinite strings}
 
\author[J.\ Eckhardt]{Jonathan Eckhardt}
\address{Department of Mathematical Sciences\\ Loughborough University\\ Epinal Way\\ Loughborough\\ Leicestershire LE11 3TU \\ UK}
\email{\mailto{J.Eckhardt@lboro.ac.uk}}

\author[A.\ Kostenko]{Aleksey Kostenko}
\address{Faculty of Mathematics and Physics\\ University of Ljubljana\\ Jadranska 19\\ 1000 Ljubljana\\ Slovenia\\ and Faculty of Mathematics\\ University of Vienna\\ Oskar-Morgenstern-Platz 1\\ 1090 Wien\\ Austria}
\curraddr{Faculty of Mathematics and Physics\\ University of Ljubljana\\ Jadranska 19\\ 1000 Ljubljana\\ Slovenia\\ and Institute for Analysis and Scientific Computing\\Vienna University of Technology\\Wiedner Hauptstra\ss e 8–10/101\\
1040 Wien\\ Austria}\email{\mailto{Aleksey.Kostenko@fmf.uni-lj.si}}


\thanks{{\it Research of A.K.\ is supported by the Austrian Science Fund (FWF) under Grant I-4600 and by the Slovenian Research Agency (ARIS) under Grants No.\ N1-0137 and P1-0291.}}


\begin{abstract}
Generalized indefinite strings provide a canonical model for self-adjoint operators with simple spectrum (other classical models are Jacobi matrices, Krein strings and $2\times 2$ canonical systems). 
 We prove a number of Szeg\H{o}-type theorems for generalized indefinite strings and related spectral problems (including Krein strings, canonical systems and Dirac operators). 
More specifically, for several classes of coefficients (that can be regarded as Hilbert--Schmidt perturbations of model problems), we provide a complete characterization of the corresponding set of spectral measures.
 In particular, our results also apply to the isospectral Lax operator for the conservative Camassa--Holm flow and allow us to establish existence of global weak solutions with various step-like initial conditions of low regularity via the inverse spectral transform. 
\end{abstract}

\maketitle

{\scriptsize
\tableofcontents}

\section{Introduction}\label{sec:Intro}

\subsection{Main results}\label{ss:01main}

A {\em generalized indefinite string} is a triple $(L,\omega,\dip)$ such that $L\in(0,\infty]$, $\omega$ is a real distribution in $H^{-1}_{\loc}[0,L)$ and $\dip$ is a positive Borel measure on $[0,L)$. 
Associated with such a generalized indefinite string is the ordinary differential equation  
  \begin{align}\label{eqnGISODE}
  -f'' = z\, \omega f + z^2 \dip f
 \end{align}
 on the interval $[0,L)$, where $z$ is a complex spectral parameter. 
 Spectral problems of this form go back at least to work of M.\ G.\ Krein and H.\ Langer from the 1970s on indefinite analogues of the classical moment problem~\cite{krla79,la76}.
 In the generality above, they were introduced in~\cite{IndefiniteString}, where we proved that they serve as a canonical model for self-adjoint operators with simple spectrum. 
 We are going to summarize some relevant facts about generalized indefinite strings in Section~\ref{sec:prelim} as far as they are needed in this article, but in order to state our main theorems below, let us note that there is a unique function $\Wr$ in $L^2_{\loc}[0,L)$ such that 
 \begin{align}
  \omega(h) = -\int_0^L \Wr(x)h'(x)dx
 \end{align}
 for all $h\in H^1_{\cc}[0,L)$, called the {\em normalized anti-derivative} of the distribution $\omega$. 
 
 One of the main objects in spectral theory for a generalized indefinite string $(L,\omega,\dip)$ is the {\em Weyl--Titchmarsh function} $m$ defined on $\C\backslash\R$ by 
\begin{align}
  m(z) = \frac{\psi'(z,0-)}{z\psi(z,0)},
\end{align}
where $\psi(z,\redot)$ is the unique (up to scalar multiples) nontrivial solution of the differential equation~\eqref{eqnGISODE} that lies in the homogeneous Sobolev space $\Hast$ and in $L^2([0,L);\dip)$. 
This Weyl--Titchmarsh function $m$ turns out to be a {\em Herglotz--Nevanlinna function}, that is, it is analytic, maps the upper complex half-plane $\C_+$ into the closure of the upper complex half-plane and satisfies the symmetry relation
 \begin{align}
  m(z^\ast)^\ast = m(z)
 \end{align}
 for all $z\in\C\backslash\R$ (here and henceforth, we will use $z^\ast$ to denote the complex conjugate of a complex number $z$). 
 Even more, we established in~\cite[Section~6]{IndefiniteString} that {\em the map $(L,\omega,\dip)\mapsto m$ is a homeomorphism between the set of all generalized indefinite strings (equipped with a reasonable topology) and the set of all Herglotz--Nevanlinna functions (equipped with the topology of locally uniform convergence)}. 
 The main results of the present article are a couple of Szeg\H{o}-type theorems for this correspondence. 
 In the first place, we will prove the following theorem, which gives a characterization of all Weyl--Titchmarsh functions for a class of generalized indefinite strings that can be understood as (relative) Hilbert--Schmidt perturbations of a certain explicitly solvable model example. 
 Throughout this article, we will let $\alpha$ and $\beta$ be arbitrary positive constants.

 \begin{mainthm}\label{thm:KSforExB1}
  A Herglotz--Nevanlinna function $m$ is the Weyl--Titchmarsh function of a generalized indefinite string $(L,\omega,\dip)$ with $L=\infty$ and  
    \begin{align}\label{eqnCondS1}
      \int_0^\infty \biggl(\Wr(x)-c-\frac{x}{1+2\sqrt{\alpha}x}\biggr)^2 x\, dx +  \int_{[0,\infty)} x\, d\dip(x) < \infty
    \end{align}
    for some constant $c\in\R$ if and only if all the following conditions hold:
    \begin{enumerate}[label=(\roman*), ref=(\roman*), leftmargin=*, widest=iii]
    \item\label{itmKSc1} The function $m$ has a meromorphic extension to $\C\backslash[\alpha,\infty)$ that is analytic at zero.
    \item\label{itmKSc2} The negative poles $\sigma_-$ and the positive poles $\sigma_+$ of $m$ in $(-\infty,\alpha)$ satisfy
    \begin{align}\label{eqnLT-I}
    \sum_{\lambda\in \sigma_-} \frac{1}{|\lambda|^{\nicefrac{3}{2}}} + \sum_{\lambda\in \sigma_+}(\alpha-\lambda)^{\nicefrac{3}{2}} <\infty.
    \end{align}
    \item\label{itmKSc3} The boundary values of the function $m$ satisfy\footnote{Remember that $m(\lambda+\I0) = \lim_{\varepsilon\downarrow 0} m(\lambda + \I \varepsilon)$ exists for almost all $\lambda\in\R$ because $m$ is of bounded type in the upper complex half-plane.} 
    \begin{align}\label{eqnSzego-I}
 \int_{\alpha}^{\infty}  \frac{\sqrt{\lambda-\alpha}}{\lambda ^3}\log(\im\, m(\lambda+\I0)) d\lambda >- \infty.
    \end{align}
    \end{enumerate}
    \end{mainthm}

\begin{remark}
A few remarks about the conditions in Theorem~\ref{thm:KSforExB1} are in order:
\begin{enumerate}[label=(\alph*), ref=(\alph*), leftmargin=*, widest=e]
\item Condition~\ref{itmKSc1} says that the essential spectrum of every generalized indefinite string $(L,\omega,\dip)$ with $L=\infty$ and~\eqref{eqnCondS1} for some constant $c\in\R$ is contained in the interval $[\alpha,\infty)$. 
This can be considered as a consequence of Weyl's theorem on compact perturbations. 
In fact, we are going to show in Section~\ref{sec:CompactPert} that it is possible to interpret such generalized indefinite strings as relatively compact perturbations (actually, of Hilbert--Schmidt class) of a particular model example whose spectrum coincides with the interval $[\alpha,\infty)$; see Example~\ref{exaalpha}.

\item Analyticity of the Weyl--Titchmarsh function $m$ at zero alone has strong consequences on the coefficients of a generalized indefinite string $(L,\omega,\dip)$.
In fact, it implies that $L=\infty$ and that the inverse of an associated linear relation $\T$ is a bounded linear operator. 
 We recently proved in~\cite[Proposition~5.2~(i)]{DSpec} that the latter is equivalent to 
\begin{align}
\limsup_{x\rightarrow \infty}x\int_x^\infty (\Wr(s) - \Wr_0)^2ds + x\int_{[x,\infty)}d\dip  <\infty 
\end{align}
for some constant $\Wr_0\in\R$, which is then necessarily given by 
\begin{align}
\Wr_0  = \lim_{x\rightarrow\infty}\frac{1}{x}\int_0^x \Wr(s)ds.
\end{align}
Moreover, the constant $\Wr_0$ turns out to be nothing but the value of $m$ at zero as we will see in Section~\ref{secmatzero}.
In particular, this shows that under the conditions in Theorem~\ref{thm:KSforExB1} one has 
\begin{align}\label{eq:intromat0}
m(0) = c + \frac{1}{2\sqrt{\alpha}},
\end{align}
which relates the constant $c$ in~\eqref{eqnCondS1} explicitly to the function $m$. 

\item Condition~\ref{itmKSc2} is a Lieb--Thirring bound on the eigenvalues below $\alpha$, which implies that negative eigenvalues may accumulate only at $-\infty$ and that positive eigenvalues in the interval $(0,\alpha)$ may accumulate only at $\alpha$. 
 In fact, we will obtain a sharp Lieb--Thirring inequality in Corollary~\ref{cor:LT01new}.

\item Condition~\ref{itmKSc3} implies that the absolutely continuous spectrum is essentially supported on the interval $(\alpha,\infty)$, which has been shown to be necessary for~\eqref{eqnCondS1} to hold for some constant $c\in\R$ before in~\cite[Theorem~3.2]{ACSpec}. 
 Since the pointwise boundary values of $\im\, m$ depend only on the absolutely continuous part of the spectrum, we see that there is no restriction on the support and structure of the singular spectrum in $[\alpha,\infty)$. 
\end{enumerate}
\end{remark}

\begin{remark}\label{rem:LaguerreIntro}
The Herglotz--Nevanlinna function
\begin{align}\label{eq:LaguerreWT}
m_\alpha(z) = \E^{\alpha - z} E_1(\alpha-z) = \int_\alpha^{\infty} \frac{\E^{-(\lambda-\alpha)}}{\lambda-z}d\lambda, 
\end{align}
where $E_1$ denotes the principal value of the generalized exponential integral~\dlmf{8.19.2}, satisfies the conditions~\ref{itmKSc1}, \ref{itmKSc2} and~\ref{itmKSc3} of Theorem~\ref{thm:KSforExB1} (indeed, the function $m_\alpha$ is nothing but the Stieltjes transform of the shifted Laguerre weight). 
 In this case, it is possible to express the coefficients of the corresponding generalized indefinite string via the Laguerre polynomials. 
  One may view Theorem~\ref{thm:KSforExB1} as a characterization of a certain class of perturbations of the Laguerre operator, which is usually expressed in $\ell^2$ as a Jacobi matrix (see Example~\ref{ex:Laguerre}). 
  Let us mention that the Laguerre operator has received some attention recently~\cite{kkt18, ko21, krso15}, mainly because of its appearance in the study of nonlinear waves in $(2+1)$-dimensional noncommutative scalar field theory~\cite{aca, cfw03}.  
\end{remark}

Theorem~\ref{thm:KSforExB1} can easily be stated as well in terms of the {\em spectral measure} $\rho$ of a generalized indefinite string $(L,\omega,\dip)$, which is defined such that 
\begin{align}\label{eqnmIntRep}
  m(z) = c_1 z + c_2 - \frac{1}{Lz} +  \int_\R \frac{1}{\lambda-z} - \frac{\lambda}{1+\lambda^2}\, d\rho(\lambda), 
\end{align}
where $c_1$, $c_2\in\R$ are some constants with $c_1\geq0$ and $\rho$ is a positive Borel measure on $\R$ with $\rho(\{0\})=0$ that satisfies  
  \begin{align}\label{eqnrhoPoisson}
    \int_\R \frac{d\rho(\lambda)}{1+\lambda^2} < \infty. 
  \end{align}
  The measure $\rho$ is indeed a spectral measure for an associated self-adjoint linear relation $\T$ (which we are going to define properly in Section~\ref{sec:prelim}). 
Notice that unlike the Weyl--Titchmarsh function $m$, the spectral measure $\rho$ does not determine the generalized indefinite string $(L,\omega,\dip)$ uniquely. 
 More specifically, it does not determine the length $L$ and assuming that $L$ is known, it determines $\Wr$ only up to an additive constant and $\dip$ up to a point mass at zero (see Remark~\ref{rem:uniquerho} below). 

 \begin{maincor}\label{cor:KSforExB1}
  A positive Borel measure $\rho$ on $\R$ with~\eqref{eqnrhoPoisson} is the spectral measure of a generalized indefinite string $(L,\omega,\dip)$ with $L=\infty$ and~\eqref{eqnCondS1} for some constant $c\in\R$ if and only if both of the following conditions hold:
    \begin{enumerate}[label=(\roman*), ref=(\roman*), leftmargin=*, widest=ii]
\item\label{itmCorKSc1} The support of $\rho$ is discrete in $(-\infty,\alpha)$, does not contain zero and satisfies
  \begin{align}\label{eqnLT-Irho}
    \sum_{\lambda\in \supp(\rho) \atop \lambda<0} \frac{1}{|\lambda|^{\nicefrac{3}{2}}} + \sum_{\lambda\in \supp(\rho)\atop 0<\lambda<\alpha}(\alpha-\lambda)^{\nicefrac{3}{2}} <\infty.
  \end{align}
\item\label{itmCorKSc2} The absolutely continuous part $\rho_\ac$ of $\rho$ on $(\alpha,\infty)$ satisfies
    \begin{align}\label{eqnSzego-Irho}
       \int_{\alpha}^{\infty}  \frac{\sqrt{\lambda-\alpha}}{\lambda ^3} \log\biggl(\frac{d\rho_\ac(\lambda)}{d\lambda}\biggr) d\lambda >- \infty.
    \end{align}
    \end{enumerate}
    \end{maincor}

  Results like Theorem~\ref{thm:KSforExB1} and Corollary~\ref{cor:KSforExB1} have their origin in the work of G.\ Szeg\H{o} on orthogonal polynomials on the unit circle (see~\cite{simSzego} for a detailed historical account and further references) and have been pioneered by R.\ Killip and B.\ Simon for Jacobi matrices~\cite{kisi03} and Schr\"odinger operators~\cite{kisi09}. 
  In both cases, the proof relies on two crucial ingredients: 
    (a) Continuous dependence of spectral data on the coefficients (for example, the spectral measure of a Jacobi matrix on the Jacobi parameters or the Weyl--Titchmarsh function on the potential for a Schr\"odinger operator) 
    and (b) a trace formula (or sum rule in the terminology of~\cite{kisi03,kisi09,simSzego}). 
   Of course, in addition one also needs a sophisticated direct and inverse spectral theory available for the respective operators.
    Unlike Jacobi matrices and Schr\"odinger operators, generalized indefinite strings were introduced only relatively recently in~\cite{IndefiniteString} and are far less well developed. 
    However, they are closely related to $2\times 2$ canonical systems~\cite{dB68,rem,rom} and this connection provides us with the necessary continuity property.
    In fact, this is widely known as continuity of the Krein--de Branges correspondence, a fundamental result of Krein--de Branges theory. 
    On the other hand, trace formulas are usually connected to conserved quantities of related completely integrable systems; the Toda lattice for Jacobi matrices and the Korteweg--de Vries equation for Schr\"odinger operators. 
    This role is taken by the Camassa--Holm equation~\cite{caho93} for generalized indefinite strings, whose conservation laws suggest suitable trace formulas.   
     However, the known conserved quantities for smooth classical solutions of the Camassa--Holm equation are insufficient and have to be amended. 
   Instead one has to consider more general conservative weak solutions, a particular kind of weak solutions of low regularity that allows continuation of solutions after blow-up.  
   Conservation laws for these solutions appear to be new and we are unaware of any previous derivations.
   In fact, we already employed a corresponding trace formula earlier when establishing results akin to the ones by P.~Deift and R.~Killip~\cite{deki99} for generalized indefinite strings in~\cite{ACSpec}. 
     Here we will show that this trace formula indeed holds in a strong sense, meaning that finiteness of one side implies finiteness of the other side; see~\eqref{eqnTFalphageq} and~\eqref{eqnTFalphaleq}.
     This readily proves Theorem~\ref{thm:KSforExB1}.  
     
 When compared to Jacobi matrices and Schr\"odinger operators, there are several more principal differences.
  First of all, the spectral problem~\eqref{eqnGISODE} is nonlinear in the spectral parameter.
  However, even without the quadratic term, the spectral parameter enters the problem in the `wrong' place. 
  This makes it difficult to view~\eqref{eqnGISODE} as an additive perturbation as  the underlying Hilbert space varies together with the coefficients. 
  We are going to clarify this situation in Section~\ref{sec:CompactPert} based on recent advances in~\cite{DSpec}. 
  A second issue lies in the fact that in contrast to Sch\"odinger operators even a relatively small alteration of the coefficients $\omega$ and $\dip$ (for example, changing them on a compact interval) can lead to perturbations that are not of trace class anymore (for Jacobi matrices the latter are even of finite rank). 
  In fact, this only occurs when the spectrum is not semi-bounded, which constitutes another difference to the Jacobi matrix and Schr\"odinger operator cases.
 On a technical level, this makes it necessary to consider suitably regularized perturbation determinants and infinite products that are only conditionally convergent in general.
 Finally, one more major difference to Jacobi matrices and Schr\"odinger operators lies in the crucial role played by the high-energy spectral asymptotics, a very well understood classical subject in both these cases. 
 For generalized indefinite strings, sufficiently strong asymptotics are not known in fact. 
 Instead, we will have to rely on a qualified understanding of spectral asymptotics at zero. 
 These will lead to our trace formulas as well as corresponding conservation laws for the conservative Camassa--Holm flow. 
 
  In addition to Theorem~\ref{thm:KSforExB1}, we will also prove a similar result for another class of generalized indefinite strings that give rise to essential spectrum $(-\infty,-\beta]\cup[\beta,\infty)$ for some positive constant $\beta$. 
 To this end, we will write 
 \begin{align}\label{eqndiprhosing}
    \dip(B) = \int_B \varrho(x)^2 dx + \dip_\sing(B), 
 \end{align}
  where $\varrho$ is the (positive) square root of the Radon--Nikod\'ym derivative of $\dip$ with respect to the Lebesgue measure and $\dip_\sing$ is the singular part of $\dip$.
    Whereas Theorem~\ref{thm:KSforExB1} is motivated by the study of a particular phase space associated with the Hamiltonian of the Camassa--Holm equation, this result is relevant for another phase space arising in connection with the two-component Camassa--Holm system~\cite{chlizh06,hoiv11}. 

 \begin{mainthm}\label{thm:KSforExB2}
   A Herglotz--Nevanlinna function $m$ is the Weyl--Titchmarsh function of a generalized indefinite string $(L,\omega,\dip)$ with $L=\infty$ and  
   \begin{align}\label{eqnCondS2}
      \int_0^\infty  (\Wr(x)-c)^2 x\,dx +  \int_0^\infty   \biggl( \varrho(x) - \frac{1}{1+2\beta x} \biggr)^2 x\, dx + \int_{[0,\infty)} x\, d\dip_\sing(x) & < \infty
    \end{align}
  for some constant $c\in\R$ if and only if all the following conditions hold:
    \begin{enumerate}[label=(\roman*), ref=(\roman*), leftmargin=*, widest=iii]
    \item\label{itmKS2c1} The function $m$ has a meromorphic extension to $\C_+\cup(-\beta,\beta)\cup\C_-$ that is analytic at zero. 
    \item\label{itmKS2c2} The poles $\sigma_\dis$ of $m$ in $(-\beta,\beta)$ satisfy
    \begin{align}\label{eqnLT-II}
    \sum_{\lambda\in \sigma_\dis} (\beta-|\lambda|)^{\nicefrac{3}{2}} <\infty.
    \end{align}
    \item\label{itmKS2c3} The boundary values of the function $m$ satisfy 
    \begin{align}\label{eqnSzego-II}
 \int_{\R\backslash(-\beta,\beta)} \frac{\sqrt{\lambda^2-\beta^2}}{|\lambda|^3}\log(\im\, m(\lambda+\I0)) d\lambda  >- \infty.
    \end{align}
    \end{enumerate}
    \end{mainthm}

\begin{remark}
 Similar to Theorem~\ref{thm:KSforExB1}, condition~\ref{itmKS2c1} says that the essential spectrum is contained in the set $(-\infty,-\beta]\cup[\beta,\infty)$. 
 The constant $c$ in~\eqref{eqnCondS2} is related to the Weyl--Titchmarsh function $m$ via $m(0) = c$. 
  Condition~\ref{itmKS2c2} is a Lieb--Thirring bound on the eigenvalues in the interval $(-\beta,\beta)$ with a corresponding sharp Lieb--Thirring inequality given in Corollary~\ref{cor:LT02new} and condition~\ref{itmKS2c3} implies that the absolutely continuous spectrum is essentially supported on the set $(-\infty,-\beta]\cup[\beta,\infty)$ (a fact proved before in~\cite[Theorem~3.2]{ACSpec2}), whereas there is no restriction on the singular spectrum in $(-\infty,-\beta]\cup[\beta,\infty)$.
\end{remark}

Let us again state this result also in terms of the spectral measure.

 \begin{maincor}\label{cor:KSforExB2}
     A positive Borel measure $\rho$ on $\R$ with~\eqref{eqnrhoPoisson} is the spectral measure of a generalized indefinite string $(L,\omega,\dip)$ with $L=\infty$ and~\eqref{eqnCondS2} for some constant $c\in\R$ if and only if both of the following conditions hold:
    \begin{enumerate}[label=(\roman*), ref=(\roman*), leftmargin=*, widest=ii]
\item The support of $\rho$ is discrete in $(-\beta,\beta)$, does not contain zero and satisfies
  \begin{align}\label{eqnLT-IIrho}
    \sum_{\lambda\in \supp(\rho) \atop |\lambda|<\beta} (\beta-|\lambda|)^{\nicefrac{3}{2}} <\infty.
  \end{align}
\item The absolutely continuous part $\rho_\ac$ of $\rho$ on $(-\infty,-\beta)\cup(\beta,\infty)$ satisfies
    \begin{align}\label{eqnSzego-IIrho}
              \int_{\R\backslash(-\beta,\beta)}  \frac{\sqrt{\lambda^2-\beta^2}}{|\lambda|^3} \log\biggl(\frac{d\rho_\ac(\lambda)}{d\lambda}\biggr) d\lambda >- \infty.
    \end{align}
    \end{enumerate}
    \end{maincor}

\begin{remark}
 It is also possible to prove similar results for classes of generalized indefinite strings with finite length $L$.
   This can be achieved by means of the second transformation in Remark~\ref{rem:uniquerho}, which changes the length of a generalized indefinite string while only altering the residue of the Weyl--Titchmarsh function at zero. 
\end{remark}

\subsection{Completely integrable systems}\label{ss:CHintro}

  One of the main motivations for studying inverse spectral theory for Jacobi matrices and Schr\"odinger operators stems from their relevance for the Toda lattice and the Korteweg--de Vries equation. 
 In the same way, inverse spectral theory for generalized indefinite strings is important for nonlinear wave equations like the Hunter--Saxton equation and the Dym equation, but in particular the Camassa--Holm equation
 \begin{align}\label{eqnCH}
    u_{t} -u_{xxt}  = 2u_x u_{xx} - 3uu_x + u u_{xxx},
 \end{align}
 which was first found by B.\ Fuchsteiner and A.\ S.\ Fokas~\cite{fofu81} to be formally integrable with Hamiltonians given by
 \begin{align}\label{eqnCHHam}
  H_1 & = \int u^2 + u_x^2\, dx, & H_2 & =  \frac{1}{2}\int u^3 + uu_x^2\, dx,
   \end{align}
   where the domain of integration is either the real line $\R$ or the circle $\mathbb{T}=\R/\Z$.  
 Its formulation as a non-local conservation law
\begin{align}
\label{eqnCHweak}
  u_t + u u_x + P_x  = 0, 
\end{align}
 where the source term $P$ is defined (for suitable functions $u$) as the convolution 
 \begin{align}\label{eq:PequCHdef}
 P = \frac{1}{2}\E^{-|\cdot|} \ast \biggl(u^2 + \frac{1}{2} u_x^2\biggr), 
 \end{align} 
is reminiscent of the three-dimensional incompressible Euler equation. 
Regarding the hydrodynamical relevance of the Camassa--Holm equation, let us mention that it can be derived as an approximation to the Green--Naghdi equations, the governing equations for shallow water waves over a flat bed~\cite{caho93,joh02,laco09}. 

An intensive study of~\eqref{eqnCH} started with the discoveries of R.~Camassa and D.~Holm~\cite{caho93} (we only refer to a very brief selection of articles~\cite{besasz00, bkst09, brco07, co01, coes98, como00, cost00, hora07, mc03, mc04, mis, xizh00}). Indeed, it turned out that equation~\eqref{eqnCH} exhibits a rich mathematical structure. 
  For instance, the Bott--Virasoro group serves for equation~\eqref{eqnCHweak} as a symmetry group and thus the Camassa--Holm equation can be viewed as a geodesic equation with respect to a right invariant metric~\cite{mis}.   
   Among many celebrated models for shallow water waves (for example, the Korteweg--de Vries equation or the Benjamin--Bona--Mahoney equation), the significance of equation~\eqref{eqnCH} stems from the fact that it is the first and long sought-for model with the following three features: 
   (a) Complete integrability, as it admits a Lax pair formulation, 
   (b) solitons, including peaked ones (called {\em peakons}) and 
   (c) finite time blow-up of smooth solutions that resembles wave-breaking to some extent; see~\cite{coes98b}. 
  More specifically, singularities develop in a way that the solution $u$ remains bounded pointwise, while the spatial derivative $u_x$ tends to $-\infty$ at some points. 
  However, the $H^1$ norm of $u$ remains bounded (actually, the $H^1$ norm serves as a Hamiltonian for the Camassa--Holm equation and is thus conserved) and $u$ approaches a limit weakly in $H^1$ as the blow-up happens, which raises the natural question of continuation past the blow-up. 
  Indeed, it was found that the Camassa--Holm equation possesses global weak solutions~\cite{xizh00}, which are not unique however and hence continuation of solutions after blow-up is a delicate matter. 
 
  A particular kind of weak solutions are so-called {\em conservative solutions}, the notion of which was suggested independently in~\cite{brco07} and~\cite{hora07}.   
  These are weak solutions (in a sense to be made precise in Definition~\ref{def:weaksolCH}) of the system 
  \begin{align}
 \begin{split}\label{eqnmy2CH}
  u_t + u u_x + P_x & = 0, \\
  \mu_t + (u\mu)_x & = (u^3 - 2Pu)_x, 
 \end{split}
 \end{align}
 where the auxiliary function $P$ satisfies
 \begin{align}\label{eq:PequDef}
  P - P_{xx} & = \frac{u^2+ \mu}{2}.
 \end{align} 
We will call system~\eqref{eqnmy2CH} the {\em two-component Camassa--Holm system} because it includes the Camassa--Holm equation as well as its two-component generalization (see \cite{chlizh06, coiv08, esleyi07, hoiv11})
 \begin{align}
 \begin{split}\label{eqn2CH}
    u_{t} -u_{xxt} & = 2u_x u_{xx} - 3uu_x + u u_{xxx} - \varrho \varrho_x, \\
   \varrho_t & = -u_x \varrho - u\varrho_x,
 \end{split}
 \end{align}
where one just needs to set $\mu = u^2 + u_x^2 + \varrho^2$. 
The role of the additional positive Borel measure $\mu$ is to control the loss of energy at times of blow-up. 
More precisely, as the model example of a peakon--antipeakon collision illustrates (see~\cite[Section~6]{brco07} for example), at times when the solution blows up, the corresponding energy, measured by $\mu$, concentrates on sets of Lebesgue measure zero.
Solutions of this kind have been constructed by a generalized method of characteristics that relied on a transformation from Eulerian to Lagrangian coordinates and was accomplished for various classes of initial data in~\cite{brco07, grhora12, grhora12b, hora07, hora07c}.
 
From our perspective, conservative solutions are of special interest because they preserve the integrable structure of the Camassa--Holm equation and can be obtained by employing the inverse spectral transform method~\cite{ConservMP,ConservCH,Eplusminus}, at least in principle. 
 The underlying isospectral problem is of the form 
\begin{align}\label{eqnCHspec}
 -f'' + \frac{1}{4} f & = z\, \omega f + z^2 \dip f,  
\end{align}
where $\omega = u-u_{xx}$ and $\mu = u^2 + u_x^2 + \dip$. Despite its relevance and a large amount of articles, relatively little is known about the spectral problem~\eqref{eqnCHspec} so far when $\omega$ is allowed to change sign and $\dip$ to be not zero. 
In fact, under the necessary low regularity restrictions (the coefficient $\omega$ may be a real distribution and $\dip$ may be a positive measure), not even the well-developed theory about self-adjoint realizations covers this kind of spectral problem; see~\cite{bebrwe08, bebrwe12, CHPencil, MeasureSL, gewe14}, where only~\cite{CHPencil} includes the additional coefficient $\dip$. 
Consequently, most of the literature on inverse spectral theory restricts to the case when $\dip$ vanishes identically and $\omega$ is strictly positive and smooth.
Under these assumptions, the differential equation can be transformed into a standard potential form that is known from the Korteweg--de Vries equation and some conclusions may be drawn from this~\cite{besasz98, co01, cogeiv06, le04, mc03b}.  
In the general case, apart from the explicitly solvable finite dimensional case~\cite{besasz00, ConservMP, InvPeriodMP} and a class of coefficients that gives rise to purely discrete spectrum~\cite{ConservCH}, only insufficient partial uniqueness results~\cite{be04, bebrwe08, bebrwe12, bebrwe15, LeftDefiniteSL, CHPencil, TFCPeriod, IsospecCH} have been obtained so far.  

The spectral problem~\eqref{eqnCHspec} differs from the one for a generalized indefinite string in~\eqref{eqnGISODE} only by a constant coefficient term and can be turned into this form in several ways by employing a simple change of variables. 
For example, when~\eqref{eqnCHspec} is considered on the half-line $[0,\infty)$ with a real-valued function $u$ in $H^1_{\loc}[0,\infty)$ and a positive Borel measure $\dip$ on $[0,\infty)$, the spectral problem~\eqref{eqnCHspec} can be transformed into a generalized indefinite string $(L,\tilde{\omega},\tilde{\dip})$ with $L=\infty$ using the change of variables $x\mapsto \log(1+x)$. 
The details of this transformation can be found in~\cite[Section~7]{ACSpec} and are also briefly summarized in Section~\ref{secCH:01}.
 In this situation, one can also define a Weyl--Titchmarsh function $m$ for the spectral problem~\eqref{eqnCHspec} on $\C\backslash\R$ by 
\begin{align}
  m(z) = \frac{\psi'(z,0-)}{z\psi(z,0)},
\end{align}
where $\psi(z,\redot)$ is the unique (up to scalar multiples) nontrivial solution of the differential equation~\eqref{eqnCHspec} that lies in $H^1[0,\infty)$ and $L^2([0,\infty);\dip)$. 
 It turns out that this function coincides with the Weyl--Titchmarsh function of the corresponding generalized indefinite string $(L,\tilde{\omega},\tilde{\dip})$ up to a pole at zero. 
 As a consequence, it allows an integral representation of the form
 \begin{align}
  m(z) = c_1 z + c_2 - \frac{1}{2z} +  \int_\R \frac{1}{\lambda-z} - \frac{\lambda}{1+\lambda^2}\, d\rho(\lambda), 
\end{align}
where $c_1$, $c_2\in\R$ are constants with $c_1\geq0$ and $\rho$ is the spectral measure of $(L,\tilde{\omega},\tilde{\dip})$. 
We will call $\rho$ the {\em spectral measure} of the spectral problem~\eqref{eqnCHspec} with a real-valued function $u$ in $H^1_{\loc}[0,\infty)$ and a positive Borel measure $\dip$ on $[0,\infty)$. 
 
 By using the relations described above, we are able to translate all results for generalized indefinite strings to the half-line spectral problem~\eqref{eqnCHspec}.
 However, for the sake of brevity, we will only state one of these results here, obtained immediately from Theorem~\ref{thm:KSforExB1}, where $\kappa$ is a fixed positive constant. 
 A few more consequences that follow readily in this case are outlined in Section~\ref{secCH:01}.
 
\begin{corollary}\label{cor:KSforCHonR+}
    A positive Borel measure $\rho$ on $\R$ with~\eqref{eqnrhoPoisson} is the spectral measure of the spectral problem~\eqref{eqnCHspec} on the half-line $[0,\infty)$ with a real-valued function $u$ in $H^1_{\loc}[0,\infty)$ and a positive Borel measure $\dip$ on $[0,\infty)$ satisfying 
    \begin{align}\label{eqnCondCHonR+1}
      \int_0^\infty (u(x) - \kappa)^2+u'(x)^2\, dx +  \int_{[0,\infty)}  d\dip < \infty
    \end{align}
     if and only if both of the following conditions hold with $\alpha = \frac{1}{4\kappa}$:
    \begin{enumerate}[label=(\roman*), ref=(\roman*), leftmargin=*, widest=ii]
\item The support of $\rho$ is discrete in $(-\infty,\alpha)$, does not contain zero and satisfies~\eqref{eqnLT-Irho}.
\item The absolutely continuous part $\rho_\ac$ of $\rho$ on $(\alpha,\infty)$ satisfies~\eqref{eqnSzego-Irho}.
    \end{enumerate}
\end{corollary} 
 
\begin{remark}
 Condition~\eqref{eqnCondCHonR+1} means that the function $u-\kappa$ belongs to the Sobolev space $H^1[0,\infty)$ and that the measure $\dip$ is finite. 
 These restrictions appear to be natural for the Camassa--Holm equation; compare the first Hamiltonian in~\eqref{eqnCHHam}. 
 In fact, by employing a generalized method of characteristics, existence of global conservative weak solutions to the Camassa--Holm equation with initial conditions such that $u-\kappa \in H^1(\R)$ was established by H.\ Holden and X.\ Raynaud in~\cite{hora07c}. 
 The constant $\kappa$ here is related to the critical wave speed and the Camassa--Holm equation is regarded to be in the dispersive regime for positive $\kappa$. 
 The case when the constant $\kappa$ is allowed to be different for $-\infty$ and for $+\infty$ has been settled in~\cite{grhora12b}. 
 We are going to show below that Theorem~\ref{thm:KSforExB1} can be used to establish the inverse spectral transform and existence of weak solutions in the special case when $\kappa$ is zero for $-\infty$ and positive for $+\infty$.
   However, analogous to the situation for the Toda lattice and the Korteweg--de Vries equation on the line in connection with Killip--Simon results~\cite{kisi03,kisi09}\footnote{Well-posedness of the Korteweg--de Vries equation in $L^2(\R)$ is a seminal result of J.\ Bourgain~\cite{bou} (see also the recent breakthrough in~\cite{kv18} by R.~Killip and M.~Visan). 
   On the other side, a complete solution of the inverse spectral problem for Schr\"odinger operators with $L^2$ potentials on the half-line is given in~\cite{kisi09} and one would hope that this enables one to integrate the Korteweg--de Vries flow on $L^2(\R)$ by means of the inverse scattering method. 
   However, this has not been achieved so far since the solution of the inverse spectral problem is given in terms of the half-line spectral measure whereas the inverse scattering transform approach in this setting requires some yet unknown full-line spectral data for potentials in $L^2(\R)$.}, at the moment it is not clear how results like Corollary~\ref{cor:KSforCHonR+} can be employed to handle the case of positive $\kappa$ at both endpoints.
   \end{remark}
 
 Unlike the corresponding result for Schr\"odinger operators in~\cite{kisi09}, Theorem~\ref{thm:KSforExB1} and Theorem~\ref{thm:KSforExB2} can, even though they are results for a half-line spectral problem, be applied to the Camassa--Holm equation and its two-component generalization on the full real line. 
 In fact, both of these theorems establish an inverse spectral transform on a particular phase space of step-like profiles that completely linearizes the conservative Camassa--Holm flow. 
   Here we will restrict our discussion again to what follows from Theorem~\ref{thm:KSforExB1}, further similar results can be deduced readily from the more detailed exposition in Sections~\ref{secCH:02} and~\ref{ss:13wellposed}.
 Let us begin by introducing an associated phase space $\CHdom^\kappa$ for the two-component Camassa--Holm system~\eqref{eqnmy2CH}, where $\kappa$ is a fixed positive constant. 

\begin{definition}\label{defDcapH1}
The set $\CHdom^\kappa$ consists of all pairs $(u,\mu)$ such that $u$ is a real-valued function in $H^1_{\loc}(\R)$ and $\mu$ is a positive Borel measure on $\R$ with
\begin{align}\label{eqnumuIntro}
   \int_B u(x)^2 + u'(x)^2\, dx \leq  \mu(B)
\end{align}
for every Borel set $B\subseteq\R$, satisfying the asymptotic growth restrictions  
\begin{align} \label{eqnDdef-Intro}  
 \int_{(-\infty,0)} \E^{-s}d\mu(s) & < \infty,  & \int_{0}^{\infty}(u(x)-\kappa)^2 + u'(x)^2dx + \int_{[0,\infty)} d\dip & < \infty,
\end{align}
where $\dip$ is the positive Borel measure on $\R$ defined such that   
\begin{align}\label{eqndipdefIntro}
 \mu(B) = \dip(B) + \int_B u(x)^2 + u'(x)^2\, dx.
\end{align}
\end{definition}

\begin{remark}
We chose to work with pairs $(u,\mu)$ and the unusual condition~\eqref{eqnumuIntro} instead of the simpler definable pairs $(u,\dip)$ for various reasons. 
For example, the measure $\mu$ is more natural when considering suitable notions of convergence; see Definition~\ref{defDTop} for details and compare with~\eqref{eqnSMPhomeo2} in Proposition~\ref{propSMPcont}.
Moreover, in the context of the conservative Camassa--Holm flow, the measure $\mu$ satisfies the transport equation in~\eqref{eqnmy2CH} and represents the energy of a solution.
\end{remark}

The first condition in~\eqref{eqnDdef-Intro} requires strong decay of both, the function $u$ and the measure $\dip$, at $-\infty$, whereas the second condition means that the function $u-\kappa$ lies in $H^1$ near $+\infty$ and that the measure $\dip$ is finite near $+\infty$.
Of course, the conditions at $-\infty$ and at $+\infty$ could also be switched due to the symmetry
\begin{align}
  (x,t) \mapsto (-x,-t)
\end{align}
of the two-component Camassa--Holm system~\eqref{eqnmy2CH}. 
However, we are not able to allow nonzero asymptotics at both endpoints with current methods.

Due to the strong decay restriction at $-\infty$, for every pair $(u,\mu)$ in $\CHdom^\kappa$ one can introduce a Weyl--Titchmarsh function $m$ for the spectral problem~\eqref{eqnCHspec} on the real line. 
 This function coincides with the Weyl--Titchmarsh function of a corresponding generalized indefinite string. 
 Namely, using the diffeomorphism $x\mapsto \E^x$ from $\R$ to $(0,\infty)$, the spectral 
problem~\eqref{eqnCHspec} is transformed to a generalized indefinite string $(L,\tilde{\omega},\tilde{\dip})$ with $L=\infty$.
One can then use Corollary~\ref{cor:KSforExB1} to characterize the corresponding set of all spectral measures $\rho$ for pairs in $\CHdom^\kappa$; see Theorem~\ref{thmCHualpha}. 
 In order to state this result, we first define $\mathcal{R}^\alpha$ for a positive constant $\alpha$ as the set of all positive Borel measures $\rho$ on $\R$ with~\eqref{eqnrhoPoisson} that satisfy conditions~\ref{itmCorKSc1} and \ref{itmCorKSc2} in Corollary~\ref{cor:KSforExB1}.  
 
 \begin{theorem}\label{eqnISTDkappa}
 The map $(u,\mu) \mapsto \rho$ from $\CHdom^\kappa$ to $\mathcal{R}^{\frac{1}{4\kappa}}$ is bijective. 
 \end{theorem}

\begin{remark}
 Just like the classical Krein--de Branges correspondence, the map $(u,\mu) \mapsto \rho$ becomes a homeomorphism with respect to reasonable topologies; see~\cite[Proposition~4.5]{Eplusminus}. 
 In fact, for our phase space $\CHdom^\kappa$ we are able to improve on this result and obtain homeomorphy with respect to finer topologies, whose definition is dictated by the trace formula in Corollary~\ref{corTFCH1} underlying the phase space $\CHdom^\kappa$ and takes into account the additional control at $+\infty$; see Definition~\ref{defD1Top}.
\end{remark}

 By means of the bijection in Theorem~\ref{eqnISTDkappa}, we are able to define the {\em conservative Camassa--Holm flow} on $\CHdom^\kappa$ by introducing a (well-defined) flow on $\mathcal{R}^\frac{1}{4\kappa}$ via 
\begin{align}\label{eqnEvolRho}
   d\rho(\lambda,t)  = \E^{-\frac{t}{2\lambda}} d\rho(\lambda,0).
\end{align}
 That this flow indeed gives rise to global weak solutions (in a sense to be made precise in Definition~\ref{def:weaksolCH}) of the two-component Camassa--Holm system~\eqref{eqnmy2CH} with initial data in $\CHdom^\kappa$ follows from our recent results in~\cite[Section~5]{Eplusminus}. 

\begin{theorem}\label{thmCHonDkappa}
Integral curves of the conservative Camassa--Holm flow on $\CHdom^\kappa$ are weak solutions of the two-component Camassa--Holm system~\eqref{eqnmy2CH}.
\end{theorem}

 Of course, by its very definition, the conservative Camassa--Holm flow on $\CHdom^\kappa$ is linearized under the {\em inverse spectral transform} $(u,\mu)\mapsto\rho$. 
 The time evolution on the spectral side can be seen to split into one part on the essential spectrum given by~\eqref{eqnEvolRho} and the usual time evolution of norming constants associated with the eigenvalues in the discrete spectrum. 
 This means that global conservative solutions with step-like initial data in $\CHdom^\kappa$ can be integrated by means of the inverse spectral transform. 
 One can expect that this will allow to deduce qualitative properties of such solutions. 

\begin{remark}\label{rem:MPs}
A pair $(u,\mu)$ is called a {\em multi-peakon profile} if it is of the form 
\begin{align}
u(x) & = \frac{1}{2}\sum_{n=1}^N \omega_n\E^{-|x-x_n|}, & \dip & = \sum_{n=1}^N\dip_n\delta_{x_n},
\end{align}
where $\delta_x$ is the unit Dirac measure centered at $x$.
The phase space $\CHdom^\kappa$ clearly does not contain any multi-peakon profiles.
However, it does include certain pairs $(u,\mu)$ that are made up of infinitely many peakons, meaning that they are of the above form with $N=\infty$. 
For example, the shifted Laguerre weight from Remark~\ref{rem:LaguerreIntro} leads to such a pair (described in more detail in Remark~\ref{remWPD1}\ref{rem:LaguerreMP}) and a complete characterization of spectral measures for the class of such pairs can be deduced from Theorem~\ref{th:HambMP}. 
The corresponding weak solutions of the two-component Camassa--Holm system~\eqref{eqnmy2CH} can be written down explicitly in terms of the moments of the time-evolved spectral measure.
Asymptotic long-time behavior of these solutions appears to be an intriguing topic, which will be addressed elsewhere. 
\end{remark}

\subsection{Further applications}\label{ss:CS+Dirac}

Since generalized indefinite strings serve as yet another canonical model of self-adjoint operators with simple spectrum~\cite{IndefiniteString}, one can apply our main results also to other important one-dimensional models and operators of mathematical physics. 
We begin with Krein strings~\cite{kakr74, kowa82} as they constitute a rather obvious and important subclass of generalized indefinite strings. 
More specifically, a {\em Krein string} is a generalized indefinite string $(L,\omega,\dip)$ such that $\omega$ is a positive Borel measure on $[0,L)$ and $\dip$ is identically zero. 
On the spectral side, this is equivalent to the condition that the Weyl--Titchmarsh function $m$ is a Stieltjes function, which means that in the integral representation~\eqref{eqnmIntRep} the constant $c_1$ vanishes, the constant $c_2$ is non-negative and the measure $\rho$ is supported on $[0,\infty)$ with
\begin{align}\label{eqnPoisson1}
\int_{[0,\infty)} \frac{d\rho(\lambda)}{1+\lambda} <\infty.
\end{align} 
Theorem~\ref{thm:KSforExB1} and Corollary~\ref{cor:KSforExB1} can thus easily be specialized to Krein strings; see Theorem~\ref{thm:KS} and Corollary~\ref{cor:KS01}. 
Furthermore, even though it is less immediately obvious, we can also apply Theorem~\ref{thm:KSforExB2} to Krein strings by using a simple relation between the Weyl--Titchmarsh functions of a Krein string $(L,\dip)$ and the generalized indefinite string $(L,0,\dip)$.  
In this way, we obtain a characterization for another class of perturbations for Krein strings that is quite different from the previous one in Theorem~\ref{thm:KS}.
This characterization will be given in Theorem~\ref{thmKSforKS02}, of which the following result is an immediate consequence, where we continue to use the notation introduced in~\eqref{eqndiprhosing}. 

 \begin{corollary}\label{corKSforKS02}
  A positive Borel measure $\rho$ on $[0,\infty)$ with~\eqref{eqnPoisson1} is the spectral measure of a Krein string $(L,\dip)$ with $L=\infty$ and  
     \begin{align}
      \int_0^\infty   \biggl(\varrho(x) - \frac{1}{1+2\sqrt{\alpha} x} \biggr)^2 x\, dx + \int_{[0,\infty)} x\, d\dip_\sing(x) & < \infty
    \end{align}
  if and only if both of the following conditions hold:
    \begin{enumerate}[label=(\roman*), ref=(\roman*), leftmargin=*, widest=ii]
      \item The support of $\rho$ is discrete in $[0,\alpha)$, does not contain zero and satisfies 
        \begin{align}
     \sum_{\lambda\in \supp(\rho)\atop 0<\lambda<\alpha}(\alpha-\lambda)^{\nicefrac{3}{2}} <\infty.
  \end{align}
\item The absolutely continuous part $\rho_\ac$ of $\rho$ on $(\alpha,\infty)$ satisfies
    \begin{align}\label{eq:SzegoKS02}
       \int_{\alpha}^{\infty}  \frac{\sqrt{\lambda-\alpha}}{\lambda^2} \log\biggl(\frac{d\rho_\ac(\lambda)}{d\lambda}\biggr) d\lambda >- \infty.
    \end{align}
    \end{enumerate}
    \end{corollary} 
    
\begin{remark}
A few remarks are in order:
\begin{enumerate}[label=(\alph*), ref=(\alph*), leftmargin=*, widest=e]
\item Analogous to our considerations in Section~\ref{ss:CHintro}, these additional results for Krein strings can be applied to the conservative Camassa--Holm flow. 
More specifically, Corollary~\ref{corKSforKS02} gives rise to another phase space and a corresponding inverse spectral transform. 
 Due to the positivity condition on the spectral side, the additional measure $\mu$ becomes superfluous and $u$ can be understood as a weak solution of the Camassa--Holm equation.
 Moreover, this positivity restriction is known to prevent blow-ups and one can expect uniqueness of weak solutions; compare~\cite{como00}. 

 \item  Let us stress that the difference between conditions~\eqref{eq:SzegoKS02} and~\eqref{eqnSzego-Irho} is only on the asymptotic behavior of the density of the spectral measure at infinity. 
 However, condition~\eqref{eq:SzegoKS02} is clearly stronger than~\eqref{eqnSzego-Irho} and ensures that the absolutely continuous part of the Krein string's weight measure is not trivial.
 In particular, it excludes Krein--Stieltjes strings (compare with Remark~\ref{rem:LaguerreIntro}).
 
\item Corollary~\ref{corKSforKS02} has an interesting connection with recent work of R.\ V.\ Bessonov and S.\ A.\ Denisov~\cite{bede20,bede21,bede22} on the spectral version of the classical Szeg\H{o}--Kolmogorov--Krein theorem, which will be discussed in Section~\ref{sec:KreinString}. 
\end{enumerate}
\end{remark}

Two more subclasses of generalized indefinite strings that are related to the Hamburger moment problem and the Stieltjes moment problem are Krein--Langer strings and Krein--Stieltjes strings. 
For these kinds of generalized indefinite strings, the coefficients $\omega$ and $\dip$ are supported on a discrete set, so that the differential equation~\eqref{eqnGISODE} reduces to a difference equation similar to the second order recurrence relations in the case of Jacobi matrices. 
They are also of particular interest because they correspond to weak solutions of the Camassa--Holm equation that are made up of infinitely many peakons.
More details and applications of our main results to these two subclasses will be given in Section~\ref{secKreinLanger}.

  For the sake of brevity, let us just briefly mention a few more applications of our results:  
 First of all, our main results obtained for generalized indefinite strings can be translated to $2\times 2$ canonical systems by using a known transformation between them (see Appendix~\ref{app:D}). 
 We will state these results for canonical systems obtained from Theorem~\ref{thm:KSforExB1} and Theorem~\ref{thm:KSforExB2} in Section~\ref{ss:KSforCS} (related results for canonical systems have been obtained before in~\cite{bede21} and in~\cite{dey} using the Arov gauge). 
 What is more important though, we can also apply our results to prove Szeg\H{o}-type theorems for one-dimensional Dirac operators (a related result can be found in~\cite{husc14}) and Schr\"odinger operators. 
  However, we decided to focus only on one straightforward application to a one-dimensional Dirac operator presented in Section~\ref{ss:Dirac}, but using a standard supersymmetry trick one may derive a corresponding Szeg\H{o}-type theorem for one-dimensional Schr\"odinger operators. 
  
We conclude this lengthy introduction with one more observation.
 The inverse spectral theory for generalized indefinite strings also sheds some light on the rich mathematical structure of the two-component Camassa--Holm system~\eqref{eqnmy2CH}.
  One may easily recognize a lot of similarities between this system and other classical completely integrable systems: The Toda lattice and multi-peakon interaction~\cite{besasz01}, the Korteweg--de Vries equation and the Camassa--Holm equation~\cite{besasz98,le04,mc03b}, the AKNS system and the two-component Camassa--Holm system~\cite{chlizh06}.  
  All these connections are not surprising at all if one realizes that the corresponding isospectral problems (Jacobi matrices, one-dimensional Schr\"odinger and Dirac operators) can all be transformed into the form of a generalized indefinite string.

\subsection{Outline of this article}\label{ss:overview}

The preliminary Section~\ref{sec:prelim} first collects all necessary information about generalized indefinite strings.
In Section~\ref{sec:CompactPert}, we use some results from our recent work~\cite{DSpec} in order to demonstrate how the conditions~\eqref{eqnCondS1} and~\eqref{eqnCondS2} can be understood as additive Hilbert--Schmidt perturbations. 
 Section~\ref{secmatzero} then deals with spectral asymptotics of Weyl--Titchmarsh functions at zero. 
 On the one hand, one may use the results from~\cite{AsymCS} to get the leading term of the asymptotics (see Proposition~\ref{lem:GISasymp}).
 However, since the Weyl--Titchmarsh function has an analytic extension to zero in our case, we are indeed able to obtain a complete Taylor series expansion at zero (see Proposition~\ref{prop:expandm}).
 This turns out to be an important prerequisite for obtaining trace formulas. 

Sections~\ref{secSbSsumrule}, \ref{secLowSemiContA} and~\ref{secKSB1proof} are dedicated to the proof of Theorem~\ref{thm:KSforExB1}. 
The overall strategy is similar to the one in~\cite{kisi03,kisi09}. 
 More specifically, Section~\ref{secSbSsumrule} contains the key result; relative trace formulas in Theorem~\ref{thm:traceflaA}. 
 Crucial for the proof of these identities are factorizations of certain meromorphic functions on the open upper complex half-plane $\C_+$ and associated trace formulas that will be derived in Appendix~\ref{appMeroMain}. 
 Next, in Section~\ref{secLowSemiContA}, we establish lower semi-continuity of certain functionals involving an associated {\em transmission coefficient}, a quantity resembling the classical transmission coefficient in one-dimensional scattering, which is also related to a perturbation determinant. 
 Having all these ingredients at hand, we finally complete the proof of our first main result in Section~\ref{secKSB1proof}.  
 As in~\cite{kisi03,kisi09}, this is achieved by showing that the trace formula underlying the class of generalized indefinite strings in Theorem~\ref{thm:KSforExB1} holds in the strong sense that if one side is finite, then so is the other one.
 Even more, we will additionally also obtain another trace formula in Corollary~\ref{corTFalpham} that only involves the half-line spectral measure and leads to a sharp Lieb--Thirring inequality in Corollary~\ref{cor:LT01new}. 
In a similar way, the proof of Theorem~\ref{thm:KSforExB2} is contained in Sections~\ref{secSbSsumrule2}, \ref{secLowSemiContB} and~\ref{secKSB2proof} as well as Appendix~\ref{appMeroMainII}.

The remaining sections are dedicated to applications of our main results. 
To begin with, we describe applications to the special case of Krein strings in Section~\ref{sec:KreinString}. 
We then continue to apply our results to Krein--Stieltjes strings and Krein--Langer strings in Section~\ref{secKreinLanger}.
 Applications to the conservative Camassa--Holm flow will be explored in Section~\ref{secCH}.
In the final Section~\ref{secCO}, we present two more applications to $2\times 2$ canonical systems and one-dimensional Dirac operators.
 Relevant necessary notions and facts about canonical systems and their connection with generalized indefinite strings are summarized in Appendix~\ref{app:D}.
For convenience, we also gather some auxiliary information on two particular functions appearing frequently in our trace formulas in Appendix~\ref{app:F1F2}.

\section{Preliminaries}\label{sec:prelim} 

We are first going to introduce several spaces of functions and distributions.  
 For every fixed $L\in(0,\infty]$, we denote with $H^1_{\loc}[0,L)$, $H^1[0,L)$ and $H^1_{\cc}[0,L)$ the usual Sobolev spaces 
\begin{align}
H^1_{\loc}[0,L) & =  \lbrace f\in AC_{\loc}[0,L) \,|\, f'\in L^2_{\loc}[0,L) \rbrace, \\
 H^1[0,L) & = \lbrace f\in H^1_{\loc}[0,L) \,|\, f,\, f'\in L^2[0,L) \rbrace, \\ 
 H^1_{\cc}[0,L) & = \lbrace f\in H^1[0,L) \,|\, \supp(f) \text{ compact in } [0,L) \rbrace.
\end{align}
The space of distributions $H^{-1}_{\loc}[0,L)$ is the topological dual of $H^1_{\cc}[0,L)$, so that the mapping $\Qr\mapsto\chi$, defined by
 \begin{align}
    \chi(h) = - \int_0^L \Qr(x)h'(x)dx
 \end{align} 
 for all $h\in H^1_{\cc}[0,L)$, establishes a one-to-one correspondence between $L^2_{\loc}[0,L)$ and $H^{-1}_{\loc}[0,L)$. 
The unique function $\Qr\in L^2_{\loc}[0,L)$ corresponding to some distribution $\chi\in H^{-1}_{\loc}[0,L)$ in this way will be referred to as the {\em normalized anti-derivative} of $\chi$.
 Furthermore, a distribution in $H^{-1}_{\loc}[0,L)$ is said to be {\em real} if its normalized anti-derivative is real-valued almost everywhere on $[0,L)$.  

A particular kind of distribution in $H^{-1}_{\loc}[0,L)$ arises from Borel measures on the interval $[0,L)$.
 In fact, if $\chi$ is a complex-valued Borel measure on $[0,L)$, then we will identify it with the distribution in $H^{-1}_{\loc}[0,L)$ given by  
 \begin{align}
  h \mapsto \int_{[0,L)} h\,d\chi. 
 \end{align}
The normalized anti-derivative $\Qr$ of such a measure $\chi$ is simply given by the left-continuous distribution function 
 \begin{align}
 \Qr(x)=\int_{[0,x)}d\chi
 \end{align}
 for almost all $x\in [0,L)$, as an integration by parts (use, for example, \cite[Exercise~5.8.112]{bo07} or~\cite[Theorem~21.67]{hest65}) shows.   

 In order to obtain a self-adjoint realization of our spectral problem in a suitable Hilbert space later, we also introduce the function space  
\begin{align}
\Hast & = \begin{cases} \lbrace f\in H^1_{\loc}[0,L) \,|\, f'\in L^2[0,L),~ \lim_{x\rightarrow L} f(x) = 0 \rbrace, & L<\infty, \\ \lbrace f\in H^1_{\loc}[0,L) \,|\, f'\in L^2[0,L) \rbrace, & L=\infty, \end{cases} 
\end{align}
 as well as the linear subspace  
\begin{align}
 \Hasto & = \lbrace f\in \Hast \,|\, f(0) = 0 \rbrace, 
\end{align}
 which turns into a Hilbert space when endowed with the scalar product 
 \begin{align}\label{eq:normti}
 \spr{f}{g}_{\Hasto} = \int_0^L f'(x) g'(x)^\ast dx.
 \end{align}
The space $\Hasto$ can be viewed as a completion with respect to the norm induced by~\eqref{eq:normti} of the space of all smooth functions which have compact support in $(0,L)$. 
In particular, the space $\Hasto$ coincides algebraically and topologically with the usual Sobolev space $H^1_0[0,L)$ when $L$ is finite. 

\begin{definition}
A {\em generalized indefinite string} is a triple $(L,\omega,\dip)$ such that $L\in(0,\infty]$, $\omega$ is a real distribution in $H^{-1}_{\loc}[0,L)$ and $\dip$ is a positive Borel measure on the interval $[0,L)$.  
\end{definition}

 For a generalized indefinite string $(L,\omega,\dip)$, the normalized anti-derivative of the distribution $\omega$ will always be denoted with $\Wr$ in the following. 
 Associated with such a generalized indefinite string is the inhomogeneous differential equation
 \begin{align}\label{eqnDEinho}
  -f''  = z\, \omega f + z^2 \dip f + \chi, 
 \end{align}
 where $\chi$ is a distribution in $H^{-1}_{\loc}[0,L)$ and $z$ is a complex spectral parameter. 
 Of course, this differential equation has to be understood in a weak sense:   
  A solution of~\eqref{eqnDEinho} is a function $f\in H^1_{\loc}[0,L)$ such that 
 \begin{align}\label{eqnIntDEho}
 f'\NL h(0) + \int_{0}^L f'(x) h'(x) dx = z\, \omega(fh) + z^2 \int_{[0,L)}fh\,d\dip +  \chi(h)
 \end{align}
 for all functions $h\in H^1_{\cc}[0,L) $ and a (unique) constant $f'\NL \in\C$. 
 
 With this notion of solution, we are able to introduce the fundamental system of solutions $\theta(z,\redot)$ and $\phi(z,\redot)$ of the homogeneous differential equation
   \begin{align}\label{eqnDEho}
  -f'' = z\, \omega f + z^2 \dip f
 \end{align}
  satisfying the initial conditions
 \begin{align}
  \theta(z,0)& = \phi'\NLz =1, &  \theta'\NLz & = \phi(z,0) =0,
 \end{align}
 for every $z\in\C$; see \cite[Lemma~3.2]{IndefiniteString}. 
 Even though the derivatives of these functions are only locally square integrable in general, there are unique left-continuous functions $\theta^\qd(z,\redot)$ and $\phi^\qd(z,\redot)$ on $[0,L)$ such that 
 \begin{align}
   \theta^\qd(z,x) & = \theta'(z,x) + z\Wr(x)\theta(z,x), & \phi^\qd(z,x) & = \phi'(z,x) + z\Wr(x)\phi(z,x),
 \end{align}
 for almost all $x\in[0,L)$;  see \cite[Equation~(4.12)]{IndefiniteString}.
 These functions will henceforth be referred to as {\em quasi-derivatives} of the solutions $\theta(z,\redot)$ and $\phi(z,\redot)$. 
 As functions of the spectral parameter $z$, the solutions as well as their quasi-derivatives are entire; see~\cite[Corollary~3.5]{IndefiniteString} for example. 
 We defer justification of the remaining claims in the following theorem to Appendix~\ref{app:D}, where we summarize the connection between generalized indefinite strings and canonical systems. 
 
 \begin{theorem}\label{thm:FSScartwright}
For every fixed $x\in[0,L)$, the functions
    \begin{align}\label{eqnFSSstring}
   z & \mapsto \theta(z,x), & z & \mapsto \theta^\qd(z,x)/z, & z & \mapsto \phi(z,x), & z & \mapsto \phi^\qd(z,x), 
 \end{align}
 are real entire and have only real and simple zeros (unless they are identically zero\footnote{This can happen only if $x$ is zero or $\omega$ and $\dip$ both vanish on $[0,x)$; see Proposition~\ref{propFS} below.}).
Moreover, they belong to the Cartwright class with exponential type given by 
\begin{align}\label{eqnthetaphiExpTyp}
  \int_0^x \varrho(s) ds,
\end{align}
 where $\varrho$ is the square root of the Radon--Nikod\'ym derivative of $\dip$ with respect to the Lebesgue measure.
 \end{theorem}
 
  At the origin, when $z$ is zero, the differential equation~\eqref{eqnDEho} can of course be solved explicitly and our fundamental system is given by
 \begin{align}\label{eqncsatzero}
   \theta(0,x) & = 1, & \theta^\qd(0,x) & = 0, & \phi(0,x) & = x, & \phi^\qd(0,x) & = 1.
 \end{align}
 More crucially, we will furthermore require the following formulas for the derivatives of this fundamental system with respect to the spectral parameter at the origin, which have been derived in~\cite[Proposition~4.1]{ACSpec}. 
 Let us note that differentiation with respect to the spectral parameter will be denoted with a dot here and is always meant to be done after taking quasi-derivatives.
 
 \begin{proposition}\label{propFS}
  For every $x\in[0,L)$, one has   
  \begin{align}
    \label{eqndtheta0} \dot{\theta}(0,x) & =  - \int_0^x \Wr(s)ds, &   \dot{\theta}^\qd(0,x) & =  0, \\
    \label{eqndphi0} \dot{\phi}(0,x) & = \int_0^x \int_0^s \Wr(t)dt\, ds - \int_0^x \Wr(s) s\, ds, &   \dot{\phi}^\qd(0,x) & = \int_0^x \Wr(s)ds,
  \end{align}
  as well as  
  \begin{align}
    \label{eqnddtheta0}  \ddot{\theta}(0,x) & = \biggl(\int_0^x \Wr(s)ds\biggr)^2 - 2 \int_0^x \int_0^s \Wr(t)^2 dt\,ds - 2 \int_0^x \int_{[0,s)} d\dip\,ds,                                                   \\
    \label{eqnddthetap0} \ddot{\theta}^\qd(0,x) & = -2 \int_0^x \Wr(s)^2 ds -2 \int_{[0,x)} d\dip,                                                    \\
    \label{eqnddphip0} \ddot{\phi}^\qd(0,x) & = \biggl(\int_0^x \Wr(s)ds\biggr)^2 - 2 \int_0^x \Wr(s)^2 s\, ds -2 \int_{[0,x)} s\, d\dip(s).                                                 
  \end{align}
 \end{proposition}

 The differential equation~\eqref{eqnDEinho} for a generalized indefinite string $(L,\omega,\dip)$ gives rise to an associated self-adjoint linear relation in the Hilbert space
 \begin{align}
 \cH = \Hasto\times L^2([0,L);\dip),
\end{align}
which is endowed with the scalar product
\begin{align}
 \spr{f}{g}_{\cH} = \int_0^L f_1'(x) g_1'(x)^\ast dx + \int_{[0,L)} f_2(x) g_2(x)^\ast d\dip(x).
\end{align} 
 Here, the respective components of some vector $f\in\cH$ are always denoted by adding subscripts, that is, with $f_1$ and $f_2$.  
Now the linear relation $\T$ in the Hilbert space $\cH$ is defined by saying that some pair $(f,g)\in\cH\times\cH$ belongs to $\T$ if and only if  
\begin{align}\label{eqnDEre1}
-f_1'' & =\omega g_{1} + \dip g_{2}, &  \dip f_2 & =\dip g_{1}.
\end{align}
In order to be precise, the right-hand side of the first equation in~\eqref{eqnDEre1} has to be understood as the $H^{-1}_{\loc}[0,L)$ distribution given by 
\begin{align}
 h \mapsto \omega(g_1h) + \int_{[0,L)} g_2 h\, d\dip,
\end{align} 
so that the first equation can be understood as a special case of~\eqref{eqnDEinho}. 
 Moreover, the second equation in~\eqref{eqnDEre1} holds if and only if $f_2$ is equal to $g_1$ almost everywhere on $[0,L)$ with respect to the measure $\dip$. 
  The linear relation $\T$ defined in this way turns out to be self-adjoint in the Hilbert space $\cH$; see \cite[Theorem~4.1]{IndefiniteString}.
  It is indeed closely related to the differential equation~\eqref{eqnDEinho}: 
  A pair $(f,g)\in\cH\times\cH$ belongs to $\T-z$ if and only if 
\begin{align}\label{eqnTloc-z}
 -f_1'' & = z\, \omega f_1 +z^2 \dip f_1 + \omega g_1 + z\,\dip g_1 + \dip g_2, & \dip f_2 = z\,\dip f_1 +  \dip g_1. 
\end{align}
In particular, this shows that some $f\in\cH$ belongs to $\ker{\T-z}$ if and only if $f_1$ is a solution of the differential equation~\eqref{eqnDEho} and $\dip f_2 = z\, \dip f_1$. 

For the sake of simplicity, we shall always mean the spectrum of the corresponding linear relation when we speak of the spectrum of a generalized indefinite string in the following.
We will use the same convention for the various spectral types.   
 
   A central object in the spectral theory for the linear relation $\T$ is the associated {\em Weyl--Titchmarsh function} $m$. 
   This function can be defined on $\C\backslash\R$ by 
 \begin{align}\label{eqnmdef}
  m(z) =  \frac{\psi'\NLz}{z\psi(z,0)},
 \end{align} 
 where $\psi(z,\redot)$ is the unique (up to constant multiples) non-trivial solution of the differential equation~\eqref{eqnDEho} that lies in $\Hast$ and $L^2([0,L);\dip)$, guaranteed to exist by~\cite[Lemma~4.2]{IndefiniteString}. 
 We have seen in~\cite[Lemma~5.1]{IndefiniteString} that the Weyl--Titchmarsh function $m$ is a Herglotz--Nevanlinna function.
 As such, it admits an integral representation of the form (see~\cite[Section~5]{IndefiniteString} for a justification of the third term on the right-hand side)
\begin{align}\label{eqnWTmIntRep}
 m(z) = c_1 z + c_2 - \frac{1}{Lz} +  \int_\R \frac{1}{\lambda-z} - \frac{\lambda}{1+\lambda^2}\, d\rho(\lambda)  
\end{align}
for some non-negative constant $c_1$, some real constant $c_2$ and a positive Borel measure $\rho$ on $\R$ with $\rho(\lbrace0\rbrace)=0$ and   
\begin{align}\label{eqnWTrhoPoisson}
 \int_\R \frac{d\rho(\lambda)}{1+\lambda^2} < \infty. 
\end{align}
Here, we employ the convention that whenever an $L$ appears in a denominator, the corresponding fraction has to be interpreted as zero when $L$ is not finite.

The measure $\rho$ turns out to be a {\em spectral measure} for the linear relation $\T$ in the sense that the operator part of $\T$ is unitarily equivalent to multiplication with the independent variable in $L^2(\R;\rho)$; see \cite[Theorem~5.8]{IndefiniteString}.
 Of course, this establishes an immediate connection between the spectral properties of the linear relation $\T$ and the measure $\rho$. 
 For example, the spectrum of $\T$ coincides with the topological support of $\rho$ and thus can be read off the singularities of $m$ (more precisely, the function $m$ admits an analytic continuation away from the spectrum of $\T$). 

 The main basic result of inverse spectral theory for generalized indefinite strings says that every Herglotz--Nevanlinna function is the Weyl--Titchmarsh function of a unique generalized indefinite string; see~\cite[Theorem~6.1]{IndefiniteString}.
  
\begin{theorem}\label{thm:IST}  
The map $(L,\omega,\dip) \mapsto m$ is a bijection between the set of all generalized indefinite strings and the set of all Herglotz--Nevanlinna functions.
\end{theorem}

\begin{remark}\label{rem:uniquerho}
According to the definition of the spectral measure $\rho$ above via the integral representation~\eqref{eqnWTmIntRep}, it recovers the Weyl--Titchmarsh function $m$ only up to the three parameters $c_1$, $c_2$ and $L$.
In view of Theorem~\ref{thm:IST}, this means that the collection of all generalized indefinite strings having the same spectral measure $\rho$ is a three parameters family. 
The parameters $c_1$ and $c_2$ simply correspond to changes of point masses of $\omega$ and $\dip$ at zero, which are invisible to the linear relation $\T$ and thus also to the spectral measure $\rho$.
More precisely, if $(L,\omega_1,\dip_1)$ and $(L,\omega_2,\dip_2)$ are two generalized indefinite strings such that 
\begin{align}
  \Wr_1(x) & = b + \Wr_2(x), & \int_{[0,x)} d\dip_1 & = a + \int_{[0,x)} d\dip_2,
 \end{align}
 for almost all $x\in [0,L)$ and some $a$, $b\in\R$, then the corresponding Weyl--Titchmarsh functions $m_1$ and $m_2$ are related by 
  \begin{align}
    m_1(z) = az + b + m_2(z), 
  \end{align}
 while the corresponding spectral measures are the same. 
 The parameter $L$ corresponds of course to a change of length, which can be achieved via an elementary transformation.
 In fact, if $(L_1,\omega_1,\dip_1)$ and $(L_2,\omega_2,\dip_2)$ are two generalized indefinite strings such that  
  \begin{align}
   \begin{split}
    \Wr_1(x) & = \biggl(1-\frac{x}{L_1}+\frac{x}{L_2}\biggr)^{-2} \Wr_2(\eta(x)) \\
     & \qquad\qquad + 2\biggl(\frac{1}{L_2}-\frac{1}{L_1}\biggr) \int_0^{\eta(x)} \biggl(1+\frac{s}{L_1}-\frac{s}{L_2}\biggr) \Wr_2(s)ds, 
    \end{split} \\
    \int_{[0,x)}d\dip_1 & = \int_{[0,\eta(x))} \biggl(1+\frac{s}{L_1}- \frac{s}{L_2}\biggr)^2 d\dip_2(s), 
  \end{align}
  for almost all $x\in [0,L_1)$, where the bijection $\eta:[0,L_1)\rightarrow[0,L_2)$ is given by 
  \begin{align}
   \eta(x) & = \frac{x}{1-\frac{x}{L_1}+\frac{x}{L_2}},
  \end{align}
  then the corresponding Weyl--Titchmarsh functions $m_1$ and $m_2$ are related by 
  \begin{align}
    m_1(z) + \frac{1}{L_1z} = m_2(z) + \frac{1}{L_2z},
  \end{align}
  while the corresponding spectral measures are the same. 
  In conclusion, we see that the spectral measure $\rho$ determines a generalized indefinite string $(L,\omega,\dip)$ only up to transformations of the above types. 
\end{remark}

We will also rely on a continuity property of the correspondence $(L,\omega,\dip)\mapsto m$. 
In order to state it, let $(L_n,\omega_n,\dip_n)$ be a sequence of generalized indefinite strings, let $m_n$ be the corresponding Weyl--Titchmarsh functions and let $\Wr_n$ be the normalized anti-derivatives of $\omega_n$.

 \begin{proposition}\label{propSMPcont}
 The Weyl--Titchmarsh functions $m_n$ converge locally uniformly to $m$ if and only if  
 \begin{align}\label{eqnLiminfL}
    \sup\biggl\lbrace x\in\lbrace 0\rbrace \cup [0,\liminf_{k\rightarrow\infty} L_{n_k})\,\biggl|\, \limsup_{k\rightarrow\infty} \int_0^x \Wr_{n_k}(s)^2 ds + \int_{[0,x)} d\dip_{n_k} < \infty \biggr\rbrace = L
  \end{align}
 holds for each subsequence $n_k$ and\footnote{Note that all quantities are well-defined for large enough $n\in\N$ in view of~\eqref{eqnLiminfL}.}  
 \begin{align}\label{eqnSMPhomeo}
   \lim_{n\rightarrow\infty} \int_0^{x} \Wr_n(s)ds  & = \int_0^x \Wr(s)ds, \\
\label{eqnSMPhomeo2}
   \lim_{n\rightarrow\infty} \int_0^x \Wr_n(s)^2 ds + \int_{[0,x)} d\dip_n & = \int_0^x \Wr(s)^2 ds + \int_{[0,x)} d\dip,
 \end{align}
 for almost every $x\in [0,L)$.  
\end{proposition}

\begin{proof}
The claim is a slight variation of~\cite[Proposition~6.2]{IndefiniteString} and it is enough to justify why~\eqref{eqnSMPhomeo2} is equivalent to pointwise convergence of 
 \begin{align*} 
   \lim_{n\rightarrow\infty} \int_0^x \biggl(\int_0^s \Wr_{n}(t)^2 dt + \int_{[0,s)} d\dip_{n}\biggr)ds   = \int_0^x \biggl(\int_0^s \Wr(t)^2 dt + \int_{[0,s)} d\dip\biggr) ds
 \end{align*}
 for all $x\in[0,L)$. 
Due to positivity of the integrands, we only need to show that the latter implies~\eqref{eqnSMPhomeo2} for almost every $x\in [0,L)$. 
 However, this immediately follows from~\cite[Theorem~B]{kas12} for example.    
\end{proof}

 We conclude this preliminary section with two examples of {\em unperturbed} generalized indefinite strings associated with Theorem~\ref{thm:KSforExB1} and Theorem~\ref{thm:KSforExB2}. 
    
  \begin{mainexa}\label{exaalpha}
   Let $\alpha>0$ and consider the generalized indefinite string $(L,\omega,\dip)$ such that $L$ is infinite, the distribution $\omega$ is given via its normalized anti-derivative $\Wr$ by   
    \begin{align}\label{eq:Walpha}
      \Wr(x) = \frac{x}{1+2\sqrt{\alpha}x} 
    \end{align} 
    and the measure $\dip$ vanishes identically. 
   Under these assumptions, the corresponding differential equation~\eqref{eqnDEho} simply reduces to 
    \begin{align}
    - f''(x) = \frac{z}{(1+2\sqrt{\alpha}x)^2}  f(x).
    \end{align}
      For every $k$ in the open upper complex half-plane $\C_+$, the function $\psi(k,\redot)$ given by 
    \begin{align}
      \psi(k,x) = (1+2\sqrt{\alpha}x)^{\frac{\I k}{2\sqrt{\alpha}} + \frac{1}{2}}
    \end{align}   
    is a solution of this differential equation with $z=k^2+\alpha$ that lies in $\dot{H}^1[0,\infty)$.
    Consequently, the corresponding Weyl--Titchmarsh function $m$ is given by  
    \begin{align}\label{eq:Malpha}
      m(k^2+\alpha) = \frac{\psi'(k,0-)}{(k^2+\alpha)\psi(k,0)}= \frac{1}{\sqrt{\alpha}-\I k}
    \end{align}
  and the corresponding spectral measure $\rho$ is given by 
    \begin{align}\label{eq:Rhoalpha}
    \rho(B) = \frac{1}{\pi} \int_{B\cap [\alpha,\infty)}\frac{\sqrt{\lambda-\alpha}}{\lambda}d\lambda.
    \end{align}
    In particular, the latter means that the corresponding spectrum coincides with the interval $[\alpha,\infty)$ and is purely absolutely continuous. 
    The Weyl--Titchmarsh function $m$ and the spectral measure $\rho$ clearly satisfy the conditions in Theorem~\ref{thm:KSforExB1} and Corollary~\ref{cor:KSforExB1} in this case. 
  \end{mainexa}
  
      \begin{remark}
    The Weyl--Titchmarsh function and the spectral measure of the generalized indefinite string in Example~\ref{exaalpha} also serve as a Weyl-Titchmarsh function and a spectral measure for the one-dimensional half-line Schr\"odinger operator with constant potential $\alpha$, subject to mixed boundary conditions at the finite endpoint.
    \end{remark}

     \begin{mainexa}\label{exa2alpha}
         Let $\beta>0$ and consider the generalized indefinite string $(L,\omega,\dip)$ such that $L$ is infinite, the distribution $\omega$ is identically zero and the measure $\dip$ is given by 
    \begin{align}\label{eq:DipBeta}
       \dip(B) = \int_B \frac{1}{(1+2 \beta x)^2} dx.
    \end{align} 
   Under these assumptions, the corresponding differential equation~\eqref{eqnDEho} simply reduces to 
    \begin{align}
    - f''(x) = \frac{z^2}{(1+2 \beta x)^{2}}  f(x).
    \end{align}
 For every $k$ in the open upper complex half-plane $\C_+$, the function $\psi(k,\redot)$ given by 
    \begin{align}
      \psi(k,x) = (1+2\beta x)^{\frac{\I k}{1-k^2} + \frac{1}{2}} 
    \end{align}   
    is a solution of this differential equation with 
    \begin{align}
      z = \zeta(k) := \beta \frac{1+k^2}{1-k^2}
    \end{align}
    that lies in both $\dot{H}^1[0,\infty)$ and $L^2([0,\infty);\dip)$.
    Consequently, the corresponding Weyl--Titchmarsh function $m$ is given by  
    \begin{align}\label{eq:Mbeta}
      m\circ \zeta(k) = \frac{\psi'(k,0-)}{\zeta(k)\psi(k,0)} = \frac{1+\I k}{1-\I k} 
    \end{align}
  and the corresponding spectral measure $\rho$ is given by 
    \begin{align}\label{eq:Rhobeta}
    \rho(B) = \frac{1}{\pi} \int_{B\backslash (-\beta,\beta)} \frac{\sqrt{\lambda^2-\beta^2}}{|\lambda|}d\lambda.
    \end{align}
   In particular, the latter means that the corresponding spectrum coincides with the set $(-\infty,-\beta]\cup [\beta,\infty)$ and is purely absolutely continuous. 
    The Weyl--Titchmarsh function $m$ and the spectral measure $\rho$ clearly satisfy the conditions in Theorem~\ref{thm:KSforExB2} and Corollary~\ref{cor:KSforExB2} in this case. 
         \end{mainexa}

    \begin{remark}
    The Weyl--Titchmarsh function and the spectral measure of the generalized indefinite string in Example~\ref{exa2alpha} also serve as a Weyl-Titchmarsh function and a spectral measure for a one-dimensional half-line Dirac operator; see Example~\ref{xmpl:Diracbeta}.  
    \end{remark}

\section{Compact perturbations}\label{sec:CompactPert} 

 The aim of this section is to demonstrate how certain generalized indefinite strings can be understood as compact or Hilbert--Schmidt perturbations of the generalized indefinite strings in Example~\ref{exaalpha} or Example~\ref{exa2alpha}.
 In particular, this will establish necessity of condition~\ref{itmKSc1} in Theorem~\ref{thm:KSforExB1} and condition~\ref{itmKS2c1} in Theorem~\ref{thm:KSforExB2}. 
 
 \begin{theorem}\label{thm:HSforAB}
   Let $S$ be a generalized indefinite string $(L,\omega,\dip)$ with $L=\infty$ and let $\alpha$, $\beta>0$. 
    \begin{enumerate}[label=(\roman*), ref=(\roman*), leftmargin=*, widest=ii]
\item\label{itmHSforA}
If there is a constant $c\in\R$ such that 
\begin{align}\label{eq:KIOdipComp}
\lim_{x\rightarrow \infty} x\int_x^\infty \biggl(\Wr(s) - c - \frac{s}{1+2\sqrt{\alpha} s} \biggr)^2ds  + x \int_{[x,\infty)}d\dip =0,
\end{align}
then zero does not belong to the spectrum of $S$ and the essential spectrum of $S$ coincides with the interval $[\alpha,\infty)$.
\item\label{itmHSforB}
If there is a constant $c\in\R$ such that 
\begin{align}\label{eq:KIOwComp}
  \lim_{x\rightarrow\infty}  x\int_x^\infty (\Wr(s) - c)^2ds + x\int_x^\infty \biggl(\varrho(s) - \frac{1}{1+2\beta s}\biggr)^2 ds + x\int_{[x,\infty)} d\dip_\sing = 0,
\end{align}
where $\varrho$ is the square root of the Radon--Nikod\'ym derivative of $\dip$ and $\dip_\sing$ is the singular part of $\dip$  (both with respect to the Lebesgue measure), then zero does not belong to the spectrum of $S$ and the essential spectrum of $S$ coincides with the set $(-\infty,-\beta]\cup [\beta,\infty)$.
\end{enumerate}
\end{theorem}

Our proof of the above result will rely on Weyl's theorem, which guarantees stability of the essential spectrum under compact perturbations. 
However, it is not immediately obvious how to understand the corresponding linear relations here as additive perturbations. 
 To this end, we first need to recall a few facts, details of which can be found in~\cite[Section~3]{DSpec}: 
In the Hilbert space $\dot{H}^1_0[0,\infty)$, let us introduce the closed linear operator $\KIO_\chi$ with symbol $\chi\in H^{-1}_{\loc}[0,\infty)$ via its bilinear form   
    \begin{align}\label{eq:KIOviaForm}
      \spr{\KIO_\chi f}{g}_{\dot{H}^1_0[0,\infty)} = \chi(fg^\ast),
    \end{align}
    which is initially only defined for functions with compact support (for a thorough discussion we refer the reader to~\cite[Section~3]{DSpec}). 
  We note that this operator $\KIO_\chi$ is self-adjoint if and only if the distribution $\chi$ is real.
If $\chi$ is even a positive Borel measure on $[0,\infty)$, then the bilinear form in~\eqref{eq:KIOviaForm} implies the factorization 
    \begin{align}\label{eq:KIOviaFactor}
      \KIO_\chi = \rI_\chi^\ast \rI_\chi
    \end{align}
 of the operator $\KIO_\chi$, where $\rI_\chi\colon\dot{H}^1_0[0,\infty)\rightarrow L^2([0,\infty);\chi)$ is the inclusion map.   
In the following proposition, we are going to collect several useful criteria from~\cite[Theorem~4.6, Proposition~5.1 and Theorem~5.2]{DSpec} for the inverse of $\T$ to be a bounded/compact/Hilbert--Schmidt operator.  

\begin{proposition}\label{corPelin}
   Let $S$ be a generalized indefinite string $(L,\omega,\dip)$ with $L=\infty$ and let $\T$ be the corresponding linear relation. 
   Then the following assertions hold:
    \begin{enumerate}[label=(\roman*), ref=(\roman*), leftmargin=*, widest=iii]
\item\label{itmPelini}
Zero belongs to the resolvent set of $\T$ if and only if the operators $\KIO_\omega$ and $\KIO_\upsilon$ are bounded, which is further equivalent to validity of 
\begin{align}\label{eq:m-anal0Wr}
\limsup_{x\rightarrow \infty} \,x \int_x^\infty (\Wr(s) - c)^2 ds + x \int_{[x,\infty)}d\dip <\infty
\end{align}
for some constant $c\in\R$. 
 In this case, the inverse of $\T$ is given by
  \begin{align}\label{eqnTinvBOM}
   \T^{-1} = \begin{pmatrix} \KIO_\omega & \rI_\dip^\ast  \\ \rI_\dip & 0 \end{pmatrix} 
  \end{align}
  and the constant $c$ is given by 
   \begin{align}\label{eq:def-c_w}
c=\lim_{x\rightarrow \infty}\, \frac{1}{x} \int_0^x \Wr(s) ds.
\end{align}
  \item\label{itmPelinii}
  The spectrum of $\T$ is purely discrete if and only if the operators $\KIO_\omega$ and $\KIO_\upsilon$ are compact, which is further equivalent to validity of
  \begin{align}\label{eq:Tcomp}
\lim_{x\rightarrow \infty} \,x \int_x^\infty (\Wr(s) - c)^2 ds + x \int_{[x,\infty)}d\dip = 0
\end{align}
for some constant $c\in\R$. 
\item\label{itmPeliniii}
The inverse of $\T$ belongs to the Hilbert--Schmidt ideal if and only if so do the operators $\KIO_\omega$ and $\rI_\dip$,  which is further equivalent to validity of
  \begin{align}\label{eq:TgS2}
  \begin{split}
  \int_0^\infty (\Wr(x) - c)^2 x\, dx + \int_{[0,\infty)}x\,d\dip(x) <\infty
\end{split}
\end{align}
for some constant $c\in\R$. 
 \end{enumerate} 
\end{proposition}

\begin{proof}[Proof of Theorem~\ref{thm:HSforAB}~\ref{itmHSforA}]
Since
\begin{align*}
\limsup_{x\to\infty}x\int_x^\infty \frac{1}{(1+2\sqrt{\alpha} s)^2}ds = \limsup_{x\to\infty} \frac{1}{2\sqrt{\alpha}}\frac{x}{1+2\sqrt{\alpha}x} =\frac{1}{4\alpha},
\end{align*}
it is immediate to see that~\eqref{eq:KIOdipComp} implies condition~\eqref{eq:m-anal0Wr} and hence, by Proposition~\ref{corPelin}~\ref{itmPelini}, zero belongs to the resolvent set of $\T$ and the inverse of $\T$ has the form~\eqref{eqnTinvBOM}. 
Moreover, in this case (see~\cite[Corollary~4.7]{DSpec}), the non-zero spectrum of $\T^{-1}$ coincides with the non-zero spectrum of the block operator matrix
 \begin{align*}
     \begin{pmatrix} \KIO_\omega & \sqrt{\KIO_\dip}  \\ \sqrt{\KIO_\dip} & 0 \end{pmatrix}
 \end{align*}
 acting in $\dot{H}^1_0[0,\infty)\times \dot{H}^1_0[0,\infty)$ and all non-zero eigenvalues of this block operator matrix are simple.
 In particular, this also holds for the generalized indefinite string from Example~\ref{exaalpha}, so that the spectrum of the corresponding block operator matrix 
 \begin{align*}
   \begin{pmatrix} \KIO_{\omega_{\alpha}} & 0  \\ 0 & 0 \end{pmatrix}
\end{align*}
 coincides with the interval $[0,\alpha^{-1}]$, where $\omega_a$ is the distribution in $H^{-1}_{\loc}[0,\infty)$ whose normalized anti-derivative is given by~\eqref{eq:Walpha}.
 Now the remaining claim follows because by Proposition~\ref{corPelin}~\ref{itmPelinii} and~\eqref{eq:KIOdipComp}, the difference 
 \begin{align}\label{eqnPertA}
    \begin{pmatrix} \KIO_\omega & \sqrt{\KIO_\dip}  \\ \sqrt{\KIO_\dip} & 0 \end{pmatrix} -   \begin{pmatrix} \KIO_{\omega_{\alpha}} & 0  \\ 0 & 0 \end{pmatrix} = \begin{pmatrix} \KIO_{\omega-\omega_\alpha} & \sqrt{\KIO_\dip}  \\ \sqrt{\KIO_\dip} & 0 \end{pmatrix}
\end{align}
 is compact and we are left to apply Weyl's theorem. 
\end{proof}

\begin{remark}
  Comparing condition~\eqref{eqnCondS1} in Theorem~\ref{thm:KSforExB1} with condition~\eqref{eq:TgS2} in Proposition~\ref{corPelin}~\ref{itmPeliniii}, we immediately conclude that condition~\eqref{eqnCondS1} holds if and only if the perturbation in~\eqref{eqnPertA} is a Hilbert--Schmidt operator.
  In this case, the Hilbert--Schmidt norm of the perturbation in~\eqref{eqnPertA} is given by 
     \begin{align}
       \|\KIO_{\omega - \omega_{\alpha}}\|^2_{2} + 2\,\tr (\KIO_\dip) 
 & = 2\int_0^\infty \biggl(\Wr(x) - c - \frac{x}{1+2\sqrt{\alpha}x} \biggr)^2x\,dx + 2\int_{[0,\infty)}x\,d\dip(x).
 \end{align}
  Let us also mention that the above considerations, together with further results from~\cite[Theorem~5.2]{DSpec} (see also~\cite{AJPR,rowo19}), imply various criteria for the perturbation in~\eqref{eqnPertA} to belong to general Schatten--von Neumann ideals.
\end{remark}

The way we view the generalized indefinite strings in Theorem~\ref{thm:KSforExB2} and Theorem~\ref{thm:HSforAB}~\ref{itmHSforB} as compact perturbations of Example~\ref{exa2alpha} is slightly more involved and hence we present its proof separately. 

\begin{proof}[Proof of Theorem~\ref{thm:HSforAB}~\ref{itmHSforB}]
 It is again easy to see that~\eqref{eq:KIOwComp} implies condition~\eqref{eq:m-anal0Wr} and hence, by Proposition~\ref{corPelin}~\ref{itmPelini}, zero belongs to the resolvent set of $\T$ and the inverse of $\T$ has the form~\eqref{eqnTinvBOM}. 
In particular, for the linear relation $\T_\beta$ corresponding to the generalized indefinite string from Example~\ref{exa2alpha}, one has  
\begin{align*}
\T^{-1}_{\beta} = \begin{pmatrix}0 & \rI_{\dip_\beta}^\ast  \\ \rI_{\dip_\beta} & 0 \end{pmatrix},
\end{align*} 
where $\dip_\beta$ is the measure on $[0,\infty)$ defined by~\eqref{eq:DipBeta}. 
Since $\T^{-1}$ and $\T^{-1}_\beta$ act on different spaces, we are not able to compare them directly. 
However, by decomposing 
\begin{align*}
L^2([0,\infty);\dip) = L^2([0,\infty);\varrho^2)\oplus L^2([0,\infty);\dip_\sing),
\end{align*}
we can first rewrite~\eqref{eqnTinvBOM} as the block operator matrix 
\begin{align*}
\T^{-1} = \begin{pmatrix} \KIO_\omega & \rI_{\varrho^2}^\ast & \rI_{\dip_\sing}^\ast  \\ \rI_{\varrho^2} & 0 & 0 \\ \rI_{\dip_\sing} & 0 & 0 \end{pmatrix}
\end{align*}
acting on $\dot{H}_0^1[0,\infty)\times L^2([0,\infty);\varrho^2)\times L^2([0,\infty);\dip_\sing)$.
Using the isometry 
\begin{align*}
 V_\varrho\colon  L^2([0,\infty);\varrho^2) & \rightarrow L^2[0,\infty) \\  f & \mapsto \varrho f
\end{align*}
 it follows that, except for possibly zero, the spectrum of $T^{-1}$ coincides with the spectrum of the block operator matrix 
\begin{align*}
  \begin{pmatrix} \KIO_\omega & \rI_{\varrho^2}^\ast V_\varrho^\ast & \rI_{\dip_\sing}^\ast  \\ V_\varrho \rI_{\varrho^2} & 0 & 0 \\ \rI_{\dip_\sing} & 0 & 0 \end{pmatrix}
\end{align*}
acting on $\dot{H}_0^1[0,\infty)\times L^2[0,\infty)\times L^2([0,\infty);\dip_\sing)$.
On the same space, we also have the block operator matrix 
\begin{align*}
   \begin{pmatrix}0 & \rI_{\varrho_\beta^2}^\ast V_{\varrho_\beta}^\ast & 0 \\ V_{\varrho_\beta} \rI_{\varrho_\beta^2} & 0 & 0 \\ 0 & 0 & 0\end{pmatrix},
\end{align*} 
whose spectrum coincides with that of $\T^{-1}_{\beta}$ (as the spectrum of $\T^{-1}_\beta$ already contains zero), where $\varrho_\beta$ is the square root of the Radon--Nikod\'ym derivative of $\dip_\beta$. 
 Now the claim follows from Weyl's theorem because, after noting that 
 \begin{align*}
  V_\varrho \rI_{\varrho^2} - V_{\varrho_\beta}\rI_{\varrho_\beta^2} = V_{\varrho-\varrho_\beta} \rI_{(\varrho-\varrho_\beta)^2},
\end{align*}
  the difference of the block operator matrices above becomes
  \begin{align}\label{eqnPertB}
    \begin{pmatrix} \KIO_\omega &  \rI_{(\varrho-\varrho_\beta)^2}^\ast V_{\varrho-\varrho_\beta}^\ast & \rI_{\dip_\sing}^\ast  \\ V_{\varrho-\varrho_\beta} \rI_{(\varrho-\varrho_\beta)^2} & 0 & 0 \\ \rI_{\dip_\sing} & 0 & 0 \end{pmatrix},
\end{align}
 which is compact. 
 In fact, the operators $\KIO_\omega$, $\rI_{(\varrho-\varrho_\beta)^2}$  and $\rI_{\dip_\sing}$ are all compact in view of Proposition~\ref{corPelin}~\ref{itmPelinii} (see also~\cite[Theorem~3.5~(ii) and Theorem~3.6~(ii)]{DSpec} and note that compactness of the embedding $\rI_{\chi}$ is equivalent to that of $\KIO_{\chi}$).
\end{proof}

\begin{remark}\label{rem:tracesB}
 In view of Proposition~\ref{corPelin}~\ref{itmPeliniii}, we notice again that condition~\eqref{eqnCondS2} in Theorem~\ref{thm:KSforExB2} means that the perturbation in~\eqref{eqnPertB} is a Hilbert--Schmidt operator with Hilbert--Schmidt norm given by  
 \begin{align}\begin{split}
 & \|\KIO_{\omega}\|^2_{2} + 2\,\tr(\KIO_{(\varrho-\varrho_\beta)^2})+ 2\,\tr (\KIO_{\dip_\sing}) \\
& \qquad = 2\int_0^\infty  (\Wr(x)-c)^2 x\,dx + 2 \int_0^\infty   \biggl( \varrho(x) - \frac{1}{1+2\beta x} \biggr)^2 x\, dx + 2 \int_{[0,\infty)} x\, d\dip_\sing(x).
 \end{split}\end{align} 
\end{remark}

\section{Spectral asymptotics near zero}\label{secmatzero}

The purpose of this section is to establish the following result about spectral asymptotics of the Weyl--Titchmarsh function of a generalized indefinite string near zero. 
In particular, this will relate the constant $c$ in condition~\eqref{eqnCondS1} in Theorem~\ref{thm:KSforExB1} to the value of the Weyl--Titchmarsh function at zero. 

\begin{theorem}\label{thm:m-at0}
  Let $(L,\omega,\dip)$ be a generalized indefinite string.
  If the corresponding Weyl--Titchmarsh function $m$ has an analytic extension to zero, then $L=\infty$ and  
\begin{align}\label{eq:m-at0}
 m(0) = \lim_{x\rightarrow \infty} \frac{1}{x} \int_0^x \Wr(s) ds.
\end{align}
\end{theorem}

 Notice that Section~\ref{sec:CompactPert} already contains a characterization of generalized indefinite strings whose Weyl--Titchmarsh functions admit an analytic extension to zero. 
Indeed, we first infer from~\eqref{eqnWTmIntRep} that $L$ must be infinite in this case. 
Furthermore, if $L=\infty$, then the singularities of the Weyl--Titchmarsh function $m$ coincide with the spectrum of the linear relation $\T$ and hence $m$ admits an analytic extension to zero exactly when zero does not belong to the spectrum of $\T$ and such a characterization is provided in Proposition~\ref{corPelin}~\ref{itmPelini}. 
The next result summarizes these considerations.

\begin{proposition}\label{lem:GISat0}
 The Weyl--Titchmarsh function $m$ of a generalized indefinite string $(L,\omega,\dip)$ has an analytic extension to zero if and only if $L=\infty$ and~\eqref{eq:m-anal0Wr} holds for a constant $c\in\R$. 
In this case, the constant $c$ is given by~\eqref{eq:def-c_w}.
\end{proposition}

Before presenting proofs of the main result of this section, let us mention the following properties, which, in fact, are immediate consequences of Theorem~\ref{thm:m-at0} and Proposition~\ref{lem:GISat0}.

\begin{corollary}\label{cor:m-at0}
  Let $(L,\omega,\dip)$ be a generalized indefinite string with $L=\infty$ and let $m$ be the corresponding Weyl--Titchmarsh function. 
      \begin{enumerate}[label=(\roman*), ref=(\roman*), leftmargin=*, widest=ii]
\item\label{iCor:m-at0forA} If $(L,\omega,\dip)$ satisfies~\eqref{eq:KIOdipComp} for some $c\in\R$ and $\alpha>0$, then $m$ has an analytic extension to zero and
\begin{align}
m(0) = c+ \frac{1}{2\sqrt{\alpha}}.
\end{align}
\item\label{iCor:m-at0forB} If $(L,\omega,\dip)$ satisfies~\eqref{eq:KIOwComp} for some $c\in\R$ and $\beta>0$, then $m$ has an analytic extension to zero and
\begin{align}
m(0) = c.
\end{align}
\end{enumerate}
\end{corollary}

 We are going to provide two different proofs of Theorem~\ref{thm:m-at0}. 
 Firstly, Theorem~\ref{thm:m-at0} is a particular case of the more general statement below, which can be deduced in a similar way as \cite[Theorem~6.1]{AsymCS}. 

\begin{proposition}\label{lem:GISasymp}
Let $(L,\omega,\dip)$ be a generalized indefinite string with $L=\infty$. 
The corresponding Weyl--Titchmarsh function $m$ satisfies 
\begin{align}\label{eq:m=zeta0}
  m(z) \rightarrow \zeta_0
\end{align}
for some $\zeta_0\in\overline{\C_+}$ as $z\rightarrow0$ nontangentially in $\C_+$ if and only if there are $a\in \R$ and $b\in[0,\infty)$ with $ a^2\le b$ such that
\begin{align}\label{eq:m=zeta0w}
 \lim_{x\rightarrow\infty} \frac{1}{x}\int_0^x \Wr(s)ds & = a, &
 \lim_{x\rightarrow\infty} \frac{1}{x}\int_0^x \Wr(s)^2 ds + \frac{1}{x}\int_{[0,x)}d\dip & = b.
\end{align}
In this case, the limits $a$, $b$ and $\zeta_0$ are connected via 
\begin{align}
a & = \re\,\zeta_0, & b & = |\zeta_0|^2, & \zeta_0 & = a + \I \sqrt{b-a^2}.
\end{align}
\end{proposition}

\begin{proof}
For $r>0$, let $(L_r,\omega_r,\dip_r)$ be the generalized indefinite string with $L_r=\infty$ and coefficients defined by   
\begin{align*}
\Wr_r(x) & = \Wr(rx), & \int_{[0,x)}d\dip_r & = \frac{1}{r}\int_{[0,rx)}d\dip.
\end{align*}  
It is straightforward to check that the solutions $\psi_r(z,\redot)$ of the corresponding differential equation~\eqref{eqnDEho} that lie in $\dot{H}^1[0,\infty)$ and $L^2([0,\infty);\dip_r)$ are related by 
\begin{align*}
\psi_r(z,x)  = \psi(z/r,rx),  
\end{align*}
and hence the corresponding Weyl--Titchmarsh functions are connected via  
\begin{align}\label{eqnWTmr}
m_r(z) = m(z/r).
\end{align}
Furthermore, a simple substitution yields that 
\begin{align}\begin{split}\label{eqnWrInt}
  \int_0^x \Wr_r(s)ds &  = \frac{1}{r}\int_0^{rx}  \Wr(s)ds, \\
 \int_0^x \Wr_r(s)^2 ds + \int_{[0,x)} d\dip_r & =  \frac{1}{r} \int_0^{rx} \Wr(s)^2 ds + \frac{1}{r} \int_{[0,rx)} d\dip.
\end{split}\end{align}
Apart from this, for a given $\zeta_0\in\overline{\C_+}$, we also introduce the generalized indefinite string $(L_0,\omega_0,\dip_0)$ with $L_0=\infty$ and coefficients defined by
\begin{align*}
\Wr_0(x) & = \re\,\zeta_0, &  \int_{[0,x)} d\dip_0 = (\im\,\zeta_0)^2 x,
\end{align*}
so that the corresponding Weyl--Titchmarsh function $m_0$ is constant on $\C_+$ and equal to $\zeta_0$ there. 
Now it remains to note that relation~\eqref{eqnWTmr} entails that $m$ satisfies~\eqref{eq:m=zeta0} as $z\rightarrow0$ nontangentially in $\C_+$ if and only if the functions $m_r$ converge to $\zeta_0$ locally uniformly on $\C_+$ as $r\rightarrow\infty$. 
In the sense of Proposition~\ref{propSMPcont}, this is equivalent to convergence of the generalized indefinite strings $(L_r,\omega_r,\dip_r)$ to $(L_0,\omega_0,\dip_0)$. 
Since this holds if and only if 
 \begin{align*}
   \lim_{r\rightarrow\infty} \int_0^{x} \Wr_r(s)ds  & = \re\,\zeta_0 x, &   \lim_{r\rightarrow\infty} \int_0^x \Wr_r(s)^2 ds + \int_{[0,x)} d\dip_r & = |\zeta_0|^2 x,
 \end{align*}
 for all $x\in[0,\infty)$, the claim follows readily after taking into account~\eqref{eqnWrInt}. 
\end{proof}

It is now easy to see that Theorem~\ref{thm:m-at0} is a consequence of Proposition~\ref{lem:GISasymp}. 

\begin{proof}[Proof of Theorem~\ref{thm:m-at0}]
One only needs to notice that if the Weyl--Titchmarsh function $m$ is analytic at zero, then $m(0)$ is real. 
In particular, this implies that $\zeta_0 = a$ in Proposition~\ref{lem:GISasymp}, so that the first equality in~\eqref{eq:m=zeta0w} becomes~\eqref{eq:m-at0}. 
\end{proof}

Our second proof of Theorem~\ref{thm:m-at0} recovers not only the value of the Weyl--Titchmarsh function at zero, but it can be used, at least in principle, to find the entire Taylor expansion near zero. 
However, deriving higher order terms would involve computing higher derivatives of the functions in Theorem~\ref{thm:FSScartwright} at zero as in Proposition~\ref{propFS}, which becomes extremely cumbersome quickly. 
For this reason, we only provide the first two terms in the Taylor expansion in the following result. 

\begin{proposition}\label{prop:expandm}
Let $(L,\omega,\dip)$ be a generalized indefinite string. 
If the corresponding Weyl--Titchmarsh function $m$ has an analytic extension to zero, then $L=\infty$ and $m$ admits the expansion 
 \begin{align}\label{eqnWTTaylor}
   m(z) & = c +  z \int_0^\infty (\Wr(x) - c)^2 dx+  z \int_{[0,\infty)}d\dip + \OO(z^2)
\end{align}
near zero, where the constant $c = m(0)$ is given by~\eqref{eq:def-c_w}.  
\end{proposition}

\begin{proof}
 The proof is again based on the continuity of the correspondence in Theorem~\ref{thm:IST}.
 However, in contrast to Proposition~\ref{lem:GISasymp}, here we are going to approximate the given generalized indefinite string $(L,\omega,\dip)$, in a way that ensures a uniform spectral gap around zero. 
 More specifically, let us consider the generalized indefinite string $(L_r,\omega_r,\dip_r)$ with $L_r=\infty$ and coefficients given by    
  \begin{align*}
    \Wr_r(x) & = \begin{cases} \Wr(x), & x\leq r, \\ c, & x>r, \end{cases} & \dip_r & = \id_{[0,r)}\dip,
  \end{align*}
  for every $r>0$, where the constant $c$ is given by~\eqref{eq:def-c_w}. 
  It is immediate to check that $(L_r,\omega_r,\dip_r)$ converges to $(L,\omega,\dip)$ in the sense of Proposition~\ref{propSMPcont} as $r\rightarrow\infty$, and hence the corresponding Weyl--Titchmarsh functions $m_r$ converge to $m$ locally uniformly on $\C\backslash\R$.
  Moreover, we clearly have the uniform bound 
  \begin{align*}
  x \int_x^\infty (\Wr_r(s) - c)^2 ds+x \int_{[x,\infty)}d\dip_r \leq x \int_x^\infty (\Wr(s) - c)^2 ds+x \int_{[x,\infty)}d\dip,
  \end{align*}
  where the right-hand side is uniformly bounded for all $x\in[0,\infty)$ in view of Proposition~\ref{corPelin}~\ref{itmPelini}. 
 Due to this bound, the operators $\KIO_{\omega_r}$, $\KIO_{\dip_r}$ and thus also the inverse of $\T_r$ are uniformly bounded (more precisely, this follows from known bounds in~\cite{chev70, muc} on norms of certain integral operators, which are unitarily equivalent to the operators $\KIO_{\chi}$ as explained in the proof of~\cite[Theorem~3.5]{DSpec}). 
    This guarantees that there is a neighborhood of zero to which $m$ and $m_r$ have an analytic extension for all $r>0$.
 As a consequence (for example, this can be seen from the integral representations), the functions $m_r$ also converge to $m$ locally uniformly on a neighborhood of zero, so that we are able to obtain~\eqref{eqnWTTaylor} as the limit of the Taylor expansions of $m_r$. 
  Since the solutions of the differential equation 
     \begin{align*}
  -f'' = z\, \omega_r f + z^2 \dip_r f
 \end{align*}
 are linear to the right of $r$ and coincide with the solutions of~\eqref{eqnDEho} to the left of $r$, we find that the Weyl--Titchmarsh functions $m_r$ are given by 
  \begin{align*}
  m_r(z) = \frac{ c\theta(z,r) - \theta^\qd(z,r)/z}{\phi^\qd(z,r) - cz\phi(z,r)}.
   \end{align*}
 Using the expressions from Proposition~\ref{propFS}, a simple calculation then shows that
  \begin{align*}
m_r(z) & = \frac{c +z c\dot{\theta}(0,r) -  z\ddot{\theta}^\qd(0,r)/2 + \OO(z^2)}{1+z\dot{\phi}^\qd(0,r) - zcr + \OO(z^2)} \\
& = c + z\big(c\dot{\theta}(0,r) - \ddot{\theta}^\qd(0,r)/2 - c\dot{\phi}^\qd(0,r) + c^2 r\big)+ \OO(z^2)\\
& = c + z \int_0^r (\Wr(x)-c)^2 dx + z \int_{[0,r)} d\dip + \OO(z^2)
  \end{align*} 
 near zero and it remains to take the limit as $r\rightarrow\infty$.
\end{proof}

 \section{Relative trace formulas, \ref*{thm:KSforExB1}}\label{secSbSsumrule}

Fix $\alpha>0$ and let $(L,\omega,\dip)$ be a generalized indefinite string with $L=\infty$ and essential spectrum contained in $[\alpha,\infty)$, so that the corresponding Weyl--Titchmarsh function $m$ has a meromorphic extension to $\C\backslash[\alpha,\infty)$ that is analytic at zero. 
We consider the Weyl solution $\psi(k,\redot)$ of the differential equation~\eqref{eqnDEho} with $z=k^2+\alpha$ defined by 
 \begin{align}\label{eq:WeylSol}
    \psi(k,x) &= \theta(z,x) + zm(z)\phi(z,x) 
 \end{align}
 for all $k\in\C_+$ such that $z$ is not a pole of $m$. 
 Notice that the quasi-derivative of the solution $\psi(k,\redot)$ is simply given by 
  \begin{align}\label{eq:WeylSolQD}
    \psi^\qd(k,x) &= \theta^\qd(z,x) + zm(z)\phi^\qd(z,x) 
 \end{align}
 and that one has $\psi^\qd(k,0)=\psi'(k,0-)= zm(z)$. 
 With these solutions, we are able to define the {\em relative Wronskian} $a(\ledot,x)$ for each $x\geq0$ by 
 \begin{align}\label{def:a-function}
   a(k,x) = \frac{W(k,0)}{W(k,x)}, 
 \end{align}
 where the functions $W(\ledot,x)$ are given by 
 \begin{align}
   W(k,x) =  (1+2\sqrt{\alpha}x)^{-\frac{\I k}{2\sqrt{\alpha}}+\frac{1}{2}}\biggl(\frac{zx-\I k+\sqrt{\alpha}}{1+2\sqrt{\alpha}x}\psi(k,x)-\psi^\qd(k,x)\biggr).
 \end{align}
Because the functions $\psi(\ledot,x)$ and $\psi^\qd(\ledot,x)$ are meromorphic on $\C_+$, so are the functions $W(\ledot,x)$ and $a(\ledot,x)$  with $W(\I\sqrt{\alpha},x)=2\sqrt{\alpha}$ and $a(\I\sqrt{\alpha},x)=1$. 
  
 \begin{remark}\label{remExt}
   The relative Wronskian can be viewed as a (suitably regularized\footnote{Under our assumptions on the coefficients $\omega$ and $\dip$, the perturbation is in general not relative trace class, which is in contrast to the case of one-dimensional Schr\"odinger operators in~\cite{kisi09}.})  
   perturbation determinant between two spectral problems (with a parameter $x\geq0$)
   \begin{align}\label{eqnRelGIS}
     -f'' = z\, \omega_x f + z^2 \dip_x f
   \end{align}
  on the interval $(-\ell_\alpha,\infty)$, where $\ell_\alpha$, the anti-derivative $\Wr_x\in L^2_{\loc}(-\ell_\alpha,\infty)$ of the distribution $\omega_x$ and the Borel measure $\dip_x$ on $(-\ell_\alpha,\infty)$ are given by
  \begin{align}
   \ell_\alpha & = \frac{1}{2\sqrt{\alpha}}, & \Wr_x(s) & = \begin{cases} \frac{s}{1+2\sqrt{\alpha}s}, & s<x, \\ \Wr(s), & s\geq x, \end{cases} & \dip_x(B) & = \dip(B\cap[x,\infty)).
  \end{align}
  In fact, it is the quotient between two characteristic Wronskians for these spectral problems. 
 Up to cancellations, the zeros and poles of the relative Wronskian correspond to the eigenvalues below $\alpha$ of~\eqref{eqnRelGIS} and the eigenvalues below $\alpha$ of~\eqref{eqnRelGIS} with $x=0$. 
 The relative trace formula that we will derive in this section is in essence the difference of trace formulas for these two spectral problems. 
 \end{remark}
 
 In order to prove the main result of this section, we first need to express the derivatives of the relative Wronskian at $k=\I\sqrt{\alpha}$ in terms of the coefficients of our generalized indefinite string.
 Remember that we always use a dot to denote differentiation with respect to the complex variable (which may be $k$ or $z$ here).

 \begin{proposition}\label{prop:a-decomp}
   Let $\alpha>0$ and let $(L,\omega,\dip)$ be a generalized indefinite string with essential spectrum contained in $[\alpha,\infty)$ and 
  \begin{align}\label{eq:matzero}
    m(0)=\frac{1}{2\sqrt{\alpha}}.
  \end{align}
  Then the relative Wronskian $a(\ledot,x)$ satisfies  
 \begin{align}\label{eq:a-decomp02}   
 \dot{a}(\I\sqrt{\alpha},x) = 2\I\sqrt{\alpha} \int_0^x \biggl(\Wr(s)-\frac{s}{1+2\sqrt{\alpha}s}\biggr)ds 
  \end{align}
   for every $x\geq0$, as well as the identity
  \begin{align}\label{eq:a-decomp03}
     \begin{split} 
  &  \ddot{a}(\I\sqrt{\alpha},x) - \dot{a}(\I\sqrt{\alpha},x)^2 - \frac{\I}{\sqrt{\alpha}}\dot{a}(\I\sqrt{\alpha},x) \\
 & \qquad =  4\sqrt{\alpha}\int_0^x (1+2\sqrt{\alpha}s)\biggl(\Wr(s)-\frac{s}{1+2\sqrt{\alpha}s}\biggr)^2 ds + 4\sqrt{\alpha}\int_{[0,x)} (1+2\sqrt{\alpha}s)d\dip(s). 
 \end{split}
 \end{align} 
 \end{proposition}
 
 \begin{proof}
  The claims basically follow from cumbersome calculations using the definition of $a$ and the expressions in Proposition~\ref{propFS}.  
 First of all, let us mention that~\eqref{eq:matzero} means that $m$ has an analytic extension to zero and hence $L$ is infinite (see Theorem~\ref{thm:m-at0}). 
 Next, we compute 
 \begin{align*}
   \dot{\psi}(\I\sqrt{\alpha},x) & 
      = - 2\I\sqrt{\alpha} \int_0^x \biggl(\Wr(s)-\frac{1}{2\sqrt{\alpha}}\biggr) ds, \\
   \ddot{\psi}(\I\sqrt{\alpha},x) & 
    =  -2\int_0^x \Wr(s)ds - 4\alpha\biggl(\int_0^x \Wr(s)ds\biggr)^2  - 8\alpha\,\dot{m}(0)x+ 8\alpha \int_0^x \int_0^s \Wr(t)^2 dt\,ds  \\
   & \qquad + 8\alpha \int_0^x \int_{[0,s)} d\dip\,ds +\frac{x}{\sqrt{\alpha}} - 4\sqrt{\alpha}x\int_0^x \Wr(s)ds + 8\sqrt{\alpha}\int_0^x \Wr(s)s\,ds, 
 \end{align*}
 where we also used the identity
 \begin{align*}
  \int_0^x \int_0^s \Wr(t)dt\, ds  & = x\int_0^x \Wr(s) ds - \int_0^x \Wr(s) s\, ds.
 \end{align*}
 Furthermore, for the quasi-derivative of $\psi$, one gets $\dot{\psi}^\qd(\I\sqrt{\alpha},x)=\I$ and 
 \begin{align*}
   \ddot{\psi}^\qd(\I\sqrt{\alpha},x) & 
    =  8\alpha\int_0^x \Wr(s)^2ds + 8\alpha\int_{[0,x)}d\dip + \frac{1}{\sqrt{\alpha}} - 4\sqrt{\alpha}\int_0^x \Wr(s)ds - 8\alpha\,\dot{m}(0).
 \end{align*}
We next compute the derivative for the function $W$ and get 
 \begin{align*}
   \dot{W}(\I\sqrt{\alpha},x) & = 2\I\sqrt{\alpha} x - \I \log(1+2\sqrt{\alpha}x) - 4\I\alpha \int_0^x \Wr(s)ds -  2\I  \\
     & =  -4\I\alpha \int_0^x \biggl(\Wr(s)-\frac{s}{1+2\sqrt{\alpha}s}\biggr)ds -2\I,
       \end{align*}
 where we used that 
 \begin{align*}
    2\sqrt{\alpha}x-  \log(1+2\sqrt{\alpha}x) = 4\alpha \int_0^x \frac{s}{1+2\sqrt{\alpha}s}ds. 
 \end{align*}
 This readily yields~\eqref{eq:a-decomp02}. 
 For the second derivative, one first has  
 \begin{align*}
  &  \frac{\ddot{W}(\I\sqrt{\alpha},x)}{W(\I\sqrt{\alpha},x)} - \frac{\dot{W}(\I\sqrt{\alpha},x)^2}{W(\I\sqrt{\alpha},x)^2} \\
  & \qquad = 2\int_0^x \Wr(s)ds -x^2  + \frac{1}{2\alpha} + 4\sqrt{\alpha}\,\dot{m}(0)+ 8\sqrt{\alpha}\int_0^x \Wr(s)s\,ds \\
     & \qquad\qquad -4\sqrt{\alpha} \int_0^x (1+2\sqrt{\alpha}s)\Wr(s)^2 ds - 4\sqrt{\alpha} \int_{[0,x)} (1+2\sqrt{\alpha}s)d\dip(s)
\end{align*}
and furthermore  
 \begin{align*}
  &  \frac{\ddot{W}(\I\sqrt{\alpha},x)}{W(\I\sqrt{\alpha},x)} - \frac{\dot{W}(\I\sqrt{\alpha},x)^2}{W(\I\sqrt{\alpha},x)^2} -\frac{\I}{\sqrt{\alpha}} \frac{\dot{W}(\I\sqrt{\alpha},x)}{W(\I\sqrt{\alpha},x)} \\
  & \qquad\qquad = -4\sqrt{\alpha} \int_0^x (1+2\sqrt{\alpha}s)\biggl(\Wr(s)-\frac{s}{1+2\sqrt{\alpha}s}\biggr)^2 ds \\
     & \qquad\qquad\qquad  - 4\sqrt{\alpha} \int_{[0,x)} (1+2\sqrt{\alpha}s)d\dip(s)-\frac{1}{2\alpha} + 4\sqrt{\alpha}\,\dot{m}(0), 
\end{align*}
 where we used that 
 \begin{align*}
    x^2- 2 \int_0^x \frac{s}{1+2\sqrt{\alpha}s}ds = 4\sqrt{\alpha} \int_0^x \frac{s^2}{1+2\sqrt{\alpha}s}ds. 
 \end{align*}
 The remaining identity~\eqref{eq:a-decomp03} now follows without much effort.
  \end{proof}
   
 This brings us to the key ingredient for the proof of Theorem~\ref{thm:KSforExB1}; {\em relative trace formulas} (called {\em step-by-step sum rules} in \cite{kisi09,kisi03}). 
 In order to state them, we introduce the two functions $\cF_1$ and $\cF_2$ on $(0,1)\cup(1,\infty)$ by 
\begin{align}\label{eq:Fsdef}
\cF_1(s) &= \frac{2s}{s^2 - 1} + \log\biggl| \frac{s-1}{s+1}\biggr|, & \cF_2(s) &= \frac{2s^3 + 2s}{(s^2 - 1)^2} + \log\biggl| \frac{s-1}{s+1}\biggr|.
\end{align}
For basic properties of these functions we refer to Appendix~\ref{app:F1F2}.
 
 \begin{theorem}[Relative trace formulas]\label{thm:traceflaA}
Let $\alpha>0$ and let $(L,\omega,\dip)$ be a generalized indefinite string with essential spectrum contained in $[\alpha,\infty)$ and~\eqref{eq:matzero}.
Then for every $x\geq0$, the limit $a(\xi,x) = \lim_{\varepsilon\downarrow 0}a(\xi+\I\varepsilon,x)$ exists and is nonzero for almost all $\xi \in\R$, satisfies
  \begin{align}\label{eq:AlogSzego}
      \int_{\R} \frac{|\log|a(\xi,x)||}{1+\xi^4}d\xi <\infty
   \end{align} 
 and the identities 
   \begin{align} \label{eq:traceflaA01}
      \begin{split}
     &  \int_0^x \biggl(\Wr(s)-\frac{s}{1+2\sqrt{\alpha}s}\biggr) ds  \\ 
     & \qquad =   \frac{\sqrt{\alpha}}{\pi} \int_{\R} \frac{\log|a(\xi,x)|}{(\xi^2+\alpha)^2}d\xi +  \frac{1}{2\alpha} \lim_{\delta\downarrow0} \sum_{\kappa\in{\rm Z}\atop\delta<\kappa<1/\delta} \cF_1\biggl(\frac{\kappa}{\sqrt{\alpha}}\biggr) - \sum_{\eta\in{\rm P}\atop\delta<\eta<1/\delta}\cF_1\biggl(\frac{\eta}{\sqrt{\alpha}}\biggr),
     \end{split}
\\ \label{eq:traceflaA02}
  \begin{split}
  &  \int_0^x (1+2\sqrt{\alpha}s)\biggl(\Wr(s)-\frac{s}{1+2\sqrt{\alpha}s}\biggr)^2ds +  \int_{[0,x)} (1+2\sqrt{\alpha}s)d\dip(s)  \\ 
  & \qquad = \frac{2}{\pi} \int_{\R} \frac{\xi^2 \log|a(\xi,x)|}{(\xi^2+\alpha)^3}d\xi  +  \frac{1}{4\alpha^{\nicefrac{3}{2}}}\lim_{\delta\downarrow0}  \sum_{\kappa\in{\rm Z}\atop\delta<\kappa<1/\delta} \cF_2\biggl(\frac{\kappa}{\sqrt{\alpha}}\biggr) - \sum_{\eta\in{\rm P}\atop\delta<\eta<1/\delta} \cF_2\biggl(\frac{\eta}{\sqrt{\alpha}}\biggr),
\end{split}
\end{align}
  hold, where ${\rm Z} = \{\kappa>0\,|\, a(\I\kappa,x)=0\}$ and ${\rm P} = \{\eta>0\,|\, a(\I\eta,x)=\infty\}$.
 \end{theorem}
 
 \begin{proof} 
  We first note that the function $a(\ledot,x)$ can be written as 
   \begin{align*}
     a(k,x) & =  (1+2\sqrt{\alpha}x)^{\frac{\I k}{2\sqrt{\alpha}}-\frac{1}{2}} \frac{G(k,0)}{G(k,x)} \frac{1}{\psi(k,x)},
   \end{align*}
  where the meromorphic functions $G(\ledot,x)$ on $\C_+$ are given by 
  \begin{align*}
     G(k,x) & = \frac{\psi^\qd(k,x)}{z\psi(k,x)} - \frac{1}{(1+2\sqrt{\alpha}x)(\sqrt{\alpha}+\I k)} - \frac{x}{1+2\sqrt{\alpha}x}. 
  \end{align*}
    It follows readily that $\im\,G(k,x)>0$ when $\re\,k>0$ and $\im\,G(k,x)< 0$ when $\re\,k<0$ because the first term is a Herglotz--Nevanlinna function in $z$. 
   As a consequence, all zeros and poles of $G(\ledot,x)$, $G(\ledot,0)$ and $\psi(\ledot,x)$ are simple and lie on the imaginary axis. 
  The fact that the zeros and poles of $a(\ledot,x)$ are simple as well is a result of the following three observations:
  \begin{enumerate}[label=(\alph*), ref=(\alph*), leftmargin=*, widest=a]
  \item\label{itmTFzp1a} If $k$ is a zero of $\psi(\ledot,x)$, then $k$ is a pole of $G(\ledot,x)$ because $k$ can not be a zero of $\psi^\qd(\ledot, x)$ as well in this case.  
  \item\label{itmTFzp1b} If $k$ is a pole of $\psi(\ledot,x)$, then $k$ is a pole of $G(\ledot,0)$ because $z=k^2+\alpha$ must be a pole of $m$ in this case, which is the first term of $G(\ledot,0)$. 
  \item\label{itmTFzp1c} A $k$ that is neither a zero nor a pole of $\psi(\ledot,x)$ is a pole of $G(\ledot,x)$ if and only if it is a pole of $G(\ledot,0)$.
  In fact, this holds true because such a $k$ is a pole of $\psi^\qd(\ledot,x)$ if and only if $z=k^2+\alpha$ is a pole of $m$. 
  \end{enumerate}
   For $x>0$ (the claims are trivial otherwise), we next write the function $a(\ledot,x)$ as 
     \begin{align*}
     a(k,x) & = \frac{\phi(k^2+\alpha,x)}{\phi_\alpha(k^2+\alpha,x)} \frac{G(k,0)}{G(k,x)}\frac{G_{+}(k)}{G_{-}(k)},
   \end{align*}
    where $\phi_\alpha$ is the respective solution for the generalized indefinite string in Example~\ref{exaalpha}, given explicitly by
    \begin{align*}
      \phi_\alpha(z,x) = \frac{(1+2\sqrt{\alpha}x)^{\frac{1}{2}}}{\sqrt{z-\alpha}} \sin\biggl(\frac{\sqrt{z-\alpha}}{2\sqrt{\alpha}}\log(1+2\sqrt{\alpha}x)\biggr),
    \end{align*}
    and the meromorphic functions $G_{\pm}$ on $\C_+$ are given by 
  \begin{align*}
     G_{-}(k) & = z\phi(z,x)\psi(k,x) = \biggl(\frac{\phi^\qd(z,x)}{z\phi(z,x)} - \frac{\psi^\qd(k,x)}{z\psi(k,x)}\biggr)^{-1}, \\
     G_{+}(k) & = z\phi_\alpha(z,x)(1+2\sqrt{\alpha}x)^{\frac{\I k}{2\sqrt{\alpha}}-\frac{1}{2}} \\
      & = \biggl(\frac{k}{z}\cot\biggl(\frac{k}{2\sqrt{\alpha}}\log(1+2\sqrt{\alpha}x)\biggr) - \frac{\I k}{z}\biggr)^{-1}.
  \end{align*}
   Since one has that $\im\,G_\pm(k)\geq0$ when $\re\,k>0$ and $\im\,G_\pm(k)\leq 0$ when $\re\,k<0$, the claims readily follow from Theorem~\ref{thm:FactorMain} and Proposition~\ref{prop:a-decomp}.
  \end{proof}
   
   Under an additional assumption on the Weyl--Titchmarsh function, the integrals on the spectral side of the relative trace formulas can also be expressed in terms of suitable {\em transmission coefficients} for the spectral problems in Remark~\ref{remExt}. 
    To this end, we first define the meromorphic functions $M(\ledot,x)$ on $\C_+$ by 
    \begin{align}
       M(k,x) = \frac{\psi^\qd(k,x)}{z\psi(k,x)}
    \end{align}
    for $x\geq0$, so that $M(k,0)= m(z)$. 
    The limit $M(\xi,x)=\lim_{\varepsilon\downarrow0}M(\xi +\I\varepsilon,x)$ exists for almost all $\xi\in\R$ (see Theorem~\ref{th:SimonFactor} for example) and we may define 
    \begin{align}\label{eq:defTx}
       T(\xi,x) =  \frac{4\xi}{\xi^2+\alpha}  \frac{(1+2\sqrt{\alpha}x) \im\, M(\xi,x)}{\Bigl| (1+2\sqrt{\alpha}x) M(\xi,x) -\frac{1}{\sqrt{\alpha}+\I \xi}-x\Bigr|^2}
    \end{align}
    as the denominator is positive for almost all $\xi\in\R$.
    In fact, one has 
    \begin{align}\label{eqnTransEst}\begin{split}
      \biggl| (1+2\sqrt{\alpha}x) M(\xi,x) -\frac{1}{\sqrt{\alpha}+\I \xi}-x\biggr|^2 & \geq \biggl( (1+2\sqrt{\alpha}x) \im\,M(\xi,x) +\frac{\xi}{\xi^2+\alpha}\biggr)^2 \\
      & \geq \frac{4\xi}{\xi^2+\alpha}(1+2\sqrt{\alpha}x) \im\,M(\xi,x),
    \end{split}\end{align}
    which also entails that $T(\xi,x)\in[0,1]$. 
    
    \begin{proposition}\label{proptraceflaAtrans}
Let $\alpha>0$ and let $(L,\omega,\dip)$ be a generalized indefinite string with essential spectrum contained in $[\alpha,\infty)$ and~\eqref{eq:matzero}.    
  If $\im\,m(\lambda+\I0)>0$ for almost all $\lambda>\alpha$, then 
 \begin{align}\label{eqnaasT}
   |a(\xi,x)|^2 = \frac{T(\xi,x)}{T(\xi,0)}
 \end{align}
 for almost all $\xi\in\R$ and all $x\geq 0$. 
 \end{proposition}

\begin{proof}
   Let $\xi>0$ be such that $\im\,m(\lambda+\I0)>0$, where $\lambda=\xi^2+\alpha$. 
   We consider the solution $\psi(\xi,\redot)$ of the differential equation~\eqref{eqnDEho} with $z=\lambda$ such that 
   \begin{align*}
     \psi(\xi,x) & = \theta(\lambda,x)+\lambda m(\lambda+\I0) \phi(\lambda,x) = \lim_{\varepsilon\downarrow0} \psi(\xi+\I\varepsilon,x), \\
     \psi^\qd(\xi,x) & = \theta^\qd(\lambda,x) + \lambda m(\lambda+\I0)\phi^\qd(\lambda,x) = \lim_{\varepsilon\downarrow0} \psi^\qd(\xi+\I\varepsilon,x). 
   \end{align*} 
   Notice that $\psi(\xi,x)\not=0$ for all $x\geq0$ because it is not possible for $\theta(\lambda,x)$ and $\phi(\lambda,x)$ to vanish both. 
   It follows that the limit 
   \begin{align*}
     M(\xi,x) = \lim_{\varepsilon\downarrow0} M(\xi+\I\varepsilon,x) = \frac{\psi^\qd(\xi,x)}{\lambda\psi(\xi,x)}
   \end{align*}
   exists for all $x\geq0$ and satisfies 
   \begin{align*}
     \im\,M(\xi,x) & = \frac{\im\,\psi(\xi,x)^\ast \psi^\qd(\xi,x)}{\lambda|\psi(\xi,x)|^2} = \frac{\im\,m(\lambda+\I0)}{|\psi(\xi,x)|^2} > 0,
   \end{align*}
   where we used that $\theta(\lambda,x)\phi^\qd(\lambda,x)-\theta^\qd(\lambda,x)\phi(\lambda,x)=1$. 
   Next, one gets 
   \begin{align*}
     \Bigl|\lim_{\varepsilon\downarrow0} W(\xi+\I\varepsilon,x)\Bigr|^2 & = \frac{\lambda^2 \im\,m(\lambda+\I0)}{(1+2\sqrt{\alpha}x)\im\,M(\xi,x)} \biggl| (1+2\sqrt{\alpha}x)M(\xi,x) - \frac{1}{\sqrt{\alpha}+\I \xi} - x\biggr|^2, 
   \end{align*}
   which is positive in view of~\eqref{eqnTransEst}.
   We conclude that the limit $\lim_{\varepsilon\downarrow0} a(\xi+\I\varepsilon,x)$ exists as well and satisfies~\eqref{eqnaasT} as claimed. 
   Similar arguments (with a few sign changes) show that this also holds true for all $\xi<0$ with $\im\,m(\xi^2+\alpha+\I0)>0$.
\end{proof} 

 Let us also mention that the sums in the relative trace formulas in Theorem~\ref{thm:traceflaA} can be written in terms of the eigenvalues of the spectral problems in Remark~\ref{remExt}. 
 In fact, it is possible to replace ${\rm Z}$ with ${\rm D}_0$ and ${\rm P}$ with ${\rm D}_x$ because one has 
 \begin{align}
    {\rm Z} & = {\rm D}_0\backslash({\rm D}_0\cap{\rm D}_x), & {\rm P} & = {\rm D}_x\backslash({\rm D}_0\cap{\rm D}_x), 
 \end{align}
 where the sets ${\rm D}_x$ are defined by 
 \begin{align}
   {\rm D}_x = \biggl\{\kappa>0\,\bigg|\, (1+2\sqrt{\alpha}x)M(\I\kappa,x)=\frac{1}{\sqrt{\alpha}-\kappa}+x\biggr\},
 \end{align}
 so that  $\{-\kappa^2+ \alpha\,|\,\kappa\in{\rm D}_x\}$ are precisely the eigenvalues below $\alpha$ of the spectral problem in Remark~\ref{remExt} with $\xi= x$.

\section{Lower semi-continuity, \ref*{thm:KSforExB1}}\label{secLowSemiContA}

Let $(L,\omega,\dip)$ be a generalized indefinite string and let $m$ be the corresponding Weyl--Titchmarsh function. 
For $\alpha>0$, we define the {\em transmission coefficient} $T$ by
\begin{align}\label{eqnTdef}
 T(\xi) = \frac{4|\xi|}{\xi^2+\alpha} \frac{\im\, m(\xi^2+\alpha+\I0)}{\Bigl|m(\xi^2+\alpha+\I0)-\frac{1}{\sqrt{\alpha}+\I |\xi|}\Bigr|^2},
\end{align}
which is well-defined for almost all $\xi\in\R$ and belongs to $[0,1]$ because 
\begin{align}\begin{split}
   \biggl|m(\xi^2+\alpha+\I0)-\frac{1}{\sqrt{\alpha}+\I |\xi|}\biggr|^2 & \geq  \biggl(\im\,m(\xi^2+\alpha+\I0)+\frac{|\xi|}{\xi^2+\alpha}\biggr)^2 \\
     & \geq \frac{4|\xi|}{\xi^2+\alpha} \im\,m(\xi^2+\alpha+\I0).
\end{split}\end{align} 
Notice that this function $T$ coincides with $T(\ledot,0)$ as given by~\eqref{eq:defTx} almost everywhere.
The purpose of this section is to establish that the quantity $\cQ_\alpha\in[0,\infty]$ defined by 
\begin{align}
\cQ_\alpha = -\int_\R \frac{\xi^2 \log T(\xi)}{(\xi^2+\alpha)^3} d\xi
\end{align}
depends lower semi-continuously on the generalized indefinite string $(L,\omega,\dip)$, where the set of all generalized indefinite strings is endowed with the topology that makes the correspondence $(L,\omega,\dip)\mapsto m$ homeomorphic (and the set of all Herglotz--Nevanlinna functions is endowed with the topology of locally uniform convergence); compare Proposition~\ref{propSMPcont}. 

\begin{theorem}\label{thm:entropyA}
  The functional $(L,\omega,\dip)\mapsto\cQ_\alpha$ is lower semi-continuous on the set of all generalized indefinite strings for every $\alpha>0$.
\end{theorem}

We will need the following simple but useful estimate for the proof of this result.

\begin{lemma}\label{lem:Est-r}
For each $\alpha>0$ one has the bound  
\begin{align}\label{eq:Est-r}
  \Biggl| \frac{z - \frac{1}{\sqrt{\alpha}-\I k}}{z - \frac{1}{\sqrt{\alpha}+\I k}}\Biggr| \leq 1+ \frac{2|k|}{\re\,k}
\end{align}
when $z\in\C_+\cup\R$ and $k\in\C_+$ with $\re\,k>0$.
\end{lemma}

\begin{proof}
 We first notice that 
 \begin{align*}
    \biggl|z - \frac{1}{\sqrt{\alpha}+\I k}\biggr| \geq \biggl|\im\biggl(z - \frac{1}{\sqrt{\alpha}+\I k}\biggr)\biggr| = \biggl|\im\, z+\frac{\re\,k}{|\sqrt{\alpha}+\I k|^2}\biggr| \geq \frac{\re\,k}{|\sqrt{\alpha}+\I k|^2}
 \end{align*}
  because $\im\, z\geq0$ and $\re\,k>0$. 
 From this we get  
\begin{align*}
  \Biggl| \frac{z - \frac{1}{\sqrt{\alpha}-\I k}}{z - \frac{1}{\sqrt{\alpha}+\I k}}\Biggr| & =  \Biggl| 1-\frac{2\I k}{k^2+\alpha}\frac{1}{z - \frac{1}{\sqrt{\alpha}+\I k}}\Biggr|  \leq 1+ \frac{2 |k|}{|k^2+\alpha|} \frac{|\sqrt{\alpha}+\I k|^2}{\re\,k}
\end{align*}
and the claim follows because 
\begin{align*}
|k^2+\alpha| = |\sqrt{\alpha} - \I k||\sqrt{\alpha} +\I k| \ge  |\sqrt{\alpha} + \I k |^2
\end{align*}
 when $k\in\C_+$.
\end{proof}

\begin{remark}
There are at least two ways to compute the supremum (as $z$ runs through $\C_+$) of the left-hand side in~\eqref{eq:Est-r}. Indeed, under the corresponding assumptions on $k$, the linear fractional transformation $A$ given by 
\begin{align}
  A(z) =  \frac{z - \frac{1}{\sqrt{\alpha}-\I k}}{z - \frac{1}{\sqrt{\alpha}+\I k}}
\end{align}
 maps $\C_+$ into a disc and hence one simply needs to use formulas for the center and radius of this disc. Another option is to find $\sup_{x\in \R} |A(x)|$ by using basic calculus tools. However, the corresponding expression is cumbersome and the estimate \eqref{eq:Est-r} is sufficient for our purposes.
\end{remark}

 The actual proof of the claim in Theorem~\ref{thm:entropyA} now proceeds essentially along the lines of the one in \cite[Section~5]{kisi09}.

\begin{proof}[Proof of Theorem~\ref{thm:entropyA}]
For a generalized indefinite string $(L,\omega,\dip)$, we introduce the analog of the {\em reflection coefficient} $r$ defined on the first complex quadrant by
\begin{align*}
r(k) = \frac{m(k^2+\alpha) - \frac{1}{\sqrt{\alpha}-\I k}}{m(k^2+\alpha) - \frac{1}{\sqrt{\alpha}+\I k}}.
\end{align*}
Notice that this function is analytic and obeys the bound 
\begin{align}\label{eqnBoundonr}
  |r(k)| \leq 1 + \frac{2|k|}{\re\, k}
\end{align}
in view of  Lemma~\ref{lem:Est-r}.
The limit $r(\xi)=\lim_{\varepsilon\downarrow0} r(\xi+\I\varepsilon)$ exists and satisfies the {\em scattering relation}
\begin{align*}
T(\xi) + |r(\xi)|^2 = 1
\end{align*}
for almost every $\xi>0$, so that $|r(\xi)|\leq 1$ and, unless $T(\xi)=0$, also
\begin{align}\label{eqnTasr}
-\log T(\xi) = -\log \bigl(1-|r(\xi)|^2\bigr) =  \sum_{j=1}^\infty \frac{|r(\xi)|^{2j}}{j}.
\end{align}
More precisely, to obtain the scattering relation one first computes that 
\begin{align*}
 &  \biggl|m(\xi^2+\alpha+\I 0)-\frac{1}{\sqrt{\alpha}+\I\xi}\biggr|^2 - \biggl|m(\xi^2+\alpha+\I 0)-\frac{1}{\sqrt{\alpha}-\I\xi}\biggr|^2 \\
   & \qquad\qquad\qquad\qquad\qquad\qquad\qquad\qquad = \frac{4\xi}{\xi^2+\alpha} \im\,m(\xi^2+\alpha+\I0)
\end{align*}
and then divides by the first term on the left-hand side. 
Furthermore, we will later use that when $0<a<b$ and $q$ is a polynomial, contour integration of the analytic function $rq$ over the boundary of the rectangle with vertices $a+\I\varepsilon$, $b+\I\varepsilon$, $b+\I$, $a+\I$ and subsequently letting $\varepsilon\downarrow0$ gives 
\begin{align}\label{eqnrqintRec}
  \int_a^b r(\xi)q(\xi)d\xi = \int_{a}^{a+\I} r(k)q(k)dk + \int_{a+\I}^{b+\I} r(k)q(k)dk + \int_{b+\I}^b r(k)q(k)dk, 
\end{align} 
where we also took the bound in~\eqref{eqnBoundonr} into account.

In order to prove that $\cQ_\alpha$ is lower semi-continuous, let $(L_n,\omega_n,\dip_n)$ be a sequence of generalized indefinite strings that converge to $(L,\omega,\dip)$, so that the corresponding Weyl--Titchmarsh functions $m_n$ converge locally uniformly to $m$. 
Clearly, the corresponding reflection coefficients $r_n$ then converge locally uniformly to $r$ on the first complex quadrant. 
Given $0<a<b$ and a polynomial $q$, it follows from~\eqref{eqnrqintRec}, the convergence of $r_n$ to $r$ and the uniform bound obtained from~\eqref{eqnBoundonr} that 
\begin{align*}
   \lim_{n\rightarrow \infty} \int_a^b r_n(\xi)q(\xi) d\xi = \int_a^b  r(\xi)q(\xi) d\xi.
\end{align*}
Because the functions $r_n$ are also uniformly bounded in $L^\infty[a,b]$, we conclude that they converge to $r$ in $L^\infty[a,b]$ with respect to the weak$^\ast$ topology in the sense that 
\begin{align*}
   \lim_{n\rightarrow \infty} \int_a^b r_n(\xi) h(\xi) d\xi = \int_a^b r(\xi) h(\xi) d\xi
\end{align*}
for every function $h\in L^1[a,b]$. 
This weak$^\ast$ convergence in $L^\infty[a,b]$ entails weak convergence in the weighted space $L^p([a,b]; \xi^2(\xi^2+\alpha)^{-3}d\xi)$ for each $p>1$. 
Since the norm on a Banach space is weakly lower semi-continuous, this gives 
\begin{align*}
\liminf_{n\rightarrow \infty} \int_a^b |r_n(\xi)|^p \frac{\xi^2}{(\xi^2+\alpha)^3} d\xi \geq  \int_a^b |r(\xi)|^p \frac{\xi^2}{(\xi^2+\alpha)^3} d\xi.
\end{align*}
It then follows from~\eqref{eqnTasr} that for every $J\in\N$ one has 
\begin{align*}
\liminf_{n\rightarrow \infty} -\int_0^\infty \frac{\xi^2 \log T_{n}(\xi)}{(\xi^2+\alpha)^3} d\xi & \geq
\liminf_{n\rightarrow \infty} \int_a^b \sum_{j=1}^J \frac{|r_n(\xi)|^{2j}}{j} \frac{\xi^2}{(\xi^2+\alpha)^3} d\xi \\
 & \geq  \int_a^b \sum_{j=1}^J \frac{|r(\xi)|^{2j}}{j} \frac{\xi^2}{(\xi^2+\alpha)^3} d\xi,
\end{align*}
where $T_n$ are the corresponding transmission coefficients. 
The claim can now be deduced readily by letting $J\rightarrow\infty$, $a\downarrow 0$ and $b\rightarrow \infty$ (also notice that the transmission coefficients are even by definition).     
\end{proof}

\section{Proof of Theorem~\ref*{thm:KSforExB1}}\label{secKSB1proof} 

We are now ready to put together the proof for our first main result. 
 In fact, taking into account Corollary~\ref{cor:m-at0}~\ref{iCor:m-at0forA}, we are going to prove the following slightly amended statement:
{\em  A Herglotz--Nevanlinna function $m$ is the Weyl--Titchmarsh function of a generalized indefinite string $(L,\omega,\dip)$ with $L=\infty$ and~\eqref{eqnCondS1} with $c=0$ if and only if the conditions~\ref{itmKSc1},~\ref{itmKSc2},~\ref{itmKSc3} in Theorem~\ref{thm:KSforExB1} and  }
     \begin{align}\label{eqnKSmatzero}
      m(0) =  \frac{1}{2\sqrt{\alpha}}
    \end{align}
{\em      hold.} 
  Theorem~\ref{thm:KSforExB1} then follows readily upon noticing that adding a constant to $\Wr$ amounts to adding the same constant to the Weyl--Titchmarsh function $m$. 

\begin{proof}[Proof of necessity] 
Suppose first that $m$ is the Weyl--Titchmarsh function of a generalized indefinite string $(L,\omega,\dip)$ with $L=\infty$ that satisfies~\eqref{eqnCondS1} with $c=0$.
It follows readily from Theorem~\ref{thm:HSforAB}~\ref{itmHSforA} and Corollary~\ref{cor:m-at0}~\ref{iCor:m-at0forA} that $m$ has a meromorphic extension to $\C\backslash[\alpha,\infty)$ that is analytic at zero and that~\eqref{eqnKSmatzero} holds.
In order to verify the remaining conditions~\ref{itmKSc2} and~\ref{itmKSc3}, let us consider the sequence of generalized indefinite strings $(L,\omega_n,\dip_n)$ whose coefficients are given by 
\begin{align*}
   \Wr_n(x) & = \begin{cases} \Wr(x), & x<n, \\  \frac{x}{1+2\sqrt{\alpha}x}, & x\geq n, \end{cases}  & \dip_n(B) & = \dip(B\cap[0,n)),
\end{align*} 
so that $(L,\omega_n,\dip_n)$ converges to $(L,\omega,\dip)$ in the sense of Proposition~\ref{propSMPcont}. 
We will use the notation from Section~\ref{sec:prelim} and Section~\ref{secSbSsumrule}, where all quantities corresponding to the generalized indefinite strings $(L,\omega_n,\dip_n)$ will be denoted with an additional subscript. 
Theorem~\ref{thm:HSforAB}~\ref{itmHSforA} and Corollary~\ref{cor:m-at0}~\ref{iCor:m-at0forA} guarantee that each $(L,\omega_n,\dip_n)$ satisfies the requirements of Theorem~\ref{thm:traceflaA}. 
For large enough $x>n$, the solutions $\psi_n(k,\redot)$ are scalar multiples of the ones in Example~\ref{exaalpha} and thus 
 \begin{align}\label{eqnMnlargex}
   (1+2\sqrt{\alpha}x) M_n(k,x) & = (1+2\sqrt{\alpha}x)\frac{\psi_n^\qd(k,x)}{z\psi_n(k,x)}= \frac{1}{\sqrt{\alpha}-\I k} +x,
 \end{align}  
 where we continue to use $z=k^2+\alpha$.
 Because we may write 
 \begin{align*}
   M_n(k,0) = \frac{\theta_n(z,x)M_n(k,x)-\theta_n^\qd(z,x)/z}{\phi_n^\qd(z,x)-z\phi_n(z,x)M_n(k,x)},
 \end{align*}
 we see that $M_n(\xi,0)=\lim_{\varepsilon\downarrow0}M_n(\xi+\I\varepsilon,0)$ exists for all $\xi\in\R$ with $\xi\not=0$ and 
  \begin{align*}
   \im\,M_n(\xi,0)  = \frac{\im\,M_n(\xi,x)}{\bigl|\phi_n^\qd(\xi^2+\alpha,x)-(\xi^2+\alpha)\phi_n(\xi^2+\alpha,x)M_n(\xi,x)\bigr|^2}.
 \end{align*}
  Since $M_n(k,0)=m_n(z)$, this implies that $\im\,m_n(\lambda+\I0)>0$ for almost all $\lambda>\alpha$, so that we can apply  Proposition~\ref{proptraceflaAtrans}.
  Notice that~\eqref{eqnMnlargex} implies that $T_n(\xi,x)=1$ for almost all $\xi\in\R$.
  Moreover, because 
  \begin{align*}
    a_n(k,x) & = (1+2\sqrt{\alpha}x)^{\frac{\I k}{2\sqrt{\alpha}}+\frac{1}{2}} \frac{z}{2\I k\psi_n(k,x)}  \biggl(m_n(z)-\frac{1}{\sqrt{\alpha}+\I k}\biggr),  
  \end{align*}
  the set of poles of $a_n(\ledot,x)$ is empty (remember properties~\ref{itmTFzp1a},~\ref{itmTFzp1b} and~\ref{itmTFzp1c} in the first part of the proof of Theorem~\ref{thm:traceflaA}) and  
  \begin{align*}
    {\rm Z}_n = \{\kappa>0\,|\, a_n(\I\kappa,x)=0\} = \{\kappa>0\,|\, (\sqrt{\alpha}-\kappa) m_n(\alpha-\kappa^2) = 1\}. 
  \end{align*}
We thus obtain from Theorem~\ref{thm:traceflaA} and Proposition~\ref{proptraceflaAtrans} the identity 
\begin{align*}
      &  \int_0^n (1+2\sqrt{\alpha}s)\biggl(\Wr(s)-\frac{s}{1+2\sqrt{\alpha}s}\biggr)^2ds +  \int_{[0,n)} (1+2\sqrt{\alpha}s)d\dip(s)  \\
  & \qquad\qquad\qquad\qquad = - \frac{1}{\pi} \int_{\R} \frac{\xi^2 \log T_n(\xi,0)}{(\xi^2+\alpha)^3}d\xi  +  \frac{1}{4\alpha^{\nicefrac{3}{2}}}  \sum_{\kappa\in{\rm Z}_n} \cF_2\biggl(\frac{\kappa}{\sqrt{\alpha}}\biggr).
\end{align*}
If we similarly define ${\rm Z}' = \{\kappa>0\,|\, (\sqrt{\alpha}-\kappa) m(\alpha-\kappa^2)=1\}$, then one has 
\begin{align*}
  \liminf_{n\rightarrow\infty} \sum_{\kappa\in{\rm Z}_n} \cF_2\biggl(\frac{\kappa}{\sqrt{\alpha}}\biggr) \geq \sum_{\kappa\in{\rm Z}'} \cF_2\biggl(\frac{\kappa}{\sqrt{\alpha}}\biggr).
\end{align*}
In fact, because the Weyl--Titchmarsh functions $m_n$ converge to $m$ by Proposition~\ref{propSMPcont}, every neighborhood of a $\kappa\in{\rm Z}'$ contains a point in ${\rm Z}_n$ for large enough $n$.
As a consequence, we get that for every finite subset $J\subseteq{\rm Z}'$ one has 
\begin{align*}
     \liminf_{n\rightarrow\infty} \sum_{\kappa\in{\rm Z}_n} \cF_2\biggl(\frac{\kappa}{\sqrt{\alpha}}\biggr) \geq \sum_{\kappa\in J} \cF_2\biggl(\frac{\kappa}{\sqrt{\alpha}}\biggr), 
\end{align*}
which gives the claimed bound. 
Together with the lower semi-continuity from Theorem~\ref{thm:entropyA}, this yields 
\begin{align}\begin{split}\label{eqnTFalphageq}
 &  \int_0^\infty (1+2\sqrt{\alpha}s)\biggl(\Wr(s)-\frac{s}{1+2\sqrt{\alpha}s}\biggr)^2ds +  \int_{[0,\infty)} (1+2\sqrt{\alpha}s)d\dip(s)  \\
  & \qquad\qquad\qquad\qquad\qquad \geq - \frac{1}{\pi} \int_{\R} \frac{\xi^2 \log T(\xi,0)}{(\xi^2+\alpha)^3}d\xi +  \frac{1}{4\alpha^{\nicefrac{3}{2}}} \sum_{\kappa\in{\rm Z}'} \cF_2\biggl(\frac{\kappa}{\sqrt{\alpha}}\biggr).
\end{split}\end{align}
It remains to see that this bound implies conditions~\ref{itmKSc2} and~\ref{itmKSc3}. 
To prove the former, let ${\rm E}=\{\eta>0\,|\, m(\alpha-\eta^2)=\infty\}$, so that the set ${\rm E}\cup\{\sqrt{\alpha}\}$ interlaces the set ${\rm Z}'$ because $m$ is a Herglotz--Nevanlinna function.
Due to monotonicity of the function $\cF_2$ (see~\eqref{eq:F2monot} in Appendix~\ref{app:F1F2}), this property and the bound on the sum in~\eqref{eqnTFalphageq} imply that 
\begin{align}\label{eqnLTE}
   \sum_{\eta\in{\rm E}} \cF_2\biggl(\frac{\eta}{\sqrt{\alpha}}\biggr) < \infty.
\end{align}
From the estimates~\eqref{eq:F2est01} and~\eqref{eq:F2est03}
for the function $\cF_2$, we conclude that 
\begin{align}\label{eqnLTproto}
  \sum_{\eta\in{\rm E} \atop \eta<\sqrt{\alpha}} \frac{16\eta^{3}}{3(\alpha-\eta^2)^{\nicefrac{3}{2}}} +  \sum_{\eta\in{\rm E} \atop \eta>\sqrt{\alpha}} \frac{16\alpha^{\nicefrac{3}{2}}}{3(\eta^2-\alpha)^{\nicefrac{3}{2}}} < \infty,
\end{align}
which proves condition~\ref{itmKSc2}. 
Because one has (see Theorem~\ref{th:SimonFactor} for example) 
\begin{align}\label{eqnTMint}
  \int_0^\infty \frac{\xi^2}{(\xi^2+\alpha)^3} \biggl|\log\biggl|m(\xi^2+\alpha+\I 0) - \frac{1}{\sqrt{\alpha}+\I\xi}\biggr| \biggr| d\xi < \infty, 
\end{align}
it follows from the bound on the integral on the right-hand side of~\eqref{eqnTFalphageq} that 
\begin{align}\label{eqnImintT}
 \int_0^\infty \frac{\xi^2}{(\xi^2+\alpha)^3} |\log(\im\, m(\xi^2+\alpha+\I0))| d\xi <\infty,
\end{align}
which yields condition~\ref{itmKSc3} after the transformation $\lambda=\xi^2+\alpha$. 
\end{proof} 

\begin{proof}[Proof of sufficiency]
 Suppose now that a Herglotz--Nevanlinna function $m$ satisfies the three conditions in Theorem~\ref{thm:KSforExB1} together with~\eqref{eqnKSmatzero} and let $(L,\omega,\dip)$ be the corresponding generalized indefinite string.
 Condition~\ref{itmKSc1} entails that the essential spectrum of $(L,\omega,\dip)$ is contained in $[\alpha,\infty)$ and that $L=\infty$ in view of Theorem~\ref{thm:m-at0}, so that we can use the notation of Section~\ref{secSbSsumrule}. 
 Condition~\ref{itmKSc2} implies~\eqref{eqnLTproto}, where ${\rm E}=\{\eta>0\,|\, m(\alpha-\eta^2)=\infty\}$. 
In view of the asymptotics in~\eqref{eqnF2asymzero} and~\eqref{eqnF2asyminfty} for the function $\cF_2$ we get~\eqref{eqnLTE}. 
The set ${\rm Z}'=\{\kappa>0\,|\, (\sqrt{\alpha}-\kappa) m(\alpha-\kappa^2)=1\}$ interlaces the set ${\rm E}\cup\{\sqrt{\alpha}\}$ because $m$ is a Herglotz--Nevanlinna function, which implies that 
\begin{align*}
   \sum_{\kappa\in{\rm Z}'} \cF_2\biggl(\frac{\kappa}{\sqrt{\alpha}}\biggr)  < \infty.
\end{align*} 
 Condition~\ref{itmKSc3} transforms into~\eqref{eqnImintT}, which gives 
  \begin{align*}
  - \int_\R \frac{\xi^2 \log T(\xi,0)}{(\xi^2+\alpha)^3} d\xi < \infty
 \end{align*}
 because of~\eqref{eqnTMint}. 
In particular, we see that conditions~\ref{itmKSc1},~\ref{itmKSc3} and~\eqref{eqnKSmatzero} guarantee that the requirements of Theorem~\ref{thm:traceflaA} and Proposition~\ref{proptraceflaAtrans} are satisfied. 
From the former, we get that for each $x\geq0$ one has the identity 
  \begin{align*}
  &   \int_0^x (1+2\sqrt{\alpha}s)\biggl(\Wr(s)-\frac{s}{1+2\sqrt{\alpha}s}\biggr)^2ds +  \int_{[0,x)} (1+2\sqrt{\alpha}s)d\dip(s)  \\
&  \qquad =   \frac{2}{\pi} \int_\R \frac{\xi^2 \log|a(\xi,x)|}{(\xi^2+\alpha)^3}d\xi  +  \frac{1}{4\alpha^{\nicefrac{3}{2}}} \lim_{\delta\downarrow0}  \sum_{\kappa\in{\rm Z}\atop\delta<\kappa<1/\delta} \cF_2\biggl(\frac{\kappa}{\sqrt{\alpha}}\biggr) - \sum_{\eta\in{\rm P}\atop\delta<\eta<1/\delta} \cF_2\biggl(\frac{\eta}{\sqrt{\alpha}}\biggr),
\end{align*}
where ${\rm Z} = \{\kappa>0\,|\, a(\I\kappa,x)=0\}$ and ${\rm P} = \{\eta>0\,|\, a(\I\eta,x)=\infty\}$.
Estimating the right-hand side by using that $2\log|a(\xi,x)|\leq -\log T(\xi,0)$ for almost all $\xi\in\R$, that ${\rm Z}\subseteq{\rm Z}'$ and that $\cF_2$ is non-negative, before letting $x\rightarrow\infty$, this becomes
  \begin{align}\begin{split}\label{eqnTFalphaleq}
  &   \int_0^\infty (1+2\sqrt{\alpha}s)\biggl(\Wr(s)-\frac{s}{1+2\sqrt{\alpha}s}\biggr)^2ds +  \int_{[0,\infty)} (1+2\sqrt{\alpha}s)d\dip(s)  \\
&  \qquad\qquad\qquad\qquad\qquad \leq  - \frac{1}{\pi} \int_\R \frac{\xi^2 \log T(\xi,0)}{(\xi^2+\alpha)^3} d\xi +  \frac{1}{4\alpha^{\nicefrac{3}{2}}} \sum_{\kappa\in{\rm Z}'} \cF_2\biggl(\frac{\kappa}{\sqrt{\alpha}}\biggr),
\end{split}\end{align}
 which concludes the proof of Theorem~\ref{thm:KSforExB1}. 
\end{proof} 

Notice that the trace formula underlying Theorem~\ref{thm:KSforExB1} is given by~\eqref{eqnTFalphageq} and~\eqref{eqnTFalphaleq}.
One of its crucial properties is that all the terms involved are always non-negative (since the transmission coefficient $T(\ledot,0)$ takes values in $[0,1]$ and the function $F_2$ is positive).  
We are also able to express this trace formula in terms of the spectral measure of the generalized indefinite string. 

\begin{corollary}\label{corTFalpham}
 If $\rho$ is the spectral measure of a generalized indefinite string $(L,\omega,\dip)$ with $L=\infty$ and~\eqref{eqnCondS1} for some constant $c\in\R$, then 
  \begin{align}\begin{split}\label{eqnTFalpham}
  &   \int_0^\infty (1+2\sqrt{\alpha}x)\biggl(\Wr(x)-c-\frac{x}{1+2\sqrt{\alpha}x}\biggr)^2dx +  \int_{(0,\infty)} (1+2\sqrt{\alpha}x)d\dip(x)  \\
&  \qquad\qquad = \int_\R \frac{d\rho(\lambda)}{\lambda^2} - \frac{1}{\pi} \int_{\alpha}^\infty \frac{\sqrt{\lambda-\alpha}}{\lambda^3}d\lambda +  \frac{1}{4\alpha^{\nicefrac{3}{2}}}  \sum_{\lambda\in \supp(\rho)\atop \lambda<\alpha} \cF_2\Biggl(\sqrt{1-\frac{\lambda}{\alpha}}\Biggr) \\
&  \qquad\qquad\qquad\qquad\qquad - \frac{1}{\pi} \int_{\alpha}^{\infty}  \frac{\sqrt{\lambda-\alpha}}{\lambda ^3} \log\biggl(\frac{\pi\lambda}{\sqrt{\lambda-\alpha}} \frac{d\rho_\ac(\lambda)}{d\lambda}\biggr) d\lambda.
\end{split}\end{align}
\end{corollary}

\begin{proof}
 Because $\rho$ is also the spectral measure of the generalized indefinite string $(L,\omega-c\delta_0,\dip)$, where $\delta_0$ is the unit Dirac measure centered at zero, we may assume that $c=0$. 
  Since in this case one has~\eqref{eqnKSmatzero}, we know from the proof of Theorem~\ref{thm:KSforExB1} that the trace formula given by~\eqref{eqnTFalphageq} and~\eqref{eqnTFalphaleq} holds with equality. 
  By applying identity~\eqref{eq:Gtrace02} from Theorem~\ref{thm:FactorMain} to the meromorphic function $b$ on $\C_+$ defined by 
  \begin{align*}
    b(k) = 1 - (\sqrt{\alpha}+\I k)m(k^2+\alpha) = \frac{m(k^2+\alpha) - \frac{1}{\sqrt{\alpha}+\I k}}{-\frac{1}{\sqrt{\alpha}+\I k}}, 
  \end{align*}
  we also obtain the relation  
  \begin{align}\begin{split}\label{eqnTFalphaaux}
   &  4\sqrt{\alpha}\dot{m}(0) + m(0)^2 - \frac{1}{\sqrt{\alpha}} m(0) \\
    & \qquad = \frac{4\sqrt{\alpha}}{\pi} \int_{\R} \frac{\xi^2}{(\xi^2+\alpha)^3}  \log\Biggl((\xi^2+\alpha)\biggl|m(\xi^2+\alpha+\I 0) - \frac{1}{\sqrt{\alpha}+\I |\xi|}\biggr|^2\Biggr) d\xi   \\
    & \qquad\qquad\qquad\qquad +  \frac{1}{\alpha}\lim_{\delta\downarrow0}  \sum_{\kappa\in{\rm Z}\atop\delta<\kappa<1/\delta} \cF_2\biggl(\frac{\kappa}{\sqrt{\alpha}}\biggr) - \sum_{\eta\in{\rm P}\atop\delta<\eta<1/\delta} \cF_2\biggl(\frac{\eta}{\sqrt{\alpha}}\biggr),
   \end{split}\end{align}
  where ${\rm Z} = \{\kappa>0\,|\, b(\I\kappa)=0\}$ and ${\rm P} = \{\eta>0\,|\, b(\I\eta)=\infty\}$.
  The set ${\rm Z}$ actually coincides with the set ${\rm Z}'$ in the trace formula given by~\eqref{eqnTFalphageq} and~\eqref{eqnTFalphaleq}, which in particular entails that we may omit the limit in~\eqref{eqnTFalphaaux} as the sums converge individually. 
  Moreover, the set ${\rm P}$ corresponds to the poles of $m$ below $\alpha$ and thus to the points in $\supp(\rho)$ below $\alpha$. 
  After replacing the transmission coefficient in the trace formula given by~\eqref{eqnTFalphageq} and~\eqref{eqnTFalphaleq} with the expression in~\eqref{eqnTdef}, we can employ the relation in~\eqref{eqnTFalphaaux} and the facts that (for the former, use the integral representation~\eqref{eqnWTmIntRep} and recall that $c_1=\dip(\{0\})$; see~\cite[Lemma~7.1]{IndefiniteString})
  \begin{align*}
    \dot{m}(0) & = \dip(\{0\}) + \int_\R \frac{d\rho(\lambda)}{\lambda^2},  & m(0) & = \frac{1}{2\sqrt{\alpha}}, 
  \end{align*} 
  to arrive at the formula in~\eqref{eqnTFalpham}. 
\end{proof}

We conclude this section with a Lieb--Thirring inequality for generalized indefinite strings that follows readily from the trace formula in Corollary~\ref{corTFalpham}.

\begin{corollary}\label{cor:LT01new}
If $\rho$ is the spectral measure of a generalized indefinite string $(L,\omega,\dip)$ with $L=\infty$ and~\eqref{eqnCondS1} for some constant $c\in\R$, then 
\begin{align}\label{eq:LT01new}
\begin{split}
  & \sum_{\lambda\in \supp(\rho) \atop \lambda<0} \frac{1}{|\lambda|^{\nicefrac{3}{2}}} + \frac{1}{\alpha^3} \sum_{\lambda\in \supp(\rho)\atop 0<\lambda<\alpha}(\alpha-\lambda)^{\nicefrac{3}{2}}  \leq   \frac{3}{4} \int_{(0,\infty)} (1+2\sqrt{\alpha}x)d\dip(x) \\
  & \qquad\qquad\qquad\qquad\qquad + \frac{3}{4} \int_0^\infty (1+2\sqrt{\alpha}x)\biggl(\Wr(x)-c-\frac{x}{1+2\sqrt{\alpha}x}\biggr)^2dx.
     \end{split}  
\end{align}
\end{corollary}

\begin{proof}
Because the sum of the first, the second and the last term on the right-hand side of~\eqref{eqnTFalpham} is non-negative (to see this, it suffices to use that $r-1-\log(r) \geq 0$ whenever $r>0$), we get that 
\begin{align*}
\frac{1}{4\alpha^{\nicefrac{3}{2}}}  \sum_{\lambda\in \supp(\rho)\atop \lambda<\alpha} \cF_2\Biggl(\sqrt{1-\frac{\lambda}{\alpha}}\Biggr) & \leq 
\int_0^\infty (1+2\sqrt{\alpha}x)\biggl(\Wr(x)-c-\frac{x}{1+2\sqrt{\alpha}x}\biggr)^2dx \\
  & \qquad\qquad\qquad\qquad\qquad + \int_{(0,\infty)} (1+2\sqrt{\alpha}x)d\dip(x).
\end{align*}
The claimed inequality then follows by taking~\eqref{eq:F2est01} and~\eqref{eq:F2est03} into account.
\end{proof}

\begin{remark}
  The constants in the Lieb--Thirring inequality~\eqref{eq:LT01new} are sharp.
  For instance, one can see this by considering particular generalized indefinite strings with spectral measure given by 
  \begin{align}
 \rho(B) = \frac{1}{\pi} \int_{B\cap [\alpha,\infty)}\frac{\sqrt{\lambda-\alpha}}{\lambda}d\lambda + \gamma\delta_{\lambda_0}(B),
  \end{align}
 where $\gamma>0$ and $\lambda_0\in(-\infty,0)\cup(0,\alpha)$ are parameters. 
 From the trace formula~\eqref{eqnTFalpham}, we know that the right-hand side of~\eqref{eq:LT01new} is then given by  
\begin{align}
 \frac{3\gamma}{4\lambda_0^2} + \frac{3}{16\alpha^{\nicefrac{3}{2}}}  \cF_2\Biggl(\sqrt{1-\frac{\lambda_0}{\alpha}}\Biggr).
\end{align}
A comparison with the left-hand side of~\eqref{eq:LT01new} 
in the limit as $\gamma\rightarrow0$ and $\lambda_0\rightarrow\alpha$ or $\lambda_0\rightarrow-\infty$, also using the asymptotics of the function $\cF_2$ in~\eqref{eqnF2asymzero} and~\eqref{eqnF2asyminfty}, shows that the constants in the Lieb--Thirring inequality are indeed best possible.  
\end{remark}

 \section{Relative trace formulas, \ref*{thm:KSforExB2}}\label{secSbSsumrule2}

 In order to start the proof of Theorem~\ref{thm:KSforExB2}, fix $\beta>0$ and consider the rational function $\zeta$ defined by 
 \begin{align}
   \zeta(k) & = \beta \frac{1+k^2}{1-k^2} =  \beta \frac{1-|k|^4}{|1-k^2|^2} + 4\I \beta \frac{\re\,k\,\im\,k}{|1-k^2|^2}. 
 \end{align}
 Its decisive property is that $\zeta$ maps the upper complex half-plane $\C_+$ conformally onto $\C_+\cup(-\beta,\beta)\cup\C_-$. 
 More precisely, it maps the first complex quadrant onto the upper complex half-plane $\C_+$, the second complex quadrant onto the lower complex half-plane $\C_-$ and the positive imaginary axis onto the interval $(-\beta,\beta)$.
 Furthermore, we note that $\zeta(\I)=0$ and that the derivatives of $\zeta$ are given by 
 \begin{align}
   \dot{\zeta}(k) & = 4 \beta   \frac{k}{(1-k^2)^2}, & \ddot{\zeta}(k) & = 4 \beta  \frac{1+3 k^2}{(1-k^2)^3}.
 \end{align}
 
For a generalized indefinite string $(L,\omega,\dip)$ with $L=\infty$ and essential spectrum contained in $(-\infty,-\beta]\cup[\beta,\infty)$, the associated Weyl--Titchmarsh function $m$ has a meromorphic extension to $\C_+\cup(-\beta,\beta)\cup\C_-$ that is analytic at zero. 
We consider the Weyl solution $\psi(k,\redot)$ of the differential equation~\eqref{eqnDEho} with $z=\zeta(k)$ defined by 
 \begin{align}\label{eq:WeylSolB}
    \psi(k,x) &= \theta(z,x) + zm(z)\phi(z,x) 
 \end{align}
 for all $k\in\C_+$ such that $z$ is not a pole of $m$.  
 With these solutions, we are able to define the {\em relative Wronskian} $a(\ledot,x)$ for each $x\geq0$ by 
 \begin{align}\label{def:a-functionB}
   a(k,x) = \frac{W(k,0)}{W(k,x)}, 
 \end{align}
 where the functions $W(\ledot,x)$ are given by 
 \begin{align}
   W(k,x) =  (1+2\beta x)^{-\frac{\I k}{1-k^2}+\frac{1}{2}}\biggl(\frac{1}{1+2\beta x} \frac{1-\I k}{1+\I k}\zeta(k)\psi(k,x)-\psi^\qd(k,x)\biggr).
 \end{align}
Because the functions $\psi(\ledot,x)$ and $\psi^\qd(\ledot,x)$ are meromorphic on $\C_+$, so are the functions $W(\ledot,x)$ and $a(\ledot,x)$  with $W(\I,x)=2\beta$ and $a(\I,x)=1$. 

 \begin{proposition}\label{prop:a-decompB}
   Let $\beta>0$ and let $(L,\omega,\dip)$ be a generalized indefinite string with essential spectrum contained in $(-\infty,-\beta]\cup[\beta,\infty)$ and $m(0)=0$.
  Then the relative Wronskian $a(\ledot,x)$ satisfies   
 \begin{align}\label{eq:a-decomp02B}   
 \dot{a}(\I,x) = \I\beta \int_0^x \Wr(s) ds 
  \end{align}
   for every $x\geq0$, as well as the identity
  \begin{align}\label{eq:a-decomp03B}
     \begin{split} 
  &  \ddot{a}(\I,x) - \dot{a}(\I,x)^2 - \I \dot{a}(\I,x) \\
 & \qquad =  \beta \int_0^x (1+2\beta s)\Wr(s)^2 ds + \beta \int_{[0,x)} (1+2\beta s)d\dip(s) - \beta\int_0^x \frac{1}{1+2\beta s}ds. 
 \end{split}
 \end{align} 
 \end{proposition}
 
 \begin{proof}
  The claims follow again from calculations using the definition of $a$ and the expressions in Proposition~\ref{propFS}.  
 To begin with, we compute 
 \begin{align*}
   \dot{\psi}(\I,x) &  = - \I\beta \int_0^x \Wr(s) ds, \\
   \ddot{\psi}(\I,x) & =  \beta\int_0^x \Wr(s)ds - \beta^2\biggl(\int_0^x \Wr(s)ds\biggr)^2  - 2\beta^2\dot{m}(0)x  \\
   & \qquad\qquad\qquad\qquad + 2\beta^2 \int_0^x \int_0^s \Wr(t)^2 dt\,ds + 2\beta^2 \int_0^x \int_{[0,s)} d\dip\,ds.
 \end{align*}
 Furthermore, for the quasi-derivative of $\psi$, one gets $\dot{\psi}^\qd(\I,x)=0$ and 
 \begin{align*}
   \ddot{\psi}^\qd(\I,x) & =  2\beta^2\int_0^x \Wr(s)^2ds + 2\beta^2\int_{[0,x)}d\dip - 2\beta^2\dot{m}(0).
 \end{align*}
We next compute the derivative for the function $W$ and get 
 \begin{align*}
   \dot{W}(\I,x) & =  - 2\I\beta^2 \int_0^x \Wr(s)ds,
       \end{align*}
 which immediately gives~\eqref{eq:a-decomp02B}. 
 For the second derivative, one has  
 \begin{align*}
  &  \frac{\ddot{W}(\I,x)}{W(\I,x)} - \frac{\dot{W}(\I,x)^2}{W(\I,x)^2} -\I \frac{\dot{W}(\I,x)}{W(\I,x)} \\
  & \qquad = -\beta \int_0^x (1+2\beta s)\Wr(s)^2 ds  - \beta \int_{[0,x)} (1+2\beta s)d\dip(s)\\
     & \qquad\qquad\qquad\qquad\qquad\qquad\qquad +\frac{1}{2}+\frac{\log(1+2\beta s)}{2} +\beta\dot{m}(0),
\end{align*}
 which readily yields the remaining identity~\eqref{eq:a-decomp03B}.
  \end{proof}
   
 Before we state the relative trace formulas in this case, remember that we denote with $\varrho$ the (positive) square root of the Radon--Nikod\'ym derivative of $\dip$ with respect to the Lebesgue measure and with $\dip_\sing$ the singular part of $\dip$.
 We also continue to use the two functions $\cF_1$ and $\cF_2$ as defined by~\eqref{eq:Fsdef} in Section~\ref{secSbSsumrule}. 

 \begin{theorem}[Relative trace formulas]\label{thm:traceflaB}
Let $\beta>0$ and let $(L,\omega,\dip)$ be a generalized indefinite string with essential spectrum contained in $(-\infty,-\beta]\cup[\beta,\infty)$ and $m(0)=0$. 
Then for every $x\geq0$, the limit $a(\xi,x) = \lim_{\varepsilon\downarrow 0}a(\xi+\I\varepsilon,x)$ exists and is nonzero for almost all $\xi \in\R$, satisfies
  \begin{align}\label{eq:AlogSzegoB}
      \int_{\R} \frac{|\log|a(\xi,x)||}{1+\xi^2}d\xi <\infty
   \end{align} 
 and the identities 
   \begin{align}\label{eq:traceflaB00}
      \begin{split}
     &   \int_0^x \biggl(\frac{1}{1+2\beta s} -\varrho(s)\biggr) ds  \\
     & \qquad =   \frac{1}{\beta\pi} \int_{\R} \frac{\log|a(\xi,x)|}{1+\xi^2}d\xi +  \frac{1}{\beta}\lim_{\delta\downarrow0} \sum_{\kappa\in{\rm Z}\atop\delta<\kappa<1/\delta} \log\biggl|\frac{\kappa-1}{\kappa+1}\biggr| - \sum_{\eta\in{\rm P}\atop\delta<\eta<1/\delta} \log\biggl|\frac{\eta-1}{\eta+1}\biggr|,
     \end{split}
\\ \label{eq:traceflaB01}
      \begin{split}
     &  \int_0^x \Wr(s) ds  \\
     & \qquad =   \frac{1}{\beta\pi} \int_{\R} \frac{(1-\xi^2)\log|a(\xi,x)|}{(1+\xi^2)^2}d\xi + \frac{1}{\beta}\lim_{\delta\downarrow0} \sum_{\kappa\in{\rm Z}\atop\delta<\kappa<1/\delta} \frac{2\kappa}{\kappa^2-1} - \sum_{\eta\in{\rm P}\atop\delta<\eta<1/\delta} \frac{2\eta}{\eta^2-1},
     \end{split}
\\ \label{eq:traceflaB02}
  \begin{split}
  &  \int_0^x (1+2\beta s)\Wr(s)^2 ds  + \int_{0}^{x} (1+2\beta s)\biggl(\varrho(s)-\frac{1}{1+2\beta s}\biggr)^2 ds +  \int_{[0,x)} (1+2\beta s)d\dip_{\sing}(s)  \\
  & \qquad = \frac{8}{\beta\pi} \int_{\R} \frac{\xi^2 \log|a(\xi,x)|}{(1+\xi^2)^3}d\xi  +  \frac{1}{\beta}\lim_{\delta\downarrow0}  \sum_{\kappa\in{\rm Z}\atop\delta<\kappa<1/\delta} \cF_2(\kappa) - \sum_{\eta\in{\rm P}\atop\delta<\eta<1/\delta} \cF_2(\eta),
\end{split}
\end{align}
  hold, where ${\rm Z} = \{\kappa>0\,|\, a(\I\kappa,x)=0\}$ and ${\rm P} = \{\eta>0\,|\, a(\I\eta,x)=\infty\}$.
 \end{theorem}
 
 \begin{proof} 
  We first note that the function $a(\ledot,x)$ can be written as 
   \begin{align*} 
     a(k,x) & =  (1+2\beta x)^{\frac{\I k}{1-k^2}-\frac{1}{2}} \frac{G(k,0)}{G(k,x)} \frac{1}{\psi(k,x)},
   \end{align*}
  where the meromorphic functions $G(\ledot,x)$ on $\C_+$ are given by 
  \begin{align*}
     G(k,x) & = \frac{\psi^\qd(k,x)}{\zeta(k)\psi(k,x)} - \frac{1}{1+2\beta x} \frac{1-\I k}{1+\I k}. 
  \end{align*}
    It follows readily that $\im\,G(k,x)>0$ when $\re\,k>0$ and $\im\,G(k,x)< 0$ when $\re\,k<0$ because the first term is a Herglotz--Nevanlinna function in $z=\zeta(k)$. 
   As a consequence, all zeros and poles of $G(\ledot,x)$, $G(\ledot,0)$ and $\psi(\ledot,x)$ are simple and lie on the imaginary axis. 
  The fact that the zeros and poles of $a(\ledot,x)$ are simple as well is a result of the following three observations:
  \begin{enumerate}[label=(\alph*), ref=(\alph*), leftmargin=*, widest=a]
  \item\label{itmTFzp1Ba} If $k$ is a zero of $\psi(\ledot,x)$, then $k$ is a pole of $G(\ledot,x)$.
  \item\label{itmTFzp1Bb} If $k$ is a pole of $\psi(\ledot,x)$, then $k$ is a pole of $G(\ledot,0)$. 
  \item\label{itmTFzp1Bc} A $k$ that is neither a zero nor a pole of $\psi(\ledot,x)$ is a pole of $G(\ledot,x)$ if and only if it is a pole of $G(\ledot,0)$. 
  \end{enumerate}
   For $x>0$ (the claims are trivial otherwise), we next write the function $a(\ledot,x)$ as 
     \begin{align*}
     a(k,x) & = \frac{\phi(\zeta(k),x)}{\phi_\beta(\zeta(k),x)} \frac{G(k,0)}{G(k,x)}\frac{G_{+}(k)}{G_{-}(k)},
   \end{align*}
    where $\phi_\beta$ is the respective solution for the generalized indefinite string in Example~\ref{exa2alpha}, given explicitly by
    \begin{align*}
      \phi_\beta(z,x) = \frac{(1+2\beta x)^{\frac{1}{2}}}{\sqrt{z^2-\beta^2}} \sin\Biggl(\frac{\sqrt{z^2-\beta^2}}{2\beta}\log(1+2\beta x)\Biggr),
    \end{align*}
    and the meromorphic functions $G_{\pm}$ on $\C_+$ are given by 
  \begin{align*}
     G_{-}(k) & = \zeta(k)\phi(\zeta(k),x)\psi(k,x) = \biggl(\frac{\phi^\qd(\zeta(k),x)}{\zeta(k)\phi(\zeta(k),x)} - \frac{\psi^\qd(k,x)}{\zeta(k)\psi(k,x)}\biggr)^{-1}, \\
     G_{+}(k) & = \zeta(k)\phi_\beta(\zeta(k),x)(1+2\beta x)^{\frac{\I k}{1-k^2}-\frac{1}{2}} \\
      & = \biggl(\frac{2k}{1+k^2}\cot\biggl(\frac{k}{1-k^2}\log(1+2\beta x)\biggr) - \frac{2\I k}{1+k^2}\biggr)^{-1}.
  \end{align*}
   Since one has that $\im\,G_\pm(k)\geq0$ when $\re\,k>0$ and $\im\,G_\pm(k)\leq 0$ when $\re\,k<0$, the claims readily follow from Theorem~\ref{thm:FactorMainB}, the formula for the exponential type of solutions in~\eqref{eqnthetaphiExpTyp} and Proposition~\ref{prop:a-decompB}.
  \end{proof}
   
\begin{remark}
Adding up~\eqref{eq:traceflaB00} and~\eqref{eq:traceflaB01}, we end up with another relative trace formula involving the function $\cF_1$. 
More precisely, under the conditions of Theorem~\ref{thm:traceflaB}, one has 
   \begin{align}\label{eq:traceflaB03}
      \begin{split}
     &  \int_0^x \Wr(s) ds + \int_0^x \biggl(\frac{1}{1+2\beta s} -\varrho(s)\biggr) ds  \\
     & \qquad =   \frac{2}{\beta\pi} \int_{\R} \frac{\log|a(\xi,x)|}{(1+\xi^2)^2}d\xi +  \frac{1}{\beta}\lim_{\delta\downarrow0} \sum_{\kappa\in{\rm Z}\atop\delta<\kappa<1/\delta} \cF_1(\kappa) - \sum_{\eta\in{\rm P}\atop\delta<\eta<1/\delta} \cF_1(\eta).
     \end{split}
     \end{align}
\end{remark}   
   
   The integrals on the spectral side of the relative trace formulas can again be expressed in terms of suitable {\em transmission coefficients}. 
    To this end, we first define the meromorphic functions $M(\ledot,x)$ on $\C_+$ by 
    \begin{align}
       M(k,x) = \frac{\psi^\qd(k,x)}{\zeta(k)\psi(k,x)}
    \end{align}
    for $x\geq0$, so that $M(k,0)= m(\zeta(k))$. 
    The limit $M(\xi,x)=\lim_{\varepsilon\downarrow0}M(\xi +\I\varepsilon,x)$ exists for almost all $\xi\in\R$ (see Theorem~\ref{th:SimonFactor} for example) and we may define 
    \begin{align}\label{eq:defTxB}
       T(\xi,x) =  \frac{8\xi}{1+\xi^2}  \frac{(1+2\beta x) \im\, M(\xi,x)}{\Bigl|(1+2\beta x) M(\xi,x) -\frac{1-\I\xi}{1+\I \xi}\Bigr|^2}
    \end{align}
    as the denominator is positive for almost all $\xi\in\R$.
    In fact, one has 
    \begin{align}\label{eqnTransEstB}\begin{split}
      \biggl|(1+2\beta x) M(\xi,x) -\frac{1-\I\xi}{1+\I \xi}\biggr|^2 & \geq \biggl((1+2\beta x) \im\,M(\xi,x) +\frac{2\xi}{1+\xi^2}\biggr)^2 \\
      & \geq \frac{8\xi}{1+\xi^2}(1+2\beta x) \im\,M(\xi,x),
    \end{split}\end{align}
    which also entails that $T(\xi,x)\in[0,1]$. 
    
    \begin{proposition}\label{proptraceflaBtrans}
Let $\beta>0$ and let $(L,\omega,\dip)$ be a generalized indefinite string with essential spectrum contained in $(-\infty,-\beta]\cup[\beta,\infty)$ and $m(0)=0$.
  If $\im\,m(\lambda+\I0)>0$ for almost all $|\lambda|>\beta$, then 
 \begin{align}\label{eqnaasTB}
   |a(\xi,x)|^2 = \frac{T(\xi,x)}{T(\xi,0)}
 \end{align}
 for almost all $\xi\in\R$ and all $x\geq 0$. 
 \end{proposition}

\begin{proof}
   Let $\xi\in(0,1)\cup(1,\infty)$ be such that $\im\,m(\lambda+\I0)>0$, where $\lambda=\zeta(\xi)$. 
   We consider the solution $\psi(\xi,\redot)$ of the differential equation~\eqref{eqnDEho} with $z=\lambda$ such that 
   \begin{align*}
     \psi(\xi,x) & = \theta(\lambda,x)+\lambda m(\lambda+\I0) \phi(\lambda,x) = \lim_{\varepsilon\downarrow0} \psi(\xi+\I\varepsilon,x), \\
     \psi^\qd(\xi,x) & = \theta^\qd(\lambda,x) + \lambda m(\lambda+\I0)\phi^\qd(\lambda,x) = \lim_{\varepsilon\downarrow0} \psi^\qd(\xi+\I\varepsilon,x). 
   \end{align*} 
   Notice that $\psi(\xi,x)\not=0$ for all $x\geq0$ because it is not possible for $\theta(\lambda,x)$ and $\phi(\lambda,x)$ to vanish both. 
   It follows that the limit 
   \begin{align*}
     M(\xi,x) = \lim_{\varepsilon\downarrow0} M(\xi+\I\varepsilon,x) = \frac{\psi^\qd(\xi,x)}{\lambda\psi(\xi,x)}
   \end{align*}
   exists for all $x\geq0$ and satisfies 
   \begin{align*}
     \im\,M(\xi,x) & = \frac{\im\,\psi(\xi,x)^\ast \psi^\qd(\xi,x)}{\lambda|\psi(\xi,x)|^2} = \frac{\im\,m(\lambda+\I0)}{|\psi(\xi,x)|^2} > 0,
   \end{align*}
   where we used that $\theta(\lambda,x)\phi^\qd(\lambda,x)-\theta^\qd(\lambda,x)\phi(\lambda,x)=1$. 
   Next, one gets 
   \begin{align*}
     \Bigl|\lim_{\varepsilon\downarrow0} W(\xi+\I\varepsilon,x)\Bigr|^2 & = \frac{\lambda^2 \im\,m(\lambda+\I0)}{(1+2\beta x)\im\,M(\xi,x)} \biggl| (1+2\beta x)M(\xi,x) - \frac{1-\I\xi}{1+\I \xi} \biggr|^2, 
   \end{align*}
   which is positive in view of~\eqref{eqnTransEstB}.
   We conclude that the limit $\lim_{\varepsilon\downarrow0} a(\xi+\I\varepsilon,x)$ exists as well and satisfies~\eqref{eqnaasTB} as claimed. 
   Similar arguments show that this also holds true for all $\xi\in(-\infty,-1)\cup(-1,0)$ with $\im\,m(\zeta(\xi)+\I0)>0$.
\end{proof}

\section{Lower semi-continuity, \ref*{thm:KSforExB2}} \label{secLowSemiContB}

Let $(L,\omega,\dip)$ be a generalized indefinite string and $m$ be the corresponding Weyl--Titchmarsh function. 
For $\beta>0$ and the function $\zeta$ given as in Section~\ref{secSbSsumrule2}, we define the {\em transmission coefficient} $T$ by
\begin{align}\label{eqnT2def}
 T(\xi) = \frac{8|\xi|}{1+\xi^2} \frac{\im\, m(\zeta(\xi)+\I0)}{\Bigl|m(\zeta(\xi)+\I0)-\frac{1-\I |\xi|}{1+\I |\xi|}\Bigr|^2},
\end{align}
which is well-defined for almost all $\xi\in\R$ and belongs to $[0,1]$ because 
\begin{align}\begin{split}
   \biggl|m(\zeta(\xi)+\I0)-\frac{1-\I |\xi|}{1+\I |\xi|}\biggr|^2 & \geq  \biggl(\im\,m(\zeta(\xi)+\I0)+\frac{2|\xi|}{1+\xi^2}\biggr)^2 \\
     & \geq \frac{8|\xi|}{1+\xi^2} \im\,m(\zeta(\xi)+\I0).
\end{split}\end{align} 
Notice that this function $T$ coincides with $T(\ledot,0)$ as given by~\eqref{eq:defTxB} almost everywhere.
Just like in Section~\ref{secLowSemiContA}, we are next going to establish that the quantity $\cQ_\beta\in[0,\infty]$ defined by 
\begin{align}
\cQ_\beta = -\int_\R \frac{\xi^2 \log T(\xi)}{(1+\xi^2)^3} d\xi
\end{align}
depends lower semi-continuously on the generalized indefinite string $(L,\omega,\dip)$. 

\begin{theorem}\label{thm:entropyB}
  The functional $(L,\omega,\dip)\mapsto\cQ_\beta$ is lower semi-continuous on the set of all generalized indefinite strings for every $\beta>0$.
\end{theorem}

\begin{proof}
For a generalized indefinite string $(L,\omega,\dip)$, we introduce the analog of the {\em reflection coefficient} $r$ defined on the first complex quadrant by
\begin{align*}
r(k) = \frac{m(\zeta(k)) - \frac{1+\I k}{1-\I k}}{m(\zeta(k)) - \frac{1-\I k}{1+\I k}}.
\end{align*}
Notice that this function is analytic and obeys the bound 
\begin{align*}
  |r(k)| = \Biggl|\frac{m(\zeta(k)) +1 - \frac{2}{1-\I k}}{m(\zeta(k)) +1 - \frac{2}{1+\I k}}\Biggr| \leq 1 + \frac{2|k|}{\re\, k}
\end{align*}
in view of  Lemma~\ref{lem:Est-r}.
The limit $r(\xi)=\lim_{\varepsilon\downarrow0} r(\xi+\I\varepsilon)$ exists and satisfies the {\em scattering relation}
\begin{align*}
T(\xi) + |r(\xi)|^2 = 1
\end{align*}
for almost every $\xi>0$, which can be seen by first computing that 
\begin{align*}
 &  \biggl|m(\zeta(\xi)+\I 0)-\frac{1-\I \xi}{1+\I\xi}\biggr|^2 - \biggl|m(\zeta(\xi)+\I 0)-\frac{1+\I\xi}{1-\I\xi}\biggr|^2  = \frac{8\xi}{1+\xi^2} \im\,m(\zeta(\xi)+\I0)
\end{align*}
and then dividing by the first term on the left-hand side. 
Now the claim follows exactly in the same way as in the proof of Theorem~\ref{thm:entropyA}.
\end{proof}

\section{Proof of Theorem~\ref*{thm:KSforExB2}} \label{secKSB2proof}

We will again prove a slightly amended statement, from which Theorem~\ref{thm:KSforExB2} readily follows:
{\em   A Herglotz--Nevanlinna function $m$ is the Weyl--Titchmarsh function of a generalized indefinite string $(L,\omega,\dip)$ with $L=\infty$ and~\eqref{eqnCondS2} with $c=0$ if and only if the conditions~\ref{itmKS2c1},~\ref{itmKS2c2},~\ref{itmKS2c3} in Theorem~\ref{thm:KSforExB2} and $m(0)=0$ hold}.

\begin{proof}[Proof of necessity] 
Suppose first that $m$ is the Weyl--Titchmarsh function of a generalized indefinite string $(L,\omega,\dip)$ with $L=\infty$ that satisfies~\eqref{eqnCondS2} with $c=0$.
It follows readily from Theorem~\ref{thm:HSforAB}~\ref{itmHSforB} that $m$ has a meromorphic extension to $\C_+\cup(-\beta,\beta)\cup\C_-$ that is analytic at zero and from Corollary~\ref{cor:m-at0}~\ref{iCor:m-at0forB} that $m(0)=0$ holds.
In order to verify the remaining conditions~\ref{itmKS2c2} and~\ref{itmKS2c3}, let us consider the sequence of generalized indefinite strings $(L,\omega_n,\dip_n)$ whose coefficients are given by 
\begin{align*}
   \Wr_n(x) & = \begin{cases} \Wr(x), & x<n, \\  0, & x\geq n, \end{cases}  & \dip_n(B) & = \dip(B\cap[0,n)) + \int_{B\cap[n,\infty)} \frac{1}{(1+2\beta x)^2}dx,
\end{align*} 
so that $(L,\omega_n,\dip_n)$ converges to $(L,\omega,\dip)$ in the sense of Proposition~\ref{propSMPcont}. 
We will use the notation from Section~\ref{sec:prelim} and Section~\ref{secSbSsumrule2}, where all quantities corresponding to the generalized indefinite strings $(L,\omega_n,\dip_n)$ will be denoted with an additional subscript. 
Theorem~\ref{thm:HSforAB}~\ref{itmHSforB} and Corollary~\ref{cor:m-at0}~\ref{iCor:m-at0forB} guarantee that each $(L,\omega_n,\dip_n)$ satisfies the requirements of Theorem~\ref{thm:traceflaB}. 
For large enough $x>n$, the solutions $\psi_n(k,\redot)$ are scalar multiples of the ones in Example~\ref{exa2alpha} and thus 
 \begin{align}\label{eqnMnlargexB}
   (1+2\beta x) M_n(k,x) & = (1+2\beta x)\frac{\psi_n^\qd(k,x)}{\zeta(k)\psi_n(k,x)}= \frac{1+\I k}{1-\I k}.
 \end{align}  
 Because we may write 
 \begin{align*}
   M_n(k,0) = \frac{\theta_n(\zeta(k),x)M_n(k,x)-\theta_n^\qd(\zeta(k),x)/\zeta(k)}{\phi_n^\qd(\zeta(k),x)-\zeta(k)\phi_n(\zeta(k),x)M_n(k,x)},
 \end{align*}
 we see that $M_n(\xi,0)=\lim_{\varepsilon\downarrow0}M_n(\xi+\I\varepsilon,0)$ exists for all $\xi\in\R\backslash\{-1,0,1\}$ and 
  \begin{align*}
   \im\,M_n(\xi,0)  = \frac{\im\,M_n(\xi,x)}{\bigl|\phi_n^\qd(\zeta(\xi),x)-\zeta(\xi)\phi_n(\zeta(\xi),x)M_n(\xi,x)\bigr|^2}.
 \end{align*}
  Since $M_n(k,0)=m_n(\zeta(k))$, this implies that $\im\,m_n(\lambda+\I0)>0$ for almost all $|\lambda|>\beta$, so that we can apply  Proposition~\ref{proptraceflaBtrans}.
  Notice that~\eqref{eqnMnlargexB} implies that $T_n(\xi,x)=1$ for almost all $\xi\in\R$.
  Moreover, because 
  \begin{align*}
    a_n(k,x) & = (1+2\beta x)^{\frac{\I k}{1-k^2}+\frac{1}{2}} \frac{1+k^2}{4\I k\psi_n(k,x)}  \biggl(m_n(\zeta(k))-\frac{1-\I k}{1+\I k}\biggr),  
  \end{align*}
  the set of poles of $a_n(\ledot,x)$ is empty (remember properties~\ref{itmTFzp1Ba},~\ref{itmTFzp1Bb} and~\ref{itmTFzp1Bc} in the first part of the proof of Theorem~\ref{thm:traceflaB}) and  
  \begin{align*}
    {\rm Z}_n = \{\kappa>0\,|\, a_n(\I\kappa,x)=0\} = \{\kappa>0\,|\, (1-\kappa) m_n(\zeta(\I\kappa)) = 1 + \kappa\}. 
  \end{align*}
We thus obtain from Theorem~\ref{thm:traceflaB} and Proposition~\ref{proptraceflaBtrans} the identity 
\begin{align*}
  & \int_0^n (1+2\beta s)\Wr(s)^2 ds + \int_{0}^{n} (1+2\beta s)\biggl(\varrho(s)-\frac{1}{1+2\beta s}\biggr)^2 ds +  \int_{[0,n)} (1+2\beta s)d\dip_{\sing}(s)  \\
  & \qquad = - \frac{4}{\beta\pi} \int_{\R} \frac{\xi^2 \log T_n(\xi,0)}{(1+\xi^2)^3}d\xi  +  \frac{1}{\beta}  \sum_{\kappa\in{\rm Z}_n} \cF_2(\kappa).
\end{align*}
If we similarly define ${\rm Z}' = \{\kappa>0\,|\, (1-\kappa) m(\zeta(\I\kappa))=1+\kappa \}$, then one has 
\begin{align*}
  \liminf_{n\rightarrow\infty} \sum_{\kappa\in{\rm Z}_n} \cF_2(\kappa) \geq \sum_{\kappa\in{\rm Z}'} \cF_2(\kappa).
\end{align*}
In fact, because the Weyl--Titchmarsh functions $m_n$ converge to $m$ by Proposition~\ref{propSMPcont}, every neighborhood of a $\kappa\in{\rm Z}'$ contains a point in ${\rm Z}_n$ for large enough $n$.
As a consequence, we get that for every finite subset $J\subseteq{\rm Z}'$ one has 
\begin{align*}
     \liminf_{n\rightarrow\infty} \sum_{\kappa\in{\rm Z}_n} \cF_2(\kappa) \geq \sum_{\kappa\in J} \cF_2(\kappa), 
\end{align*}
which gives the claimed bound. 
Together with the lower semi-continuity from Theorem~\ref{thm:entropyB}, this yields 
\begin{align}\begin{split}\label{eqnTFbetageq}
 & \int_0^\infty (1+2\beta s)\Wr(s)^2 ds \\
 & \qquad + \int_{0}^{\infty} (1+2\beta s)\biggl(\varrho(s)-\frac{1}{1+2\beta s}\biggr)^2 ds +  \int_{[0,\infty)} (1+2\beta s)d\dip_{\sing}(s)  \\
  & \qquad \geq - \frac{4}{\beta\pi} \int_{\R} \frac{\xi^2 \log T(\xi,0)}{(1+\xi^2)^3}d\xi +  \frac{1}{\beta} \sum_{\kappa\in{\rm Z}'} \cF_2(\kappa).
\end{split}\end{align}
It remains to see that this bound implies conditions~\ref{itmKS2c2} and~\ref{itmKS2c3}. 
To prove the former, let ${\rm E}=\{\eta>0\,|\, m(\zeta(\I\eta))=\infty\}$, so that the set ${\rm E}\cup\{1\}$ interlaces the set ${\rm Z}'$ because $m$ is a Herglotz--Nevanlinna function.
Due to monotonicity of the function $\cF_2$, this property and the bound on the sum in~\eqref{eqnTFbetageq} imply that 
\begin{align}\label{eqnLTEB}
   \sum_{\eta\in{\rm E}} \cF_2(\eta) < \infty.
\end{align}
Since the asymptotics for the function $\cF_2$ in~\eqref{eqnF2asymzero} and~\eqref{eqnF2asyminfty} readily turn into 
\begin{align}
 \label{eqnF2asymB0}  \cF_2(s) & = \frac{4\sqrt{2}}{3\beta^{\nicefrac{3}{2}}} (\beta-\zeta(\I s))^{\nicefrac{3}{2}}(1+\oo(1)), \quad s \rightarrow0, \\
 \label{eqnF2asymBinfty}  \cF_2(s) & = \frac{4\sqrt{2}}{3\beta^{\nicefrac{3}{2}}} (\beta+\zeta(\I s))^{\nicefrac{3}{2}}(1+\oo(1)), \quad s  \rightarrow\infty,
\end{align}
 we conclude from~\eqref{eqnLTEB} that 
\begin{align}\label{eqnLTprotoB}
  \sum_{\eta\in{\rm E} \atop \eta<1} (\beta-\zeta(\I\eta))^{\nicefrac{3}{2}} +  \sum_{\eta\in{\rm E} \atop \eta>1} (\beta+\zeta(\I\eta))^{\nicefrac{3}{2}} < \infty,
\end{align}
which proves condition~\ref{itmKS2c2}. 
Because one has (see Theorem~\ref{th:SimonFactor} for example) 
\begin{align}\label{eqnTMintB}
  \int_0^\infty \frac{\xi^2}{(1+\xi^2)^3} \biggl|\log\biggl|m(\zeta(\xi)+\I 0) - \frac{1-\I\xi}{1+\I\xi}\biggr| \biggr| d\xi < \infty, 
\end{align}
it follows from the bound on the integral on the right-hand side of~\eqref{eqnTFbetageq} that 
\begin{align}\label{eqnImintTB}
 \int_0^\infty \frac{\xi^2}{(1+\xi^2)^3} |\log(\im\, m(\zeta(\xi)+\I0))| d\xi <\infty,
\end{align}
which yields condition~\ref{itmKS2c3} after the transformation $\lambda=\zeta(\xi)$. 
\end{proof} 

\begin{proof}[Proof of sufficiency]
 Suppose now that a Herglotz--Nevanlinna function $m$ satisfies the three conditions in Theorem~\ref{thm:KSforExB2} with $m(0)=0$ and let $(L,\omega,\dip)$ be the corresponding generalized indefinite string.
 Condition~\ref{itmKS2c1} entails that the essential spectrum of $(L,\omega,\dip)$ is contained in $(-\infty,-\beta]\cup[\beta,\infty)$ and that $L=\infty$ in view of Theorem~\ref{thm:m-at0}, so that we can use the notation of Section~\ref{secSbSsumrule2}. 
 Condition~\ref{itmKS2c2} implies~\eqref{eqnLTprotoB}, where ${\rm E}=\{\eta>0\,|\, m(\zeta(\I\eta))=\infty\}$. 
In view of the asymptotics in~\eqref{eqnF2asymB0} and~\eqref{eqnF2asymBinfty} for the function $\cF_2$ we get~\eqref{eqnLTEB}. 
The set ${\rm Z}'=\{\kappa>0\,|\, (1-\kappa) m(\zeta(\I\kappa))=1+\kappa\}$ interlaces the set ${\rm E}\cup\{1\}$ because $m$ is a Herglotz--Nevanlinna function, which implies that 
\begin{align*}
   \sum_{\kappa\in{\rm Z}'} \cF_2(\kappa)  < \infty.
\end{align*} 
 Condition~\ref{itmKS2c3} transforms into~\eqref{eqnImintTB}, which gives 
  \begin{align*}
  - \int_\R \frac{\xi^2 \log T(\xi,0)}{(1+\xi^2)^3} d\xi < \infty
 \end{align*}
 because of~\eqref{eqnTMintB}. 
In particular, we see that conditions~\ref{itmKS2c1},~\ref{itmKS2c3} and that $m(0)=0$ guarantee that the requirements of Theorem~\ref{thm:traceflaB} and Proposition~\ref{proptraceflaBtrans} are satisfied. 
From the former, we get that for each $x\geq0$ one has the identity 
  \begin{align*}
  & \int_0^x (1+2\beta s)\Wr(s)^2 ds + \int_{0}^{x} (1+2\beta s)\biggl(\varrho(s)-\frac{1}{1+2\beta s}\biggr)^2 ds + \int_{[0,x)} (1+2\beta s)d\dip_{\sing}(s)  \\
&  \qquad = \frac{8}{\beta\pi} \int_\R \frac{\xi^2 \log|a(\xi,x)|}{(1+\xi^2)^3}d\xi  +  \frac{1}{\beta} \lim_{\delta\downarrow0}  \sum_{\kappa\in{\rm Z}\atop\delta<\kappa<1/\delta} \cF_2(\kappa) - \sum_{\eta\in{\rm P}\atop\delta<\eta<1/\delta} \cF_2(\eta),
\end{align*}
where ${\rm Z} = \{\kappa>0\,|\, a(\I\kappa,x)=0\}$ and ${\rm P} = \{\eta>0\,|\, a(\I\eta,x)=\infty\}$.
Estimating the right-hand side by using that $2\log|a(\xi,x)|\leq -\log T(\xi,0)$ for almost all $\xi\in\R$, that ${\rm Z}\subseteq{\rm Z}'$ and that $\cF_2$ is non-negative, before letting $x\rightarrow\infty$, this becomes
  \begin{align}\begin{split}\label{eqnTFbetaleq}
  & \int_0^\infty (1+2\beta s)\Wr(s)^2 ds \\
  & \qquad + \int_{0}^{\infty} (1+2\beta s)\biggl(\varrho(s)-\frac{1}{1+2\beta s}\biggr)^2 ds + \int_{[0,\infty)} (1+2\beta s)d\dip_{\sing}(s)  \\
&  \qquad \leq  - \frac{4}{\beta\pi} \int_\R \frac{\xi^2 \log T(\xi,0)}{(1+\xi^2)^3} d\xi +  \frac{1}{\beta} \sum_{\kappa\in{\rm Z}'} \cF_2(\kappa),
\end{split}\end{align}
 which concludes the proof of Theorem~\ref{thm:KSforExB2}. 
\end{proof} 

We are again also able to express the underlying trace formula given by~\eqref{eqnTFbetageq} and~\eqref{eqnTFbetaleq} in terms of the spectral measure of the generalized indefinite string. 

\begin{corollary}\label{coreqnTFbetaWT}
 If $\rho$ is the spectral measure of a generalized indefinite string $(L,\omega,\dip)$ with $L=\infty$ and~\eqref{eqnCondS2} for some constant $c\in\R$, then 
  \begin{align}\begin{split}\label{eqnTFbetam}
  & \int_0^\infty (1+2\beta x)(\Wr(x)-c)^2 dx \\
  & \qquad + \int_{0}^{\infty} (1+2\beta x)\biggl(\varrho(x)-\frac{1}{1+2\beta x}\biggr)^2 dx + \int_{(0,\infty)} (1+2\beta x)d\dip_{\sing}(x)  \\
&  \qquad = \int_\R \frac{d\rho(\lambda)}{\lambda^2} - \frac{1}{\pi} \int_{\R\backslash(-\beta,\beta)} \frac{\sqrt{\lambda^2-\beta^2}}{|\lambda|^3} d\lambda +  \frac{1}{\beta} \sum_{\lambda\in\supp(\rho)\atop|\lambda|<\beta} \cF_2\Biggl(\sqrt{\frac{\beta-\lambda}{\beta+\lambda}}\Biggr) \\
&  \qquad\qquad - \frac{1}{\pi} \int_{\R\backslash(-\beta,\beta)}  \frac{\sqrt{\lambda^2-\beta^2}}{|\lambda|^3} \log\Biggl(\frac{\pi|\lambda|}{\sqrt{\lambda^2-\beta^2}} \frac{d\rho_\ac(\lambda)}{d\lambda}\Biggr) d\lambda.
\end{split}\end{align}
\end{corollary}

\begin{proof}
 Because $\rho$ is also the spectral measure of the generalized indefinite string $(L,\omega-c\delta_0,\dip)$, where $\delta_0$ is the unit Dirac measure centered at zero, we may assume that $c=0$. 
  Since in this case one has $m(0)=0$, we know from the proof of Theorem~\ref{thm:KSforExB2} that the trace formula given by~\eqref{eqnTFbetageq} and~\eqref{eqnTFbetaleq} holds with equality. 
  By applying identity~\eqref{eq:Gtrace03B} from Theorem~\ref{thm:FactorMainB} to the meromorphic function $b$ on $\C_+$ defined by 
  \begin{align*}
    b(k) = 1 - m(\zeta(k))\frac{1+\I k}{1-\I k} = \frac{m(\zeta(k)) - \frac{1-\I k}{1+\I k}}{-\frac{1-\I k}{1+\I k}}, 
  \end{align*}
  we also obtain the relation  
  \begin{align}\begin{split}\label{eqnTFbetaaux}
     \beta\dot{m}(0) & = \frac{4}{\pi} \int_{\R} \frac{\xi^2}{(1+\xi^2)^3}  \log\Biggl(\biggl|m(\zeta(\xi)+\I 0) - \frac{1-\I |\xi|}{1+\I |\xi|}\biggr|^2\Biggr) d\xi   \\
    & \qquad\qquad\qquad\qquad + \lim_{\delta\downarrow0}  \sum_{\kappa\in{\rm Z}\atop\delta<\kappa<1/\delta} \cF_2(\kappa) - \sum_{\eta\in{\rm P}\atop\delta<\eta<1/\delta} \cF_2(\eta),
   \end{split}\end{align}
  where ${\rm Z} = \{\kappa>0\,|\, b(\I\kappa)=0\}$ and ${\rm P} = \{\eta>0\,|\, b(\I\eta)=\infty\}$.
  The set ${\rm Z}$ actually coincides with the set ${\rm Z}'$ in the trace formula given by~\eqref{eqnTFbetageq} and~\eqref{eqnTFbetaleq}, which in particular entails that we may omit the limit in~\eqref{eqnTFbetaaux} as the sums converge individually. 
  Moreover, the set ${\rm P}$ corresponds to the poles of $m$ in $(-\beta,\beta)$ and thus to the points in $\supp(\rho)$ between $-\beta$ and $\beta$. 
  After replacing the transmission coefficient in the trace formula given by~\eqref{eqnTFbetageq} and~\eqref{eqnTFbetaleq} with the expression in~\eqref{eqnT2def}, we can employ the relation in~\eqref{eqnTFbetaaux} to arrive at the formula in~\eqref{eqnTFbetam}. 
\end{proof}

In conclusion, let us also state the corresponding Lieb--Thirring inequality, the proof of which follows readily from Corollary~\ref{coreqnTFbetaWT}. 

\begin{corollary}\label{cor:LT02new}
If $\rho$ is the spectral measure of a generalized indefinite string $(L,\omega,\dip)$ with $L=\infty$ and~\eqref{eqnCondS2} for some constant $c\in\R$, then
\begin{align}\begin{split}
 & \frac{4\sqrt{2}}{3\beta^{\nicefrac{5}{2}}} \sum_{\lambda\in \supp(\rho) \atop |\lambda|<\beta} (\beta-|\lambda|)^{\nicefrac{3}{2}}  \leq   \int_0^\infty (1+2\beta x)(\Wr(x)-c)^2 dx \\
  & \qquad + \int_{0}^{\infty} (1+2\beta x)\biggl(\varrho(x)-\frac{1}{1+2\beta x}\biggr)^2 dx  + \int_{(0,\infty)} (1+2\beta x)d\dip_{\sing}(x). 
     \end{split}\end{align}
\end{corollary}

 \section{Krein strings}\label{sec:KreinString}

The purpose of this section is to apply our results to the class of Krein strings. 
Recall that a {\em Krein string} is a generalized indefinite string $(L,\omega,\dip)$ such that the distribution $\omega$ is a positive Borel measure on $[0,L)$ and the measure $\dip$ is identically zero.
In this case, we just write $(L,\omega)$ for a Krein string and the normalized anti-derivative of $\omega$ is simply given by the distribution function 
\begin{align}
  \Wr(x) = \int_{[0,x)}d\omega.
\end{align} 
It goes back to work of M.\ G.\ Krein~\cite{kakr74}\footnote{Let us mention that the definition of the Weyl--Titchmarsh function in~\cite{kakr74}, termed {\em coefficient of dynamical compliance} there, differs from ours given by~\eqref{eqnmdef}. Namely, the corresponding Weyl--Titchmarsh function $m_{\rm Kr}$ is defined by M.\ G.\ Krein as 
\begin{align*}
m_{\rm Kr}(z) = \lim_{x\to L} \frac{\phi(z,x)}{\theta(z,x)}.
\end{align*}
Comparing this definition with~\cite[Lemma~5.2]{IndefiniteString}, we immediately see that 
\begin{align*}
m(z) = -\frac{1}{zm_{\rm Kr}(z)},
\end{align*} 
and hence, in the terminology of~\cite[\S 12]{kakr74}, our Weyl--Titchmarsh function $m$ is nothing but the coefficient of dynamical compliance of the {\em dual string} (compare with~\cite[Equation~(12.5)]{kakr74}).} (see~\cite[Proposition~7.3]{IndefiniteString} for a proof in our setting) that the Weyl--Titchmarsh functions corresponding to Krein strings are precisely the Stieltjes functions, where a Herglotz--Nevanlinna function $m$ is called a {\em Stieltjes function} if the function $z\mapsto z m(z)$ is a Herglotz--Nevanlinna function as well (see~\cite[\S~5]{kakr74a}). 
Since it is known that Stieltjes functions always have an analytic extension to $\C\backslash[0,\infty)$, the following result is an immediate consequence of Theorem~\ref{thm:KSforExB1} and~\cite[Proposition~7.3]{IndefiniteString}.
As in previous sections, we let $\alpha$ be an arbitrary positive constant. 

 \begin{theorem}\label{thm:KS}
  A Stieltjes function $m$ is the Weyl--Titchmarsh function of a Krein string $(L,\omega)$ with $L=\infty$ and  
    \begin{align}\label{eqnCondKrStr1}
     \int_0^\infty \biggl(\int_{[x,\infty)}d\omega-\frac{1}{1+4\alpha x}\biggr)^2 x\, dx & < \infty
    \end{align}
   if and only if all the following conditions hold:
   \begin{enumerate}[label=(\roman*), ref=(\roman*), leftmargin=*, widest=iii]
    \item The function $m$ has a meromorphic extension to $\C\backslash[\alpha,\infty)$ that is analytic at zero. 
    \item The poles $\sigma_+$ of $m$ in $(0,\alpha)$ satisfy
    \begin{align}\label{eqnKSLT}
     \sum_{\lambda\in \sigma_+}(\alpha-\lambda)^{\nicefrac{3}{2}} <\infty.
    \end{align}
    \item The boundary values of the function $m$ satisfy~\eqref{eqnSzego-I}.
    \end{enumerate}
    \end{theorem}

   \begin{proof}
     After noticing that one has  
     \begin{align*}
       \int_{[x,\infty)}d\omega-\frac{1}{1+4\alpha x} = \int_{[0,\infty)}d\omega - \Wr(x) - \frac{1}{2\sqrt{\alpha}} + \frac{x}{1+2\sqrt{\alpha}x} + \OO\biggl(\frac{1}{x^2}\biggr)
     \end{align*} 
     as $x\rightarrow\infty$, the claim follows readily from Theorem~\ref{thm:KSforExB1} and~\cite[Proposition~7.3]{IndefiniteString}.
   \end{proof}

\begin{remark}
 As in Proposition~\ref{lem:GISat0}, one sees that the Weyl--Titchmarsh function $m$ of a Krein string $(L,\omega)$ has an analytic extension to zero if and only if $L=\infty$ and
  \begin{align}
\limsup_{x\rightarrow \infty} \, x \int_{[x,\infty)}d\omega <\infty.
\end{align}
 In this case, Theorem~\ref{thm:m-at0} entails that  
 \begin{align}
  m(0) = \int_{[0,\infty)}d\omega.
\end{align}
To the best of our knowledge, in the case of Krein strings, the first claim goes back to the work of I.\ S.\ Kac and M.\ G.\ Krein~\cite{kakr58}. 
The second claim can be found in~\cite[\S~11.4]{kakr74} for example.
\end{remark}

Of course, Theorem~\ref{thm:KS} can again also be stated in terms of the spectral measure. 
To this end, we first note that the integral representation~\eqref{eqnWTmIntRep} for a Stieltjes function $m$ simplifies to 
    \begin{align}
  m(z) = c_0  - \frac{1}{Lz} + \int_{[0,\infty)} \frac{d\rho(\lambda)}{\lambda-z}, 
\end{align}
where $c_0\in[0,\infty)$ is some constant and the measure $\rho$ is supported on $[0,\infty)$ with   
      \begin{align}\label{eqnSpecMeasStielt}
    \int_{[0,\infty)} \frac{d\rho(\lambda)}{1+\lambda} <\infty.
    \end{align}
    
 \begin{corollary}\label{cor:KS01}
  A positive Borel measure $\rho$ on $[0,\infty)$ with~\eqref{eqnSpecMeasStielt} is the spectral measure of a Krein string $(L,\omega)$ with $L=\infty$ and~\eqref{eqnCondKrStr1} if and only if both of the following conditions hold:
    \begin{enumerate}[label=(\roman*), ref=(\roman*), leftmargin=*, widest=ii]
      \item The support of $\rho$ is discrete in $[0,\alpha)$, does not contain zero and satisfies
  \begin{align}\label{eq:LTKS01}
     \sum_{\lambda\in \supp(\rho)\atop 0<\lambda<\alpha}(\alpha-\lambda)^{\nicefrac{3}{2}} <\infty.
  \end{align}
\item The absolutely continuous part $\rho_\ac$ of $\rho$ on $(\alpha,\infty)$ satisfies~\eqref{eqnSzego-Irho}.
    \end{enumerate}
    \end{corollary} 
    
    \begin{remark}
       The trace formula in Corollary~\ref{corTFalpham} clearly also holds true in this special case, as does the corresponding Lieb--Thirring inequality
\begin{align}\label{eq:LT01newKS}
   \sum_{\lambda\in \supp(\rho)\atop 0<\lambda<\alpha}(\alpha-\lambda)^{\nicefrac{3}{2}} & \leq   \frac{3\alpha^{3}}{4}  \int_0^\infty (1+2\sqrt{\alpha}x)\biggl(\int_{[x,\infty)}d\omega-\frac{1}{2\sqrt{\alpha}+4\alpha x}\biggr)^2dx. 
\end{align}
    \end{remark}

 Even though it is somewhat less immediate, we can also apply Theorem~\ref{thm:KSforExB2} to Krein strings. 
 However, we need some simple observations for this first. 

 \begin{proposition}\label{propOddWT}
    Let $(L,\omega,\dip)$ be a generalized indefinite string. 
    The corresponding Weyl--Titchmarsh function $m$ is odd, that is, satisfies
     \begin{align}\label{eqnmodd}
    m(z) = - m(-z)
   \end{align}
   for all $z\in\C\backslash\R$, if and only if $\omega$ vanishes identically. 
    In this case, one has 
    \begin{align}\label{eqnOddWT}
      m(z) = zm_+(z^2),
    \end{align}
    where $m_+$ is the Weyl--Titchmarsh function of the Krein string $(L,\dip)$. 
\end{proposition}

\begin{proof}
   One first notices that the Weyl--Titchmarsh function $\tilde{m}$ of the generalized indefinite string $(L,-\omega,\dip)$ is given by 
   \begin{align*}
     \tilde{m}(z) = - m(-z).
   \end{align*}
   In fact, this follows readily from the definition of the Weyl--Titchmarsh functions because a function $f$ is a solution of the differential equation~\eqref{eqnDEho} if and only if it is a solution of 
  \begin{align*}
  -f'' = (-z)(-\omega) f + (-z)^2 \dip f.
 \end{align*}
   As $m$ coincides with $\tilde{m}$ if and only if $(L,\omega,\dip)$ coincides with $(L,-\omega,\dip)$, we see that the equivalence in the claim holds. 
   Moreover, in this case one has 
    \begin{align*}
     m(z) =  \frac{\psi'\NLz}{z\psi(z,0)},  
    \end{align*} 
   where $\psi(z,\redot)$ is a non-trivial solution of the differential equation~\eqref{eqnDEho} which lies in $\Hast$ and $L^2([0,L);\dip)$. 
   Since $\psi(z,\redot)$ is then a non-trivial solution of 
   \begin{align*}
     -f'' =  z^2 \dip f
   \end{align*}
   which lies in $\Hast$, we have that 
   \begin{align*}
       m_+(z^2) & =  \frac{\psi'\NLz}{z^2\psi(z,0)} = \frac{m(z)}{z}
   \end{align*}
   by the definition of the Weyl--Titchmarsh function $m_+$.
\end{proof}

\begin{corollary}\label{corOddWT}
  Let $(L,\omega,\dip)$ be a generalized indefinite string. 
  The corresponding spectral measure $\rho$ is even if and only if $\Wr$ is equal to a constant almost everywhere.
\end{corollary}

\begin{proof}
 One first notices that the Weyl--Titchmarsh function $\tilde{m}$ of the generalized indefinite string $(L,\omega-c_2\delta_0,\dip)$ is given by $\tilde{m} = m - c_2$, where $c_2$ is the constant from the integral representation~\eqref{eqnWTmIntRep} and $\delta_0$ is the unit Dirac measure centered at zero.
   Now the claim follows readily from Proposition~\ref{propOddWT} because the spectral measure $\rho$ is even if and only if $\tilde{m}$ is odd. 
\end{proof}

\begin{remark}\label{rem:ExBviaKrein}
 The Weyl--Titchmarsh function $m$ of the generalized indefinite string $(L,\omega,\dip)$ from Example~\ref{exa2alpha} is clearly odd. 
In fact, the distribution $\omega$ is identically zero in this case. 
The corresponding Krein string $(L,\dip)$ as in Proposition~\ref{propOddWT} coincides precisely with the one considered in Example~\ref{exaalpha} with $\alpha = \beta^2$. 
 In particular, its Weyl--Titchmarsh function $m_+$ and the corresponding spectral measure are simply given by the expressions in Example~\ref{exaalpha}. 
\end{remark}

In order to state the second result, for a Krein string $(L,\dip)$ we will write 
 \begin{align}\label{def:varrho}
    \dip(B) = \int_B \varrho(x)^2 dx + \dip_\sing(B), 
 \end{align}
  where $\varrho$ is the (positive) square root of the Radon--Nikod\'ym derivative of $\dip$ with respect to the Lebesgue measure and $\dip_\sing$ is the singular part of $\dip$.

 \begin{theorem}\label{thmKSforKS02}
   A Stieltjes function $m_+$ is the Weyl--Titchmarsh function of a Krein string $(L,\dip)$ with $L=\infty$ and  
   \begin{align}\label{eqnCondKrStr2}
      \int_0^\infty   \biggl(\varrho(x) - \frac{1}{1+2\sqrt{\alpha} x} \biggr)^2 x\, dx + \int_{[0,\infty)} x\, d\dip_\sing(x) & < \infty
    \end{align}
  if and only if all the following conditions hold:
    \begin{enumerate}[label=(\roman*), ref=(\roman*), leftmargin=*, widest=iii]
    \item The function $m_+$ has a meromorphic extension to $\C\backslash[\alpha,\infty)$ that is analytic at zero. 
    \item The poles $\sigma_+$ of $m_+$ in $(0,\alpha)$ satisfy~\eqref{eqnKSLT}.
    \item The boundary values of the function $m_+$ satisfy 
    \begin{align}\label{eq:SzegomKS02}
 \int_{\alpha}^\infty \frac{\sqrt{\lambda-\alpha}}{\lambda^2}\log(\im\, m_+(\lambda+\I0)) d\lambda  >- \infty.
    \end{align}
    \end{enumerate}
    \end{theorem}

 \begin{proof}
    Any Stieltjes function $m_+$ is the Weyl--Titchmarsh function of some Krein string $(L,\dip)$. 
    We denote with $m$ the Weyl--Titchmarsh function of the generalized indefinite string $(L,0,\dip)$.
    Since the function $m$ is odd by Proposition~\ref{propOddWT} and the corresponding spectral measure $\rho$ is even in view of Corollary~\ref{corOddWT}, the integral representation~\eqref{eqnWTmIntRep} takes the form
    \begin{align*}
      m(z) = c_1 z - \frac{1}{Lz} + \int_{[0,\infty)} \frac{2z}{\lambda^2-z^2}d\rho(\lambda).
    \end{align*}
    From the relation~\eqref{eqnOddWT} in Proposition~\ref{propOddWT}, it then follows that 
    \begin{align*}
      m_+(z) = c_1 - \frac{1}{Lz} + \int_{[0,\infty)} \frac{2}{\lambda^2-z}d\rho(\lambda).
    \end{align*}
    After these considerations, the claim can be deduced by applying Theorem~\ref{thm:KSforExB2} and Corollary~\ref{cor:KSforExB2} with $\beta=\sqrt{\alpha}$ to the generalized indefinite string $(L,0,\dip)$.
  \end{proof}
    
 \begin{remark}\label{rem:traceKreinB}
 A couple of remarks are in order:
 \begin{enumerate}[label=(\alph*), ref=(\alph*), leftmargin=*, widest=e]
      \item\label{rem:traceKreinBi}  Condition~\eqref{eqnCondKrStr2} in Theorem~\ref{thmKSforKS02} means that the operators $\KIO_{(\varrho-\varrho_{\alpha})^2}$ and $\KIO_{\dip_\sing}$ belong to the trace class with the trace formula (involving~\eqref{eqnCondKrStr2} on the right-hand side)
 \begin{align}\label{eq:traceKreinB}
 \tr(\KIO_{(\varrho-\varrho_{\alpha})^2})+ \tr(\KIO_{\dip_\sing}) =  \int_0^\infty   \biggl( \varrho(x) - \frac{1}{1+2\sqrt{\alpha} x} \biggr)^2 x\, dx + \int_{[0,\infty)} x\, d\dip_\sing(x),
 \end{align}
 where $\varrho_\alpha$ is the corresponding function for the Krein string in Example~\ref{exaalpha}. 
  In this sense, the class of Krein strings in Theorem~\ref{thmKSforKS02} can be viewed again as suitable perturbations of the Krein string in Example~\ref{exaalpha}.  
  \item One can easily specify the trace formula in Corollary~\ref{coreqnTFbetaWT} to the class of Krein strings in Theorem~\ref{thmKSforKS02}: 
  If $\rho$ is the spectral measure of a Krein string $(L,\dip)$ with $L=\infty$ and~\eqref{eqnCondKrStr2}, then (with $\cF_2$ given by~\eqref{eq:Fsdef} in Section~\ref{secSbSsumrule})
  \begin{align}\begin{split}\label{eqnTFbetamKS2}
  & \int_{0}^{\infty} (1+2\sqrt{\alpha} x)\biggl(\varrho(x)-\frac{1}{1+2\sqrt{\alpha} x}\biggr)^2 dx + \int_{(0,\infty)} (1+2\sqrt{\alpha} x)d\dip_{\sing}(x)  \\
&  \qquad = \int_{(0,\infty)} \frac{d\rho(\lambda)}{\lambda} - \frac{1}{\pi} \int_{\alpha}^\infty \frac{\sqrt{\lambda-\alpha}}{\lambda^2} d\lambda +  \frac{2}{\sqrt{\alpha}} \sum_{\lambda\in\supp(\rho)\atop0<\lambda<\alpha} \cF_2\Biggl(\sqrt{\frac{\sqrt{\alpha}-\sqrt{\lambda}}{\sqrt{\alpha}+\sqrt{\lambda}}}\Biggr) \\
&  \qquad\qquad - \frac{1}{\pi} \int_{\alpha}^\infty  \frac{\sqrt{\lambda-\alpha}}{\lambda^2} \log\Biggl(\frac{\pi\lambda}{\sqrt{\lambda-\alpha}} \frac{d\rho_\ac(\lambda)}{d\lambda}\Biggr) d\lambda.
\end{split}\end{align}
  In particular, one gets the Lieb--Thirring inequality
  \begin{align}\label{eq:LT02newKS}
\begin{split}
\frac{4}{3\alpha^2}\sum_{\lambda\in \supp(\rho)\atop 0<\lambda<\alpha}(\alpha-\lambda)^{\nicefrac{3}{2}}
     &\leq  \int_{0}^{\infty} (1+2\sqrt{\alpha} x)\biggl(\varrho(x)-\frac{1}{1+2\sqrt{\alpha} x}\biggr)^2 dx\\
     &\qquad\qquad\qquad\qquad + \int_{(0,\infty)} (1+2\sqrt{\alpha} x)d\dip_{\sing}(x). 
     \end{split}  
\end{align}
 \item It is interesting to notice that the difference between the perturbations in Theorem~\ref{thm:KS} and Theorem~\ref{thmKSforKS02} on the spectral side lies simply in conditions~\eqref{eqnSzego-I} and~\eqref{eq:SzegomKS02} on the asymptotic behavior of the density of the spectral measure at infinity. 
 Condition~\eqref{eq:SzegomKS02} is clearly stronger than~\eqref{eqnSzego-I} and ensures that the absolutely continuous part of the Krein string's weight measure is not trivial, which is not necessary under condition~\eqref{eqnSzego-I} as we shall see in Section~\ref{secKreinLanger}. 
 Moreover, condition~\eqref{eq:SzegomKS02} has certain similarities with recent work of R.\ Bessonov and S.\ Denisov~\cite{bede20,bede21,bede22} to be discussed below.
 \end{enumerate}
 \end{remark}   

Let us finish this section with commenting on some connections between Theorem~\ref{thmKSforKS02} and the characterization of the Szeg\H{o} class obtained in~\cite{bede20}. 
More specifically, using the recent work of R.\ Bessonov and S.\ Denisov~\cite{bede20} together with some transformation rules discovered by M.\ G.\ Krein~\cite{kr53b}, one may characterize Krein strings $(L,\dip)$ with spectrum contained in $[\alpha,\infty)$ and spectral measure $\rho$ satisfying    
 \begin{align}\label{eq:SzegoKreinSalpha}
      \int_\alpha^\infty \frac{1}{\lambda\sqrt{\lambda-\alpha}}\log\biggl(\frac{d\rho_\ac(\lambda)}{d\lambda}\biggr) d\lambda >-\infty.  
      \end{align}
Unfortunately, this characterization is rather cumbersome when compared to the one in Theorem~\ref{thmKSforKS02}. 
We are not going to provide all the details here, but shall only indicate how to end up with such a characterization. 
To this end, recall that a Krein string $(L,\dip)$ belongs to the {\em Szeg\H{o} class}~\cite{bede20} if $L+\dip([0,L)) = \infty$ and the corresponding spectral measure $\rho$ satisfies
   \begin{align}\label{eq:SzegoKreinS}
      \int_0^\infty \frac{1}{\sqrt{\lambda}(1+\lambda)} \log\biggl(\frac{d\rho_\ac(\lambda)}{d\lambda}\biggr) d\lambda >-\infty.  
      \end{align}
Let us stress that the Szeg\H{o} condition~\eqref{eq:SzegoKreinS} was stated in~\cite{bede20} for the {\em principal} spectral measure $\rho_{\rm Kr}$ of a Krein string in the sense of M.\ G.\ Krein~\cite{kakr74} arising from their Weyl--Titchmarsh function $m_{\rm Kr}$.
 However, it is not difficult to see that the measure $\rho_{\rm Kr}$ satisfies the Szeg\H{o} condition if and only if so does the spectral measure $\rho$. 
 We also note that condition~\eqref{eq:SzegoKreinS} implies that the spectrum of a Krein string from the Szeg\H{o} class coincides with the positive half-line $[0,\infty)$ and its absolutely continuous spectrum is supported on $[0,\infty)$.  
The following characterization of Krein strings in the Szeg\H{o} class was obtained in~\cite[Theorem~2]{bede20}.     

  \begin{theorem}\label{th:BD}
      A Krein string  $(L,\dip)$ with $L+\dip([0,L)) = \infty$ belongs to the Szeg\H{o} class 
      if and only if $\varrho\notin L^1[0,L)$ and
      \begin{align}
      \sum_{n\in\N}\biggl((x_{n+2} - x_n) \int_{(x_{n},x_{n+2}]}d\dip - 4\biggr) <\infty,
      \end{align}
      where $\varrho$ is given by~\eqref{def:varrho} and the sequence $(x_n)_{n\in\N}$ is defined by 
      \begin{align}
      x_n = \inf \biggl\{x\ge 0\,\bigg|\, n=\int_0^x \varrho(s)ds\biggr\}.
      \end{align}
      \end{theorem}
      
Clearly, the difference between conditions~\eqref{eq:SzegoKreinSalpha} and~\eqref{eq:SzegoKreinS} is simply a shift in the spectral parameter by $\alpha$.
For this reason, it is possible to apply the following transformation rule of M.\ G.\ Krein~\cite{kr53b} (see also~\cite[Chapter~6.9]{dymc76}) in order to turn Theorem~\ref{th:BD} into a characterization of Krein strings with spectrum contained in $[\alpha,\infty)$ and satisfying condition~\eqref{eq:SzegoKreinSalpha}.
 
\begin{lemma}\label{lem:kreintransform}
Let $(L,\dip)$ be a Krein string with $L+\dip([0,L)) = \infty$ and suppose that another Krein string $(\tilde{L},\tilde{\dip})$ is defined by 
\begin{align}\label{eq:kreintransform01}
     \tilde{L} & = \lim_{x\to L} \gx(x), &  \int_{[0,\gx(x))}d\tilde{\dip} & = \int_{[0,x)} \theta(-\alpha,s)^2d\dip(s),
      \end{align}
      where $\gx(x)$ is given by 
      \begin{align}\label{eq:kreintransform02}
        \gx(x) & = \frac{\phi(-\alpha,x)}{\theta(-\alpha,x)}.
      \end{align}
Then the corresponding spectral measures $\rho_{\rm Kr}$ and $\tilde{\rho}_{\rm Kr}$ are related via 
\begin{align}\label{eq:rhoKrtransf}
 \int_{(-\infty,\lambda)} d\tilde{\rho}_{\rm Kr} = \int_{(-\infty,\lambda-\alpha)}d\rho_{\rm Kr}.
\end{align}
The above transformation holds for negative $\alpha$ too as long as $-\alpha\le \inf \supp(\rho_{\rm Kr})$ and $\rho_{\rm Kr}(\{-\alpha\})=0$. 
The latter is equivalent to $\theta(-\alpha,\redot)\notin L^2([0,L);\dip)$.
\end{lemma}

\begin{proof}
It suffices to observe that the fundamental system of solutions for the modified Krein string $(\tilde{L},\tilde{\dip})$ is simply given by
\begin{align*}
\tilde{\theta}(z,\gx(x)) & = \frac{\theta(z-\alpha,x)}{\theta(-\alpha,x)}, &  \tilde{\phi}(z,\gx(x)) & = \frac{\phi(z-\alpha,x)}{\theta(-\alpha,x)},
\end{align*}
and hence, by the definition of $m_{\rm Kr}$ and $\tilde{m}_{\rm Kr}$, one has 
\begin{align*}
\tilde{m}_{\rm Kr}(z) = m_{\rm Kr}(z-\alpha),
\end{align*}
which immediately implies~\eqref{eq:rhoKrtransf}. 
The remaining claim that this transformation works for negative $\alpha$ as well follows from monotonicity of $\gx(x)$ as a function of $x$. 
Indeed, a simple calculation shows that 
 \begin{align*}
  \gx'(x) & = \frac{1}{\theta(-\alpha,x)^2}, 
  \end{align*}  
and hence, taking into account oscillation properties of solutions, under the given assumptions, the function $\theta(-\alpha,\redot)$ does not vanish on $[0,L)$ and $\gx(x)$ is strictly increasing as a function of $x$.
\end{proof}

It is not at all surprising that the characterization of Krein strings with spectrum contained in $[\alpha,\infty)$ and satisfying~\eqref{eq:SzegoKreinSalpha} is rather complicated. 
One explanation for this stems from the fact that there does not appear to exist a criterion formulated in terms of the Krein string's coefficient for the equality $\inf \supp(\rho) = \alpha$ to hold (although there are sharp two-sided estimates, see~\cite{kakr58} and~\cite{muc}). 
What is curious in our opinion is the fact that condition~\eqref{eq:SzegoKreinSalpha} obviously implies~\eqref{eq:SzegoKS02} with the only difference between them being the behavior of the weight near the finite spectral edge. 
Moreover, according to Remark~\ref{rem:traceKreinB}~\ref{rem:traceKreinBi}, Corollary~\ref{corKSforKS02} can be seen as a characterization of spectral measures corresponding to certain perturbations of a model string and clearly Krein strings with~\eqref{eq:SzegoKreinSalpha} are a subclass of these. 
Using Krein's transformation in Lemma~\ref{lem:kreintransform} in a backward direction indicates that the Szeg\H{o} class can be viewed as a certain subclass, however, we are unaware of any characterizations of the Szeg\H{o} class from this perspective.

\section{Krein--Langer and Krein--Stieltjes strings}\label{secKreinLanger}

We are now going to focus on the important subclass of generalized indefinite strings with coefficients supported on discrete sets.
In this case, the differential equation~\eqref{eqnDEho} simply reduces to a difference equation. 

\begin{definition}
A {\em Krein--Langer string} is a generalized indefinite string $(L,\omega,\dip)$ such that the coefficients $\omega$ and $\dip$ are supported on a discrete set in $[0,L)$. 
\end{definition}

For a Krein--Langer string $(L,\omega,\dip)$, the distribution $\omega$ and the measure $\dip$ can be written in a unique way as  
    \begin{align}\label{eqnKL}
      \omega & =  \sum_{n=0}^{N} \omega_n \delta_{x_n}, & \dip & =  \sum_{n=0}^{N} \dip_n \delta_{x_n},
    \end{align}
 for some $N\in\N\cup\{0,\infty\}$, strictly increasing points $(x_n)_{n=0}^N$ in $[0,L)$ with $x_0=0$ and $x_n\rightarrow L$ if $N=\infty$, real constants $(\omega_n)_{n=0}^N$ and non-negative constants $(\dip_n)_{n=0}^N$ with $|\omega_n|+\dip_n>0$ for all $n\geq 1$ (notice that we do allow a simultaneous vanishing of $\omega_0$ and $\dip_0$). 
 Here we also use $\delta_x$ to denote the unit Dirac measure centered at $x$.  
 The distances between consecutive point masses are given by 
 \begin{align}
   \ell_n = x_n - x_{n-1}. 
 \end{align}

\begin{definition}
A {\em Krein--Stieltjes string} is a Krein string $(L,\omega)$ such that the coefficient $\omega$ is supported on a discrete set in $[0,L)$. 
\end{definition}

 It has been discovered by M.\ G.\ Krein, that Krein--Stieltjes strings play a crucial role in the study of the Stieltjes moment problem and that (except for a point mass at zero) the measure $\omega$ can be recovered explicitly in terms of the moments of the spectral measure $\rho$ in this case. 
 In a similar way, Krein--Langer strings are related to the indefinite moment problem~\cite{krla79} and to the Hamburger moment problem~\cite{IndMoment}.
 
 A characterization of all spectral measures corresponding to Krein--Langer strings is given by~\cite[Corollary~2.8]{StieltjesType} and can also be found less explicitly in~\cite{krla79} and~\cite[Section~5]{IndMoment}. 
 We are now able to also single out those Krein--Langer strings that additionally satisfy condition~\eqref{eqnCondS1} for some constants $c\in\R$ and $\alpha>0$.  

\begin{theorem}\label{th:HambMP}
  A positive Borel measure $\rho$ on $\R$ with~\eqref{eqnrhoPoisson} is the spectral measure of a Krein--Langer string $(L,\omega,\dip)$ with $L=N=\infty$ and 
 \begin{align}\label{eq:KSstrThA01}
 \sum_{n\in\N} \biggl(\frac{\ell_n}{x_{n}}\biggr)^3 + \sum_{n\in\N} \frac{\ell_{n}}{ x_{n}}\biggl(x_{n} \biggl(c+\sum_{k<n} \omega_k \biggr) + \frac{1}{4\alpha}\biggr)^2 + \sum_{n\in\N} x_n\dip_n & <\infty
 \end{align}
for some constant $c\in\R$ if and only if all the following conditions hold:
    \begin{enumerate}[label=(\roman*), ref=(\roman*), leftmargin=*, widest=iii]
  \item\label{itmHambMPi} All moments of the measure $\rho$ are finite and the corresponding Hamburger moment problem is determinate (that is, has a unique solution).
      \item\label{itmHambMPii} The support of $\rho$ is discrete in $(-\infty,\alpha)$, does not contain zero and satisfies~\eqref{eqnLT-Irho}.
   \item\label{itmHambMPiii} The absolutely continuous part $\rho_\ac$ of $\rho$ on $(\alpha,\infty)$ satisfies~\eqref{eqnSzego-Irho}.
        \end{enumerate}
    \end{theorem} 

\begin{proof}
 We are first going to show that for a Krein--Langer string $(L,\omega,\dip)$ with $L=N=\infty$, condition~\eqref{eqnCondS1} holds for some $c\in\R$ if and only if~\eqref{eq:KSstrThA01} holds for some $c\in\R$.
 In fact, the former property is equivalent to 
  \begin{align}\label{eq:KSstrThA02}
     \int_{x_1}^\infty \biggl(4\alpha\Wr(x)-c + \frac{1}{x}\biggr)^2 x\, dx + \sum_{n\in\N} x_n\dip_n & < \infty
    \end{align}
 for some $c\in\R$. 
 Since the normalized anti-derivative $\Wr$ is simply a piecewise constant function in this case, the integral in~\eqref{eq:KSstrThA02} becomes 
\begin{align*}
  & \sum_{n\geq2} \int^{x_n}_{x_{n-1}} \biggl(\tilde{\Wr}_n+\frac{1}{x}\biggr)^2 x\,dx \\ 
& \qquad = \sum_{n\geq2}  \tilde{\Wr}_n^2 \ell_{n} \biggl(x_{n}-\frac{\ell_{n}}{2}\biggr) + 2\tilde{\Wr}_n \ell_{n} + \log\biggl(\frac{x_{n}}{x_{n-1}}\biggr) \\
& \qquad = \sum_{n\geq2}  \frac{\ell_{n}}{x_n} \biggl(1-\frac{\ell_{n}}{2x_{n}}\biggr)\biggl(x_n\tilde{\Wr}_n + \biggl(1-\frac{\ell_{n}}{2x_{n}}\biggr)^{-1}\biggr)^2 + F\biggl(\frac{\ell_{n}}{x_{n-1}}\biggr),
\end{align*}
where we introduced the quantities $\tilde{\Wr}_n$ defined by 
\begin{align*}
\tilde{\Wr}_n = -c + 4\alpha \sum_{k< n} \omega_k 
\end{align*}
and used the function $F\colon [0,\infty)\rightarrow\R$ given by 
\begin{align*}
F(x) = \log(1+x) - \frac{2x}{2+x}.
\end{align*}
The function $F$ is strictly increasing because
\begin{align*}
F'(x) = \frac{1}{1+x} - \frac{4}{(2+x)^2} = \frac{x^2}{(1+x)(2+x)^2} >0
\end{align*}
 and taking into account that $F(0) = 0$, one ends up with the estimate
\begin{align}\label{eq:Festimate}
\frac{x^3}{3(2+a)^3} \leq F(x) \leq \frac{x^3}{12}
\end{align}
for all $x\in[0,a]$ and any fixed $a> 0$. 
In particular, we see that both terms in the last sum above are non-negative and hence~\eqref{eq:KSstrThA02} becomes
 \begin{align}\label{eq:KSstrThA03}
    \sum_{n\geq2}  F\biggl(\frac{\ell_{n}}{x_{n-1}}\biggr) + \sum_{n\geq2}  \frac{\ell_{n}}{x_n} \biggl(1-\frac{\ell_{n}}{2x_{n}}\biggr)\biggl(x_n\tilde{\Wr}_n + \biggl(1-\frac{\ell_{n}}{2x_{n}}\biggr)^{-1}\biggr)^2 +  \sum_{n\in\N} x_n\dip_n & < \infty.
    \end{align}
    Using the estimate in~\eqref{eq:Festimate} and the fact that $\ell_{n} = \oo(x_n)$ as well as $x_n/x_{n-1}\rightarrow1$ are necessary for the validity of~\eqref{eq:KSstrThA03} and of~\eqref{eq:KSstrThA01}, it now follows readily that~\eqref{eq:KSstrThA03} holds for some $c\in\R$ if and only if~\eqref{eq:KSstrThA01} holds for some $c\in\R$.
 
 Now if a positive Borel measure $\rho$ on $\R$ with~\eqref{eqnrhoPoisson} is the spectral measure of a Krein--Langer string $(L,\omega,\dip)$ with $L=N=\infty$ and~\eqref{eq:KSstrThA01} for some constant $c\in\R$, then condition~\ref{itmHambMPi} follows from~\cite[Corollary~2.8]{StieltjesType} and divergence of the series
 \begin{align*}
   \sum_{n\in\N} \ell_n,
 \end{align*}
 as the latter implies that the corresponding Hamburger moment problem is determinate in view of~\cite[Theorem~5.7]{IndMoment}. 
 Conditions~\ref{itmHambMPii} and~\ref{itmHambMPiii} follow from Corollary~\ref{cor:KSforExB1} by means of the equivalence proved above.
 Conversely, if a positive Borel measure $\rho$ on $\R$ with~\eqref{eqnrhoPoisson} satisfies conditions~\ref{itmHambMPi},~\ref{itmHambMPii} and~\ref{itmHambMPiii}, then by Corollary~\ref{cor:KSforExB1} it is the spectral measure of a generalized indefinite string $(L,\omega,\dip)$ with $L=\infty$ and~\eqref{eqnCondS1} for some $c\in\R$. 
 Since all moments of $\rho$ are finite, we infer from~\cite[Corollary~2.7]{StieltjesType} that the coefficients $\omega$ and $\dip$ are supported on an infinite but discrete set in $[0,L_d)$ for some $L_d\in(0,\infty]$.
 However, because the corresponding Hamburger moment problem is determinate, we conclude from~\cite[Theorem~5.7]{IndMoment} that $L_d=\infty$ and thus $(L,\omega,\dip)$ is a Krein--Langer string with $N=\infty$. 
 It remains to apply the equivalence proved above to see that~\eqref{eq:KSstrThA01} holds for some constant $c\in\R$.
\end{proof}    
    
\begin{remark}
 Determinacy of the Hamburger moment problem in condition~\ref{itmHambMPi} in Theorem~\ref{th:HambMP} can be replaced by a more explicit condition on certain Hankel determinants as in~\cite[Corollary~2.8]{StieltjesType}. 
 Moreover, the coefficients in condition~\eqref{eq:KSstrThA01} can also be expressed in terms of these Hankel determinants (see~\cite[Corollary~2.8]{StieltjesType} or~\cite[Section~5.3]{IndMoment}). 
\end{remark}

    
\begin{corollary}\label{cor:StiMP}
  A positive Borel measure $\rho$ on $[0,\infty)$ with~\eqref{eqnSpecMeasStielt} is the spectral measure of a Krein--Stieltjes string $(L,\omega)$ with $L=N=\infty$ and 
 \begin{align}\label{eq:KSstrThA}
 \sum_{n\in\N} \biggl(\frac{\ell_n}{x_{n}}\biggr)^3 + \sum_{n\in\N} \frac{\ell_{n}}{ x_{n}}\biggl(x_{n} \sum_{k\ge n} \omega_k - \frac{1}{4\alpha}\biggr)^2 <\infty
 \end{align}
 if and only if all the following conditions hold:
    \begin{enumerate}[label=(\roman*), ref=(\roman*), leftmargin=*, widest=iii]
    \item\label{itmStiMPi}  All moments of the measure $\rho$ are finite and the corresponding Stieltjes moment problem is determinate (that is, has a unique solution).
      \item\label{itmStiMPii} The support of $\rho$ is discrete in $[0,\alpha)$, does not contain zero and satisfies~\eqref{eq:LTKS01}.
     \item\label{itmStiMPiii} The absolutely continuous part $\rho_\ac$ of $\rho$ on $(\alpha,\infty)$ satisfies~\eqref{eqnSzego-Irho}.
  \end{enumerate}
    \end{corollary} 

\begin{proof}
   If a positive Borel measure $\rho$ on $[0,\infty)$ with~\eqref{eqnSpecMeasStielt} is the spectral measure of a Krein--Stieltjes string $(L,\omega)$ with $L=N=\infty$ and~\eqref{eq:KSstrThA}, then~\eqref{eq:KSstrThA01} holds as well with 
   \begin{align*}
      c=  -\sum_{k=0}^\infty \omega_k. 
   \end{align*}
  Conditions~\ref{itmStiMPi},~\ref{itmStiMPii} and~\ref{itmStiMPiii} now follow from Theorem~\ref{th:HambMP} because determinacy of the Hamburger moment problem implies that of the Stieltjes moment problem. 
   
 Conversely, if a positive Borel measure $\rho$ on $[0,\infty)$ with~\eqref{eqnSpecMeasStielt} satisfies conditions~\ref{itmStiMPi},~\ref{itmStiMPii} and~\ref{itmStiMPiii}, then $\rho$ also satisfies the conditions of Theorem~\ref{th:HambMP} in view of~\cite[Corollary~8.9]{sc17}.
 This guarantees that $\rho$ is the spectral measure of a Krein--Langer string $(L,\omega,\dip)$ with $L=N=\infty$ and~\eqref{eq:KSstrThA01} for some $c\in\R$.
 Since removing possible point masses at zero does not change the spectral measure, we may assume that $\omega_0=\dip_0=0$. 
 It then follows from~\cite[Lemma~7.2]{IndefiniteString} that $(L,\omega,\dip)$ is actually a Krein string and thus a Krein--Stieltjes string. 
 Because~\eqref{eq:KSstrThA01} holds for some $c\in\R$ and $\omega$ is a finite measure in view of Corollary~\ref{cor:KS01}, we infer that 
  \begin{align*}
 \sum_{n\in\N} \biggl(\frac{\ell_n}{x_{n}}\biggr)^3 + \sum_{n\in\N} \frac{\ell_{n}}{ x_{n}}\biggl(x_{n} \biggl(\tilde{c}-\sum_{k\ge n} \omega_k\biggr) + \frac{1}{4\alpha}\biggr)^2 <\infty
 \end{align*}
 for some $\tilde{c}\in\R$. 
 However, because the series 
 \begin{align*}
   \sum_{n\in\N} \ell_n x_n \geq x_1 \sum_{n\in\N} \ell_n
 \end{align*}
 diverges, we conclude that $\tilde{c}$ must necessarily be zero, which proves~\eqref{eq:KSstrThA}.
\end{proof}

\begin{remark}
  Determinacy of the Stieltjes moment problem in condition~\ref{itmStiMPi} in Corollary~\ref{cor:StiMP} can be replaced with the stronger determinacy of the Hamburger moment problem. 
  In fact, under condition~\ref{itmStiMPii} one has $\rho(\{0\})=0$, which guarantees that the two notions coincide; see~\cite[Corollary~8.9]{sc17}. 
\end{remark}

\begin{example}[Laguerre operator]\label{ex:Laguerre}
An example of a measure satisfying all conditions of Theorem~\ref{th:HambMP} and Corollary~\ref{cor:StiMP} is given by\footnote{More generally, one may consider the measure $\rho_\alpha^\gamma$ given by 
\begin{align*} 
   \rho_\alpha^\gamma(B) = \frac{1}{\Gamma(1+\gamma)} \int_B \id_{[\alpha,\infty)}(\lambda)(\lambda-\alpha)^\gamma\E^{-(\lambda-\alpha)} d\lambda,
\end{align*}
 where $\gamma>-1$. 
 This leads to generalized Laguerre--Sonin polynomials $L_n^{(\gamma)}$ (see~\cite{sze} for example) and all expressions become much more cumbersome.} 
\begin{align}\label{eq:LaguerreMeasure}
    \rho_{\alpha}(B) = \int_B \id_{[\alpha,\infty)}(\lambda)\E^{-(\lambda-\alpha)} d\lambda,
\end{align}
which is nothing but the shifted Laguerre weight. 
 The corresponding Krein--Stieltjes string $(L,\omega)$ such that $\omega$ has no point mass at zero is explicitly given by (see~\cite[Equation~(3.3)]{IndMoment})
\begin{align}\label{eq:LaguerreCoeff}
\omega_n & = \frac{1}{nL_{n-1}(-\alpha)L_{n}(-\alpha)}, & \ell_n & = L_{n-1}(-\alpha)^2, 
\end{align}
 where the $L_n$ are the classical Laguerre polynomials~\cite[\S~18.3]{dlmf}, \cite[Chapter~V]{sze}: 
\begin{align}\label{eq:laguerpol}
L_n(z) = \frac{\E^{z}}{n!}\frac{d^n}{dz^n}\E^{-z} z^{n}  = \sum_{k=0}^n\binom{n}{k}\,\frac{(-z)^k}{k!}.
\end{align}
From the asymptotic behavior of Laguerre polynomials (see~\cite[Theorem~8.22.3]{sze} for example), 
we infer that 
\begin{align}\label{eqnLagCoeff01}
\omega_n & = \frac{4\pi\alpha}{\E^{-\alpha}}\frac{\E^{-4\sqrt{\alpha n}}}{\sqrt{\alpha n}}(1+\oo(1)) , & \ell_n & =  \frac{\E^{-\alpha}}{4\pi}\frac{\E^{4\sqrt{\alpha n}}}{\sqrt{\alpha n}}(1+\oo(1)), 
\end{align}
as $n\rightarrow\infty$, which, after a little effort, also implies that   
\begin{align}\label{eqnLagCoeff02}
 x_n = \E^{4\sqrt{\alpha n} +\OO(1)}.
\end{align}
Of course, the coefficients in~\eqref{eq:LaguerreCoeff} for the Krein--Stieltjes string $(L,\omega)$ are a bit more cumbersome than the corresponding Jacobi parameters, as the Jacobi matrix generated by the measure $\rho_\alpha$ is simply given by 
\begin{align}\label{eq:LagOper}
J_\alpha = \begin{pmatrix} 
1+\alpha & 1 & 0 & 0 & \cdots \\
1 & 3+\alpha & 2 & 0 & \cdots \\
0 & 2 & 5+\alpha & 3 & \cdots \\
0 & 0 & 3 & 7+\alpha & \cdots \\
\vdots & \vdots & \vdots & \vdots & \ddots 
\end{pmatrix}.
\end{align}
However, it is not clear to us whether it is possible to obtain analogs of Theorem~\ref{th:HambMP} or Corollary~\ref{cor:StiMP} for Jacobi matrices that are corresponding perturbations of the Laguerre operator $J_\alpha$.
\end{example}

 \section{Conservative Camassa--Holm flow}\label{secCH}

  In this section, we are going to demonstrate how our results so far apply to the isospectral problem  
\begin{align}\label{eqnCHISP}
 -f'' + \frac{1}{4} f = z\, \omega f + z^2 \dip f 
\end{align}
for the conservative Camassa--Holm flow. 
 Since this differential equation differs from~\eqref{eqnDEho} only by a constant coefficient term, it is not surprising that we are able to obtain complete solutions of inverse problems for certain classes of coefficients when~\eqref{eqnCHISP} is considered on a half-line. 
 Less immediately obvious but more important, we can also use our results to solve some inverse problems for certain classes of step-like coefficients (with strong decay at one endpoint and positive asymptotics at the other endpoint) on the full line, which establishes inverse spectral transforms for the conservative Camassa--Holm flow on the corresponding phase spaces.

 \subsection{Inverse spectral theory for half-line coefficients}\label{secCH:01}

  Let $u$ be a real-valued function in $H^1_\loc[0,\infty)$ and let $\dip$ be a positive Borel measure on $[0,\infty)$. 
  The real distribution $\omega$ in $H^{-1}_\loc[0,\infty)$ is defined by  
\begin{align}
 \omega(h) = \int_0^\infty u(x)h(x)dx + \int_0^\infty u'(x)h'(x)dx
\end{align}
for all $h\in H^1_{\cc}[0,\infty)$, so that $\omega = u - u''$ in a weak sense. 
 Under these assumptions, a solution of the differential equation~\eqref{eqnCHISP} is a function $f\in H^1_{\loc}[0,\infty)$ such that 
 \begin{align}
  f'(0-) h(0) + \int_{0}^\infty f'(x) h'(x) dx + \frac{1}{4} \int_0^\infty f(x)h(x)dx = z\, \omega(fh) + z^2 \int_{[0,\infty)} f h\, d\dip 
 \end{align} 
 for a unique constant $f'(0-)\in\C$ and every $h\in H^1_\cc[0,\infty)$.
The Weyl--Titchmarsh function $m$ associated with this spectral problem is then defined on $\C\backslash\R$ by 
 \begin{align}
  m(z) =  \frac{\psi'\NLz}{z\psi(z,0)},
 \end{align} 
  where $\psi(z,\redot)$ is the (up to scalar multiples) unique non-trivial solution of the differential equation~\eqref{eqnCHISP} which lies in $H^1[0,\infty)$ and $L^2([0,\infty);\dip)$, guaranteed to exist by~\cite[Corollary 7.2]{ACSpec}. 
 As for generalized indefinite strings, this function is a Herglotz--Nevanlinna function that contains all the spectral information. 
 In fact, the measure in the corresponding integral representation is a spectral measure for a suitable self-adjoint realization; see~\cite{bebrwe08},~\cite{LeftDefiniteSL} and~\cite{CHPencil} for details. 

 We have shown in~\cite[Section~7]{ACSpec} that for this kind of coefficients, it is always possible to transform the spectral problem~\eqref{eqnCHISP} into an associated generalized indefinite string $(\infty,\tilde{\omega},\tilde{\dip})$, where the distribution $\tilde{\omega}$ is defined via its normalized anti-derivative $\tilde{\Wr}$ by 
 \begin{align}
   \tilde{\Wr}(x)  =  u(0) - \frac{u(\log(1+x))+ u'(\log(1+x))}{1+x}  
 \end{align}
 and the measure $\tilde{\dip}$ is given by  
 \begin{align}
   \tilde{\dip}(B) =  \int_{\log(1+B)} \E^{-x} d\dip(x).
 \end{align}
 The Weyl--Titchmarsh function $\tilde{m}$ of this generalized indefinite string $(\infty,\tilde{\omega},\tilde{\dip})$ is related to $m$ simply via 
 \begin{align}\label{eqnmms}
   m(z) = \tilde{m}(z) - \frac{1}{2z}, 
 \end{align} 
 that is, the functions $m$ and $\tilde{m}$ coincide up to a pole at zero (see~\cite[page~3557]{ACSpec}). 
Conversely, every generalized indefinite string $(\infty,\tilde{\omega},\tilde{\dip})$ arises in this way.  

 By means of this connection, it is easily possible to translate all the results of this article to the spectral problem~\eqref{eqnCHISP} on the half-line.
 We refrain from stating these theorems explicitly here, but only mention what the conditions in our main theorems turn into. 
 Let us just leave a hint here that in order to verify the following equivalences, one should also remember that a real-valued function $h\in H^1_\loc[0,\infty)$ belongs to $H^1[0,\infty)$ if and only if $h+h'$ belongs to $L^2[0,\infty)$.  
  
    \begin{enumerate}[label=(\roman*), ref=(\roman*), leftmargin=*, widest=iii]
    \item\label{itmCHSpecThm1}
  Condition~\eqref{eqnCondS1} in Theorem~\ref{thm:KSforExB1} on the generalized indefinite string $(\infty,\tilde{\omega},\tilde{\dip})$ with $c=u(0)-(2\sqrt{\alpha})^{-1}$ is equivalent to the condition that the function $u-\frac{1}{4\alpha}$ belongs to $H^1[0,\infty)$ and the measure $\dip$ is finite (see Corollary~\ref{cor:KSforCHonR+}).
    
    \item\label{itmCHSpecThm2} 
    Condition~\eqref{eqnCondS2} in Theorem~\ref{thm:KSforExB2} on the generalized indefinite string $(\infty,\tilde{\omega},\tilde{\dip})$ with $c=u(0)$ is equivalent to the condition that the function $u$ belongs to $H^1[0,\infty)$, the function $\varrho - \frac{1}{2\beta}$ belongs to $L^2[0,\infty)$ and the measure $\dip_\sing$ is finite, where $\varrho$ is the (positive) square root of the Radon--Nikod\'ym derivative of $\dip$ with respect to the Lebesgue measure and $\dip_\sing$ is the singular part of $\dip$.
    \end{enumerate} 
Due to the importance of the case when $\omega$ is a positive Borel measure and $\dip$ vanishes identically (these assumptions prevent blow-ups and one expects uniqueness of weak solutions; see~\cite{como00}), we also mention what one obtains in this case from the conditions in Section~\ref{sec:KreinString} on Krein strings. 
     \begin{enumerate}[label=(\roman*), ref=(\roman*), leftmargin=*, widest=iii, resume]
    \item\label{itmCH+SpecThm1}
  Condition~\eqref{eqnCondKrStr1} in Theorem~\ref{thm:KS} on the generalized indefinite string $(\infty,\tilde{\omega},\tilde{\dip})$ is equivalent to the condition that the function $u-\frac{1}{4\alpha}$ belongs to $H^1[0,\infty)$.  
        \item\label{itmCH+SpecThm2}
    Condition~\eqref{eqnCondKrStr2} in Theorem~\ref{thmKSforKS02} on the generalized indefinite string $(\infty,\tilde{\omega},\tilde{\dip})$ is equivalent to the condition that the function $\varrho_\omega - (2\sqrt{\alpha})^{-1}$ belongs to $L^2[0,\infty)$ and the measure $\omega_\sing$ is finite, where $\varrho_\omega$ is the (positive) square root of the Radon--Nikod\'ym derivative of $\omega$ with respect to the Lebesgue measure and $\omega_\sing$ is the singular part of $\omega$.
  \end{enumerate}  
 Finally, one can also translate the conditions from Theorem~\ref{th:HambMP} and Corollary~\ref{cor:StiMP}, in which case one ends up with functions $u$ that are made up of infinitely many peakons (compare Remark~\ref{remWPD1}~\ref{rem:LaguerreMP} below).
 
 \begin{remark}
 All of the conditions on the function $u$ and the measure $\dip$ listed above are related to conserved quantities of the Camassa--Holm equation and its two-component generalization (both considered on the real line). 
\begin{enumerate}[label=(\alph*), ref=(\alph*), leftmargin=*, widest=e]
    \item The $H^1(\R)$ norm of $u-\kappa$ is conserved~\cite{caho93} for classical solutions $u$ of the Camassa--Holm equation~\eqref{eqnCH}, where $\kappa$ is a positive constant related to the critical wave speed.
    \item The sum of the $H^1(\R)$ norm of $u$ and the $L^2(\R)$ norm of $\varrho - \kappa$ is conserved~\cite[Equation~(1.13)]{hoiv11} for classical solutions $(u,\varrho)$ of the two-component Camassa--Holm system~\eqref{eqn2CH}. 
    \item The $L^2(\R)$ norm of $\varrho_\omega - \sqrt{\kappa}$ is conserved~\cite[Equation~(195)]{cogeiv07} for classical solutions $u$ of the Camassa--Holm equation~\eqref{eqnCH} under the additional assumption that $\omega=u-u_{xx}$ is positive. 
    We note that under this positivity restriction, there is a relationship of conserved quantities with the Korteweg--de Vries equation via the Liouville correspondence (see~\cite{besasz98, le04, mc03b} for instance).
    \end{enumerate}
  In conclusion, let us mention that these quantities are only conserved for classical solutions. 
  For conservative weak solutions, they will have to be modified appropriately to also take singular parts of the solution into account.  
 \end{remark}

 \subsection{Inverse spectral theory for step-like coefficients}\label{secCH:02}

As mentioned above, our results can also be applied to certain classes of step-like coefficients on the full line. 
  In this setting, we denote with $H^1_{\loc}(\R)$ and $H^1_{\cc}(\R)$ the function spaces   
\begin{align}
H^1_{\loc}(\R) & =  \lbrace f\in AC_{\loc}(\R) \,|\, f'\in L^2_{\loc}(\R) \rbrace, \\
 H^1_{\cc}(\R) & = \lbrace f\in H^1_{\loc}(\R) \,|\, \supp(f) \text{ compact}\rbrace,
\end{align}
and with $H^{-1}_{\loc}(\R)$ the topological dual space of $H^1_{\cc}(\R)$.  
 We are first going to describe a particular phase space $\CHdom$ for the conservative Camassa--Holm flow, introduced in~\cite[Definition~1.1]{Eplusminus} as follows:
 
\begin{definition}\label{defD}
 The set $\CHdom$ consists of all pairs $(u,\mu)$ such that $u$ is a real-valued function in $H^1_{\loc}(\R)$ and $\mu$ is a positive Borel measure on $\R$ with
\begin{align}\label{eqnumu}
   \int_B u(x)^2 + u'(x)^2\, dx \leq  \mu(B)
\end{align}
for every Borel set $B\subseteq\R$, satisfying the asymptotic growth restrictions  
\begin{align}
 \label{eqnMdef-}   \int_{-\infty}^0 \E^{-s} \bigl(u(s)^2 + u'(s)^2 \bigr) ds +  \int_{(-\infty,0)} \E^{-s} d\dip(s) & < \infty,  \\
 \label{eqnMdef+}   \limsup_{x\rightarrow\infty}\, \E^{x} \biggl(\int_{x}^{\infty}\E^{-s}(u(s) + u'(s))^2ds + \int_{[x,\infty)}\E^{-s}d\dip(s)\biggr) & < \infty,
\end{align}
where $\dip$ is the positive Borel measure on $\R$ defined such that 
\begin{align}\label{eqndipdef}
 \mu(B) = \dip(B) + \int_B u(x)^2 + u'(x)^2\, dx. 
\end{align}
\end{definition}

\begin{remark}
 A couple of remarks are in order:
\begin{enumerate}[label=(\alph*), ref=(\alph*), leftmargin=*, widest=e]
    \item
Condition~\eqref{eqnMdef-} in Definition~\ref{defD} requires strong decay of both, the function $u$ and the measure $\dip$, at $-\infty$ (for $u$ it means that the function $x\mapsto \E^{-x/2}u(x)$ belongs to $H^1$ near $-\infty$), whereas condition~\eqref{eqnMdef+} on the behavior near $+\infty$ is rather mild and satisfied as soon as $u+u'$ is bounded and $\dip$ is a finite measure for instance. 
Despite its cumbersome looking form, after a simple change of variables, condition~\eqref{eqnMdef+} turns into a boundedness condition under the action of the classical Hardy operator (see~\cite{Eplusminus} for further details).
    \item
We chose to work with pairs $(u,\mu)$ and the unusual condition~\eqref{eqnumu} instead of the simpler definable pairs $(u,\dip)$ for various reasons. 
For example, the measure $\mu$ is more natural when considering suitable notions of convergence on $\CHdom$; see~\cite[Definition~2.5]{Eplusminus}.
Moreover, in the context of the conservative Camassa--Holm flow, the measure $\mu$ corresponds to the energy of a solution. 
\end{enumerate}
\end{remark}

Associated with each pair $(u,\mu)$ in $\CHdom$ is a distribution $\omega$ in $H^{-1}_{\loc}(\R)$ defined by
\begin{align}
 \omega(h) = \int_\R u(x)h(x)dx + \int_\R u'(x)h'(x)dx
\end{align}
for all $h\in H^1_{\cc}(\R)$, so that $\omega = u - u''$ in a weak sense, and a measure $\dip$ on $\R$ given by~\eqref{eqndipdef}. 
With these coefficients, the differential equation~\eqref{eqnCHISP} has to be understood in a distributional sense again (see \cite[Appendix~A]{ConservCH} for more details):
  A solution of~\eqref{eqnCHISP} is a function $f\in H^1_{\loc}(\R)$ such that 
 \begin{align}
   \int_{\R} f'(x) h'(x) dx + \frac{1}{4} \int_\R f(x)h(x)dx = z\, \omega(fh) + z^2 \int_\R f h \,d\dip 
 \end{align} 
 for every $h\in H^1_\cc(\R)$.
 Even though the derivative of such a solution $f$ is in general only defined almost everywhere, there is always a unique left-continuous function $f^\qd$ on $\R$ such that 
 \begin{align} 
     f^\qd = f' +\frac{1}{2} f - z (u +u') f 
\end{align} 
 almost everywhere on $\R$, called the {\em quasi-derivative} of $f$. 

 The main consequence of the strong decay condition on the pair $(u,\mu)$ at $-\infty$ in~\eqref{eqnMdef-} is the existence of a particular fundamental system $\phi(z,\redot)$, $\theta(z,\redot)$ of solutions to the differential equation~\eqref{eqnCHISP} with the asymptotics 
\begin{align}
  \phi(z,x) \E^{-\frac{x}{2}}& \rightarrow 1, & \phi^\qd(z,x) \E^{-\frac{x}{2}} & \rightarrow 1,   &   \theta(z,x)\E^{\frac{x}{2}} & \rightarrow 1, &  \theta^\qd(z,x)\E^{-\frac{x}{2}} & \rightarrow 0,   
\end{align}
as $x\rightarrow-\infty$; see~\cite[Theorem~3.1]{Eplusminus}. 
With this fundamental system of solutions, it is possible to define a {\em Weyl--Titchmarsh function} $m$ on $\C\backslash\R$ via 
\begin{align}
  m(z) = -\lim_{x\rightarrow\infty} \frac{\theta(z,x)}{z\phi(z,x)},
\end{align}
which is a Herglotz--Nevanlinna function that allows a particular integral representation of the form 
\begin{align}
  m(z) = \int_\R \frac{z}{\lambda(\lambda-z)} d\rho(\lambda), 
\end{align}
where $\rho$ is the same positive Borel measure on $\R$ as in the representation~\eqref{eqnWTmIntRep}, that satisfies~\eqref{eqnWTrhoPoisson}.  
The measure $\rho$ can be seen to be a spectral measure for a suitable self-adjoint realization again (we only refer to~\cite[Section~5]{CHPencil} for more details).
It turns out that zero does not belong to the support of $\rho$ and that the size of the spectral gap around zero is closely connected to the quantity in~\eqref{eqnMdef+}; see~\cite[Proposition~3.5]{Eplusminus} and~\eqref{eq:lam0est} below. 

\begin{definition}
The set $\cR_0$ consists of all positive Borel measures $\rho_0$ on $\R$ that satisfy
 \begin{align}
 \int_\R \frac{d\rho_0(\lambda)}{1+\lambda^2} < \infty
\end{align}
 and whose topological support does not contain zero. 
\end{definition}

One of the main results in~\cite[Section~4]{Eplusminus} states that {\em the spectral transform $(u,\mu)\mapsto\rho$ is a bijection between the phase space $\CHdom$ and the set $\cR_0$}. 
 In particular, unlike for generalized indefinite strings, the pair $(u,\mu)$ and thus the coefficients $\omega$ and $\dip$ in~\eqref{eqnCHISP} are uniquely determined by the spectral measure $\rho$ here.
 
It will be crucial for us below that the function $m$ and the measure $\rho$ introduced above for a pair $(u,\mu)$ in $\CHdom$ are the Weyl--Titchmarsh function and the spectral measure of the generalized indefinite string $(\infty,\tilde{\omega},\tilde{\dip})$, where the distribution $\tilde{\omega}$ is defined via its normalized anti-derivative $\tilde{\Wr}$ by 
 \begin{align}\label{eq:DvsStringWr}
 \tilde{\Wr}(x) = -\frac{u(\log x) + u'(\log x)}{x}  
 \end{align}
 and the measure $\tilde{\dip}$ is given by  
 \begin{align}\label{eq:DvsStringDip}
  \int_{[0,x)} d\tilde{\dip}   = \int_{(-\infty,\log x)} \E^{- s} d\dip(s).
 \end{align} 
 In fact, the function $\tilde{\Wr}$ is square integrable, the measure $\tilde{\dip}$ is finite with no point mass at zero and both satisfy the asymptotic condition
  \begin{align}
  \limsup_{x\rightarrow\infty}\,  x \int_{x}^\infty  \tilde{\Wr}(s)^2ds  + x \int_{[x,\infty)} d\tilde{\dip} < \infty.
\end{align}
 Conversely, every generalized indefinite string $(\infty,\tilde{\omega},\tilde{\dip})$ with these properties arises in this way from a unique pair $(u,\mu)$ in $\CHdom$; see~\cite[Section~2 and Section~3]{Eplusminus}.
 
 We are now able to use these connections to translate our main results. 
 As before, the constants $\alpha$ and $\beta$ will always be assumed to be positive. 
 Our first theorem characterizes the spectral measures $\rho$ for those pairs $(u,\mu)$ in $\CHdom$ such that the function $u-\kappa$ lies in $H^1$ near $+\infty$ for a positive constant $\kappa$ and the measure $\dip$ is finite near $+\infty$. 
 
 \begin{theorem}\label{thmCHualpha}
   A pair $(u,\mu)$ in $\CHdom$ satisfies 
   \begin{align}\label{eqnCHualpha}
     \int_0^\infty \biggl(u(x)-\frac{1}{4\alpha}\biggr)^2 + u'(x)^2 \, dx + \int_{[0,\infty)} d\dip < \infty
   \end{align}
 if and only if both of the following conditions hold for the spectral measure $\rho$ in $\mathcal{R}_0$:
    \begin{enumerate}[label=(\roman*), ref=(\roman*), leftmargin=*, widest=ii]
\item\label{itmCHalphai} The support of $\rho$ is discrete in $(-\infty,\alpha)$ and satisfies~\eqref{eqnLT-Irho}. 
\item\label{itmCHalphaii} The absolutely continuous part $\rho_\ac$ of $\rho$ on $(\alpha,\infty)$ satisfies~\eqref{eqnSzego-Irho}.
    \end{enumerate}
 \end{theorem}
 
 \begin{proof}
   Condition~\eqref{eqnCHualpha} is equivalent to 
       \begin{align}\label{thmCHualphaGIS}
      \int_0^\infty \biggl(\tilde{\Wr}(x)+\frac{1}{2\sqrt{\alpha}}-\frac{x}{1+2\sqrt{\alpha}x}\biggr)^2 x\, dx +  \int_{[0,\infty)} x\, d\tilde{\dip}(x) < \infty,
    \end{align}
    so that necessity of conditions~\ref{itmCHalphai} and~\ref{itmCHalphaii} follows immediately from Corollary~\ref{cor:KSforExB1}. 
    The conditions are also sufficient because according to Corollary~\ref{cor:KSforExB1} (note that it is known that $\rho$ has no mass near zero) they imply that $(\infty,\tilde{\omega},\tilde{\dip})$ satisfies condition~\eqref{eqnCondS1} for some $c\in\R$.
       However, since $\tilde{\Wr}$ is known to be square integrable, this constant $c$ is necessarily $-(2\sqrt{\alpha})^{-1}$ and thus one arrives at~\eqref{eqnCHualpha} again.  
 \end{proof}
 
 In order to apply Theorem~\ref{thm:KSforExB2} next, we will write 
 \begin{align}
    \dip(B) = \int_B \varrho(x)^2 dx + \dip_\sing(B), 
 \end{align}
  where $\varrho$ is the (positive) square root of the Radon--Nikod\'ym derivative of $\dip$ with respect to the Lebesgue measure and $\dip_\sing$ is the singular part of $\dip$.
   One is then able to characterize the spectral measures $\rho$ for those pairs $(u,\mu)$ in $\CHdom$ such that the function $u$ lies in $H^1$ near $+\infty$, the function $\varrho-\kappa$ lies in $L^2$ near $+\infty$ for a positive constant $\kappa$ and the measure $\dip_{\sing}$ is finite near $+\infty$. 
 
  \begin{theorem}\label{thm:SzegoStepLikeII}
   A pair $(u,\mu)$ in $\CHdom$ satisfies 
   \begin{align}\label{eqnCHubeta}
     \int_0^\infty u(x)^2 + u'(x)^2 \, dx + \int_0^\infty   \biggl( \varrho(x) - \frac{1}{2\beta} \biggr)^2 dx + \int_{[0,\infty)} d\dip_{\sing} < \infty
   \end{align}
  if and only if both of the following conditions hold for the spectral measure $\rho$ in $\mathcal{R}_0$:
    \begin{enumerate}[label=(\roman*), ref=(\roman*), leftmargin=*, widest=ii]
\item The support of $\rho$ is discrete in $(-\beta,\beta)$ and satisfies~\eqref{eqnLT-IIrho}.
\item The absolutely continuous part $\rho_\ac$ of $\rho$ on $(-\infty,-\beta)\cup(\beta,\infty)$ satisfies~\eqref{eqnSzego-IIrho}.
    \end{enumerate} 
 \end{theorem}
 
 \begin{proof}
  We first note that 
   \begin{align*}
    \tilde{\varrho}(x) & = \frac{1}{x} \varrho(\log x), & \int_{[0,x)} d\tilde{\dip}_\sing & = \int_{(-\infty,\log x)} \E^{- s} d\dip_\sing(s),
 \end{align*} 
  where $\tilde{\varrho}$ is the (positive) square root of the Radon--Nikod\'ym derivative of $\tilde{\dip}$ with respect to the Lebesgue measure and $\tilde{\dip}_\sing$ is the singular part of $\tilde{\dip}$. 
    Then condition~\eqref{eqnCHubeta} is equivalent to 
       \begin{align*}
      \int_0^\infty \tilde{\Wr}(x)^2 x\, dx + \int_0^\infty \biggl( \tilde{\varrho}(x) - \frac{1}{1+2\beta x} \biggr)^2 x\, dx + \int_{[0,\infty)} x\, d\tilde{\dip}_\sing(x) < \infty
    \end{align*}
    and the claim follows from Corollary~\ref{cor:KSforExB2} as in the proof of Theorem~\ref{thmCHualpha}.
    \end{proof}
 
    According to~\cite[Proposition~3.7]{Eplusminus}, the spectral measure $\rho$ of a pair $(u,\mu)$ in $\CHdom$ is supported on the positive half-line if and only if the measure $\dip$ vanishes identically and the distribution $\omega$ is a positive Borel measure on $\R$. 
    In particular, this means that $\mu$ is an absolutely continuous measure on $\R$ that is uniquely determined by the function $u$ in view of~\eqref{eqndipdef}.  
   The set of all pairs $(u,\mu)$ in $\CHdom$ with these properties will be denoted with $\CHdom^+$ in the following.
   
   \begin{definition}
     The set $\CHdom^+$ consists of all pairs $(u,\mu)$ in $\CHdom$ such that the distribution $\omega$ is a positive Borel measure on $\R$ and the measure $\dip$ vanishes identically. 
   \end{definition}
   
   From Theorem~\ref{thmCHualpha}, we immediately obtain a characterization of the spectral measures $\rho$ for those pairs $(u,\mu)$ in $\CHdom^+$ such that the function $u-\kappa$ lies in $H^1$ near $+\infty$ for a positive constant $\kappa$. 
   
    \begin{corollary}
   A pair $(u,\mu)$ in $\CHdom^+$ satisfies 
         \begin{align}\label{eqnCHKalpha}
            \int_0^\infty \biggl(u(x)-\frac{1}{4\alpha}\biggr)^2 + u'(x)^2 \, dx < \infty
     \end{align}
  if and only if both of the following conditions hold for the spectral measure $\rho$ in $\mathcal{R}_0$:
    \begin{enumerate}[label=(\roman*), ref=(\roman*), leftmargin=*, widest=ii]
      \item The support of $\rho$ is discrete in $[0,\alpha)$ and satisfies~\eqref{eq:LTKS01}.
\item The absolutely continuous part $\rho_\ac$ of $\rho$ on $(\alpha,\infty)$ satisfies~\eqref{eqnSzego-Irho}.
    \end{enumerate}
  \end{corollary}
  
  One more result can be derived from Theorem~\ref{thmKSforKS02} about Krein strings. 
  To this end, we first introduce the set $\CHdom^+_1$ of all pairs $(u,\mu)$ in $\CHdom^+$ such that 
  \begin{align}
    \int_{(-\infty,0)} \E^{-x} d\omega(x) < \infty.
  \end{align}
   We have shown in~\cite[Corollary~3.9]{Eplusminus} that this set can also be characterized in terms of the corresponding spectral measure $\rho$. 
  More specifically, a pair $(u,\mu)$ in $\CHdom^+$ belongs to $\CHdom^+_1$ if and only if 
  \begin{align}
    \int_{(0,\infty)} \frac{d\rho(\lambda)}{\lambda} < \infty. 
  \end{align}
  In this case, we shall write (recall that $\omega$ is a positive Borel measure on $\R$ when the pair $(u,\mu)$ belongs to $\CHdom^+$) 
 \begin{align}
    \omega(B) = \int_B \varrho_\omega(x)^2 dx + \omega_\sing(B), 
 \end{align}
  where $\varrho_\omega$ is the (positive) square root of the Radon--Nikod\'ym derivative of $\omega$ with respect to the Lebesgue measure and $\omega_\sing$ is the singular part of $\omega$.
   Our final theorem characterizes the spectral measures $\rho$ for those pairs $(u,\mu)$ in $\CHdom^+_1$ such that the function $\varrho_\omega-\sqrt{\kappa}$ lies in $L^2$ near $+\infty$ for a positive constant $\kappa$ and the measure $\omega_\sing$ is finite near $+\infty$. 
 
    \begin{theorem}
   A pair $(u,\mu)$ in $\CHdom^+_1$ satisfies 
     \begin{align}\label{eqnCHwalpha}
         \int_0^\infty \biggl( \varrho_\omega(x) - \frac{1}{2\sqrt{\alpha}} \biggr)^2 dx + \int_{[0,\infty)} d\omega_{\sing} < \infty
     \end{align}
   if and only if both of the following conditions hold for the spectral measure $\rho$ in $\cR_0$:
      \begin{enumerate}[label=(\roman*), ref=(\roman*), leftmargin=*, widest=ii]
      \item The support of $\rho$ is discrete in $[0,\alpha)$ and satisfies~\eqref{eq:LTKS01}.
\item The absolutely continuous part $\rho_\ac$ of $\rho$ on $(\alpha,\infty)$ satisfies~\eqref{eq:SzegoKS02}.
    \end{enumerate}
  \end{theorem}
 
 \begin{proof}
   If $(u,\mu)$ is a pair in $\CHdom^+_1$, then the measure $\tilde{\dip}$ vanishes identically and the function $\tilde{\Wr}$ has a non-decreasing representative that is bounded from below; see~\cite[Lemma~7.2]{IndefiniteString} and~\cite[Remark~3.8 and Corollary~3.9]{Eplusminus}.
   It follows that $\tilde{\omega}+c\delta_0$ is a positive Borel measure on $[0,\infty)$ for some constant $c\in\R$, where $\delta_0$ is the unit Dirac measure centered at zero.
   Since adding point masses at zero to $\tilde{\omega}$ does not change the spectral measure, we see that $\rho$ is also the spectral measure of the Krein string $(\infty,\tilde{\omega}+c\delta_0)$.
   Furthermore, we will use that 
   \begin{align*}
     \int_{[x,y)} \E^{-s}d\omega(s) = \int_{[\E^x,\E^y)} d\tilde{\omega},
   \end{align*}
   which follows from~\cite[Equation~(3.26)]{Eplusminus} and implies that (take into account that a point mass at zero does not change the absolutely continuous part of a measure)
   \begin{align*}
    \tilde{\varrho}_\omega(x) & = \frac{1}{x} \varrho_\omega(\log x), & \int_{(0,x)} d\tilde{\omega}_\sing & = \int_{(-\infty,\log x)} \E^{- s} d\omega_\sing(s),
 \end{align*} 
  where $\tilde{\varrho}_\omega$ is the (positive) square root of the Radon--Nikod\'ym derivative of $\tilde{\omega}$ with respect to the Lebesgue measure and $\tilde{\omega}_\sing$ is the singular part of $\tilde{\omega}$. 
     Now condition~\eqref{eqnCHwalpha} is equivalent to 
   \begin{align*}
      \int_0^\infty \biggl(\tilde{\varrho}_\omega(x) - \frac{1}{1+2\sqrt{\alpha} x} \biggr)^2 x\, dx + \int_{[0,\infty)} x\, d\tilde{\omega}_\sing(x) & < \infty
    \end{align*}
    and the claim follows from Corollary~\ref{corKSforKS02} as in the proof of Theorem~\ref{thmCHualpha}.
 \end{proof}

\subsection{Well-posedness results}\label{ss:13wellposed}

 We are now going to demonstrate how our results in Section~\ref{secCH:02} can be used to obtain global conservative (weak) solutions of the two-component Camassa--Holm system on phase spaces that arise from the conditions on the coefficients. 
 For the sake of brevity, we will focus on the discussion of the phase space obtained from condition~\eqref{eqnCHualpha} in Theorem~\ref{thmCHualpha}, which determines precisely the set $\CHdom^\kappa$ as given in Definition~\ref{defDcapH1} (with the parameters $\alpha$ and $\kappa$ related by $4\alpha\kappa=1$). 
 In order to simplify lengthy expressions, we will also restrict to the particular case when $\kappa=1$, so that $\CHdom^1$ consists of all pairs $(u,\mu)$ in $\CHdom$ such that  
   \begin{align}\label{eqnD1cond}
     \int_{0}^{\infty}(u(x)-1)^2 + u'(x)^2dx + \int_{[0,\infty)} d\dip & < \infty.
   \end{align}
   Note that condition~\eqref{eqnD1cond} is stronger than~\eqref{eqnMdef+}, which ensures that this characterization of $\CHdom^1$ agrees with our original Definition~\ref{defDcapH1} in the introduction.

  Two main ingredients will be necessary to establish the conservative Camassa--Holm flow on $\CHdom^1$ below. 
  First of all, we will need the characterization of all spectral measures corresponding to pairs in $\CHdom^1$ as given by Theorem~\ref{thmCHualpha} with $\alpha=\nicefrac{1}{4}$. 
  Secondly, we are also going to employ the trace formula in Corollary~\ref{corTFalpham}, which readily translates to our current setting. 
  In order to shorten notation in the following, for a pair $(u,\mu)$ in $\CHdom^1$ we set 
   \begin{align}\label{eq:energy1}
     \cE_1(u,\mu) = \int_{\R} (1+\E^{-x}) \biggl(u(x)+u'(x) - \frac{1}{1+\E^{-x}}\biggr)^2 dx + \int_{\R} (1+\E^{-x})d\dip(x).
   \end{align}  
  We also continue to use the function $\cF_2$ as defined by~\eqref{eq:Fsdef} in Section~\ref{secSbSsumrule}.
  
   \begin{corollary}\label{corTFCH1}
      If $\rho$ is the spectral measure of a pair $(u,\mu)$ in $\CHdom^1$, then 
  \begin{align}\begin{split}\label{eqnTFCH1}
  \cE_1(u,\mu) & = \int_\R \frac{d\rho(\lambda)}{\lambda^2} - \frac{1}{2\pi} \int_{\nicefrac{1}{4}}^\infty \frac{\sqrt{4\lambda-1}}{\lambda^3}d\lambda + 2 \sum_{\lambda\in \supp(\rho)\atop \lambda<\nicefrac{1}{4}} \cF_2\bigl(\sqrt{1-4\lambda}\bigr) \\
&  \qquad\qquad\qquad\qquad - \frac{1}{2\pi} \int_{\nicefrac{1}{4}}^{\infty}  \frac{\sqrt{4\lambda-1}}{\lambda^3} \log\biggl(\frac{2\pi\lambda}{\sqrt{4\lambda-1}} \frac{d\rho_\ac(\lambda)}{d\lambda}\biggr) d\lambda.
\end{split}\end{align}
   \end{corollary}
   
   \begin{proof}
     This follows from Corollary~\ref{corTFalpham} applied to the corresponding generalized indefinite string $(\infty,\tilde{\omega},\tilde{\dip})$ as defined in Section~\ref{secCH:02}, which satisfies~\eqref{thmCHualphaGIS}.
   \end{proof}
    
    Since both integrals in~\eqref{eq:energy1} are non-negative, the functional $\cE_1$ can be defined on all of $\CHdom$, so that $\CHdom^1$ consists of all pairs $(u,\mu)$ in $\CHdom$ such that $\cE_1(u,\mu)$ is finite.  
    We will make use of this defining functional below to improve on a natural topology on $\CHdom$ that was introduced in~\cite[Section~2]{Eplusminus}.
    
  \begin{definition}\label{defDTop}
    We say that a sequence of pairs $(u_n,\mu_n)$ converges to $(u,\mu)$ in $\CHdom$ if the functions $u_n$ converge to $u$ pointwise on $\R$ and 
    \begin{align}\label{eqnContMU}
      \int_{(-\infty,x)} \E^{-s}d\mu_n(s) \rightarrow \int_{(-\infty,x)} \E^{-s}d\mu(s)
    \end{align}
    for almost every $x\in\R$.
  \end{definition}   
 
   That this mode of convergence on $\CHdom$ is indeed induced by a metric can be seen from~\cite[Proposition~4.5]{Eplusminus} for example, which states that a sequence of pairs $(u_n,\mu_n)$ converges to $(u,\mu)$ in $\CHdom$ if and only if the corresponding Weyl--Titchmarsh functions $m_n$ converge to $m$ locally uniformly.
   We will always regard $\CHdom$ equipped with the topology induced by such a metric. 
 The same convention applies to the following stronger mode of convergence on $\CHdom^1$ that renders the functional $\cE_1$ continuous. 
 
 \begin{definition}\label{defD1Top}
We say that a sequence of pairs $(u_n,\mu_n)$ converges to $(u,\mu)$ in $\CHdom^1$ if it converges to $(u,\mu)$ in $\CHdom$ and $\cE_1(u_n,\mu_n)$ converges to $\cE_1(u,\mu)$. 
  \end{definition} 

We are first going to compare these two modes of convergence with more standard ones.
 
  \begin{lemma}\label{lem:convergenceD}
If a sequence of pairs $(u_n,\mu_n)$ converges to $(u,\mu)$ in $\CHdom$, then the following assertions hold: 
 \begin{enumerate}[label=(\roman*), ref=(\roman*), leftmargin=*, widest=iii]
           \item\label{itmtopDii} The  functions $u_n$ converge to $u$ weakly in $H^1(-\infty,x)$ for every $x\in\R$.
          \item\label{itmtopDiv} For all functions $h\in C_{\mathrm{b}}(\R)$ with support bounded from above one has 
    \begin{align}
      \int_\R h(x)\E^{-x}d\mu_n(x)\rightarrow\int_\R h(x)\E^{-x}d\mu(x). 
    \end{align}
 \end{enumerate}
If in addition $\cE_1(u_n,\mu_n)$ is bounded and converges to $\cE_1(u,\mu)$, that is, the sequence of pairs $(u_n,\mu_n)$ converges to $(u,\mu)$ in $\CHdom^1$, then the following assertions hold:
 \begin{enumerate}[label=(\roman*), ref=(\roman*), leftmargin=*, widest=iii, resume]
               \item\label{itmtopD1ii}   The functions $u_n-1$ converge to $u-1$ weakly in $H^1[x,\infty)$ for every $x\in\R$. 
                          \item\label{itmtopD1i}  The functions $u_n$ converge to $u$ uniformly on $\R$.
 \end{enumerate}
\end{lemma}

\begin{proof}
 We first observe that the convergence in~\eqref{eqnContMU} entails that  
\begin{align*}  
  C(x) := \sup_{n\in\N}\int_{(-\infty,x)} \E^{-s} d\mu_n(s) <\infty
  \end{align*}
  for every $x\in\R$, which guarantees that the $H^1(-\infty,x)$ norms of the functions $u_n$ are uniformly bounded.
Together with pointwise convergence of the functions $u_n$, this implies that the functions $u_n$ converge weakly in $H^1(-\infty,x)$.
   The claim in~\ref{itmtopDiv} simply follows from pointwise convergence almost everywhere of the distribution function in~\eqref{eqnContMU}. 
  Under the additional assumption that $\cE_1(u_n,\mu_n)$ is bounded and converges to $\cE_1(u,\mu)$, we first note that the bound 
  \begin{align*}
    \int_x^\infty (u_n(s)-1)^2 + u_n'(s)^2 ds & = \int_x^\infty (u_n(s)+u_n'(s)-1)^2ds + (u_n(x)-1)^2\\
     & \leq 2 \cE_1(u_n,\mu_n) + \E^{-2x} + \frac{4}{3}\E^x C(x) + 2
  \end{align*}
  holds for every $x\in\R$, where we used~\cite[Lemma~2.3]{Eplusminus} to estimate $u_n(x)$.
  This shows that the $H^1[x,\infty)$ norms of the functions $u_n-1$ are uniformly bounded, which implies that the functions $u_n-1$ converge weakly in $H^1[x,\infty)$.
  The remaining claim in~\ref{itmtopD1i} follows from~\ref{itmtopDii} and~\ref{itmtopD1ii}.
\end{proof}      
   
   \begin{remark}
   One may be tempted to try for a finer topology on $\CHdom$ or $\CHdom^1$ that ensures strong convergence of the function $u$ in $H^{1}_\loc(\R)$ instead of weak convergence.
    However, due to the presence of finite-time blow-up (see~\cite[Section~6]{brco07} for the prototypical example of a peakon-antipeakon collision), solutions of the Camassa--Holm equation would not be continuous in time with respect to such a topology. 
  \end{remark}   
   
 Let us consider next the two-component Camassa--Holm system    
\begin{align}
 \begin{split}\label{eqnOurCH}
  u_t + u u_x + P_x & = 0, \\
  \mu_t + (u\mu)_x & = (u^3 - 2Pu)_x, 
 \end{split}
 \end{align}
 where the auxiliary function $P$ satisfies
 \begin{align}
  P - P_{xx} & = \frac{u^2+ \mu}{2}.
 \end{align}    
 We first specify a precise meaning of weak solutions to this system in $\CHdom$.
   
 \begin{definition}\label{def:weaksolCH}
A {\em global conservative solution} of the two-component Camassa--Holm system~\eqref{eqnOurCH} with initial data $(u_0,\mu_0)\in\CHdom$ is a continuous curve 
\begin{align}
  \gamma\colon t\mapsto(u(\ledot,t),\mu(\ledot,t))
\end{align}
from $\R$ to $\CHdom$ with $\gamma(0)=(u_0,\mu_0)$ that satisfies~\eqref{eqnOurCH} in the sense that for every test function $\varphi\in C_\cc^\infty(\R\times\R)$ one has  
 \begin{align}\label{eqnCHsysweak1}
 & \int_\R \int_\R u(x,t) \varphi_t(x,t) + \biggl(\frac{u(x,t)^2}{2} + P(x,t) \biggr) \varphi_x(x,t) \,dx \,dt = 0, \\
 \begin{split} 
 &  \int_\R \int_\R \varphi_t(x,t) + u(x,t) \varphi_x(x,t) \,d\mu(x,t) \,dt  \\ 
 &   \qquad\qquad\qquad\qquad = 2\int_\R \int_\R u(x,t)\biggl(\frac{u(x,t)^2}{2} - P(x,t) \biggr) \varphi_x(x,t) \,dx \,dt,
 \end{split}
 \end{align}
 where the function $P$ on $\R\times\R$ is given by 
 \begin{align}\label{eq:Pdef}
  P(x,t) =  \frac{1}{4} \int_\R \E^{-|x-s|} u(s,t)^2 ds +  \frac{1}{4} \int_\R \E^{-|x-s|} d\mu(s,t).
 \end{align}
\end{definition}   
   
  In~\cite[Section~5]{Eplusminus}, we showed how global conservative solutions of the two-component Camassa--Holm system~\eqref{eqnOurCH} can be obtained by defining the {\em conservative Camassa--Holm flow} $\Phi$ on $\CHdom$ as a map 
  \begin{align}\label{eqnCHflow}
   \Phi\colon  \CHdom\times\R \rightarrow\CHdom
 \end{align}
 in the following way:
 Given a pair $(u,\mu)$ in $\CHdom$ with associated spectral measure $\rho$ as well as a $t\in\R$, the corresponding image $\Phi^t(u,\mu)$ under $\Phi$ is defined as the unique pair in $\CHdom$ for which the associated spectral measure is given by    
   \begin{align}\label{eqnSMEvo}
    B \mapsto \int_B \E^{-\frac{t}{2\lambda}} d\rho(\lambda).
  \end{align}
  We note that $\Phi$ is well-defined as the measure given by~\eqref{eqnSMEvo} belongs to $\cR_0$ whenever so does $\rho$ and hence the existence of a unique corresponding pair $\Phi^t(u,\mu)$ in $\CHdom$ is guaranteed by the fact that the spectral transform $(u,\mu)\mapsto\rho$ is a bijection between $\CHdom$ and $\cR_0$. 
  The definition of this flow is motivated by the well-known time evolution of spectral data for spatially decaying classical solutions of the Camassa--Holm equation as well as multi-peakons; see \cite[Section~6]{besasz98}.
  By combining~\cite[Proposition~5.2]{Eplusminus} and~\cite[Theorem~5.3]{Eplusminus}, we see that the flow $\Phi$ indeed gives rise to global conservative solutions: 

 \begin{theorem}\label{thm:consCHweak}
  For every pair $(u_0,\mu_0)\in\CHdom$, the integral curve $t\mapsto \Phi^t(u_0,\mu_0)$ is a global conservative solution of the two-component Camassa--Holm system~\eqref{eqnOurCH} with initial data $(u_0,\mu_0)$.
 \end{theorem} 

\begin{remark}\label{rem:BressanUniq}
 The question about uniqueness of conservative weak solutions to the Camassa--Holm equation and its two-component generalization is a subtle one. 
 Uniqueness of conservative weak solutions to the Camassa--Holm equation has been established in~\cite{bcz15} (see also~\cite{bre16}) under the assumption that the initial data $u_0$ belongs to $H^1(\R)$. 
 However, the notion of weak solution employed in~\cite{bcz15} is stronger than ours, so that this uniqueness result does not apply in our case. 
 In particular, it requires H\"older continuity of the function $u$ in both variables as well as Lipschitz continuity of $t\mapsto u(\ledot,t)$ as a map into $L^2(\R)$.
\end{remark}

In combination with continuous dependence on the initial data established in~\cite[Proposition~5.2]{Eplusminus}, Theorem~\ref{thm:consCHweak} leads to a well-posedness result for the two-com\-po\-nent Camassa--Holm system~\eqref{eqnOurCH} on $\CHdom$. 
To this end, we first recall the notion of {\em multi-peakon profiles}, particular kinds of pairs in $\CHdom$ playing the role of multi-solitons for the Camassa--Holm equation. 

\begin{definition}\label{def:MPfin}
A pair $(u,\mu)$ in $\CHdom$ is called a {\em multi-peakon profile} if $\omega$ and $\dip$ are Borel measures that are supported on a finite set. 
The set of all multi-peakon profiles in $\CHdom$ is denoted by $\Peakons$. 
\end{definition}

If $(u,\mu)$ is a multi-peakon profile, then the function $u$ is of the form
\begin{align}\label{eq:MPprofile}
u(x) & = \frac{1}{2}\sum_{n=1}^N p_n \E^{-|x-x_n|}  
\end{align}
for some $N\in\N\cup\{0\}$,  $x_1,\ldots,x_N\in\R$ and $p_1,\ldots,p_N\in\R$. 

\begin{remark}
 With respect to the transformation $(u,\mu)\mapsto (\infty,\tilde{\omega},\tilde{\dip})$ described in Section~\ref{secCH:02}, one sees that a pair $(u,\mu)$ in $\CHdom$ is a multi-peakon profile if and only if the function $\tilde\Wr$ defined by~\eqref{eq:DvsStringWr} is piecewise constant (with finitely many steps) and the measure $\tilde\dip$ defined by~\eqref{eq:DvsStringDip} has finite support.
In this case, the corresponding spectral problem~\eqref{eqnGISODE} allows a complete direct and inverse spectral theory, which goes back to work of M.\ G.\ Krein and H.\ Langer on the indefinite moment problem~\cite{krla79} (see~\cite{IndMoment} and~\cite{StieltjesType} for further details).    
\end{remark}

A pair $(u,\mu)$ in $\CHdom$ is a multi-peakon profile if and only if the support of the corresponding spectral measure  $\rho$ is a finite set. 
In this case, the pair $(u,\mu)$ can be recovered explicitly in terms of the moments of the  spectral measure $\rho$; see~\cite{besasz00} and \cite[Section~4]{ConservMP}. 
Moreover, the set of all multi-peakon profiles $\Peakons$ is clearly invariant under the conservative Camassa--Holm flow $\Phi$ and it has been proved in~\cite{ConservMP} that the conservative Camassa--Holm flow on $\Peakons$ gives rise to the same conservative multi-peakon solutions that had been constructed before in~\cite{brco07} and~\cite{hora07a,hora07}. 
The main result of~\cite[Section~5]{Eplusminus} asserts that this flow on $\Peakons$ extends continuously and uniquely to bounded subsets of $\CHdom$, which leads to a well-posedness result of the two-component Camassa--Holm system on $\CHdom$. 
More precisely, let us first define the bounded sets 
\begin{align}
\CHdom(R) = \{(u,\mu)\in\CHdom\, |\, E(u,\mu)\le R\}
\end{align} 
for each positive $R>0$, where the functional $E$ is given by 
\begin{align}
E(u,\mu) = \sup_{x\in \R}\, \E^{\frac{x}{2}} \biggl(\int_{x}^{\infty}\E^{-s}(u(s) + u'(s))^2ds + \int_{[x,\infty)}\E^{-s}d\dip(s)\biggr)^{\nicefrac{1}{2}}.
\end{align}
The decisive role of this functional is that it controls the size of the {\em spectral gap} 
\begin{align}
 \lambda_0(\rho) = \inf\{|\lambda|\,|\,\lambda\in\supp(\rho)\}
\end{align}
around zero of the spectral measure $\rho$ corresponding to $(u,\mu)$. 
 More explicitly, we have shown in~\cite[Proposition~3.5]{Eplusminus} that one always has 
\begin{align}\label{eq:lam0est}
    \frac{1}{6\lambda_0(\rho)} \leq E(u,\mu) \leq \frac{\sqrt{2}}{\lambda_0(\rho)}.
\end{align}
Because the spectral gap is clearly preserved under the flow $\Phi$, this allows one to control $E$ globally by  
\begin{align}\label{eq:mainConservLaw}
\frac{1}{6\sqrt{2}}E(u,\mu)\leq E(\Phi^t(u,\mu)) \leq 6\sqrt{2}\, E(u,\mu).
\end{align}
The latter is crucial for establishing that the conservative Camassa--Holm flow is continuous on the bounded sets $\CHdom(R)$; see~\cite[Proposition~5.2]{Eplusminus}. 

\begin{proposition}\label{prop:wellposedD}
The conservative Camassa--Holm flow $\Phi$ is continuous when restricted to $\CHdom(R)\times\R$ for every $R>0$. 
\end{proposition}

\begin{remark}
 Theorem~\ref{thm:consCHweak} and Proposition~\ref{prop:wellposedD} yield well-posedness of the two-component Camassa--Holm system~\eqref{eqnOurCH} on $\CHdom$ in the following sense:
The conservative Camassa--Holm flow extends uniquely from the set of all multi-peakon profiles to a continuous map on $\CHdom(R)\times \R$ for every $R>0$. 
In particular, for any $(u,\mu)\in\CHdom(R)$ and $t\in\R$, if a sequence of multi-peakon profiles $(u_n,\mu_n)\in \Peakons\cap \CHdom(R)$  converges to $(u,\mu)$ in $\CHdom$ and $t_n$ converges to $t$, then the multi-peakon solutions $\Phi^{t_n}(u_n,\mu_n)$ converge to $\Phi^t(u,\mu)$ in $\CHdom$. 
\end{remark}

The Szeg\H{o}-type theorems obtained in Section~\ref{secCH:02} enable us to improve on the above well-posedness result. 
More specifically, we are able to establish stability in stronger topologies when restricted to particular phase spaces. 
We will once more only state these results for the phase space $\CHdom^1$ for the sake of brevity. 

  \begin{theorem}\label{thm:wellposedD1}
 The set $\CHdom^1$ is invariant under the conservative Camassa--Holm flow $\Phi$ and the restricted flow $\Phi|_{\CHdom^1\times\R}\colon \CHdom^1\times\R\to \CHdom^1$ is continuous with respect to the topology on $\CHdom^1$.
 \end{theorem}
   
 \begin{proof}
 That the set $\CHdom^1$ is invariant under the flow $\Phi$ follows from Theorem~\ref{thmCHualpha} because the condition on the spectral measure there is preserved. 
 
 In order to prove continuity, the crucial observation is that even though our functional $\cE_1$ is not preserved by the flow $\Phi$, we are able to control its behavior:
  For $t\in\R$ and a pair $(u,\mu)$ in $\CHdom^1$ with spectral measure $\rho$ one has 
    \begin{align}\label{eq:E1t}
      \cE_1(\Phi^t(u,\mu)) = \cE_1(u,\mu)+ t + \int_\R \frac{\E^{-\frac{t}{2\lambda}}-1}{\lambda^2} d\rho(\lambda).
    \end{align}    
   Now let $t_n\in\R$ be a sequence that converges to $t$ and suppose that the sequence of pairs $(u_n,\mu_n)$ converges to $(u,\mu)$ in $\CHdom^1$.
  From the trace formula in Corollary~\ref{corTFCH1} (remember that the sum of the first, the second and the last term on the right-hand side of~\eqref{eqnTFCH1} is non-negative) and the fact that $F_2(\sqrt{1-4\lambda})\rightarrow\infty$ as $\lambda\to 0$, we first infer that the supports of the corresponding spectral measures $\rho_n$ have a uniform gap $(-\lambda_0,\lambda_0)$ around zero, so that the pairs $(u_n,\mu_n)$ belong to $\CHdom(R)$ for some $R>0$ in view of~\eqref{eq:lam0est}. 
 Taking into account Proposition~\ref{prop:wellposedD}, we see that the sequence $\Phi^{t_n}(u_n,\mu_n)$ converges to $\Phi^t(u,\mu)$ in $\CHdom$ and hence it only remains to show that 
   \begin{align*}
     \cE_1(\Phi^{t_n}(u_n,\mu_n)) \rightarrow \cE_1(\Phi^t(u,\mu)).
   \end{align*}
 Because of~\eqref{eq:E1t} and $\cE_1(u_n,\mu_n) \rightarrow \cE_1(u,\mu)$, it then suffices to show that 
   \begin{align*}
    \int_\R \frac{d\rho_n(\lambda)}{\lambda^2} & \rightarrow \int_\R \frac{d\rho(\lambda)}{\lambda^2}, & \int_\R \frac{\E^{-\frac{t_n}{2\lambda}}}{\lambda^2} d\rho_n(\lambda) & \rightarrow \int_\R \frac{\E^{-\frac{t}{2\lambda}}}{\lambda^2} d\rho(\lambda).
   \end{align*}
   However, this follows readily from the bound    
   \begin{align*}
    & \biggl|\int_\R \frac{\E^{-\frac{t_n}{2\lambda}}}{\lambda^2} d\rho_n(\lambda) - \int_\R \frac{\E^{-\frac{t}{2\lambda}}}{\lambda^2} d\rho(\lambda)\biggr| \\
    & \qquad \leq \sup_{\lambda\in\R\backslash(-\lambda_0,\lambda_0)} \bigl|\E^{-\frac{t_n}{2\lambda}}-\E^{-\frac{t}{2\lambda}}\bigr| \int_\R \frac{d\rho_n(\lambda)}{\lambda^2} + \biggl| \int_\R \frac{\E^{-\frac{t}{2\lambda}}}{\lambda^2} d\rho_n(\lambda) - \int_\R \frac{\E^{-\frac{t}{2\lambda}}}{\lambda^2} d\rho(\lambda) \biggr|
   \end{align*}  
   and the convergence $\rho_n\rightarrow\rho$ that we get from~\cite[Proposition~4.5]{Eplusminus}.
 \end{proof}
  
  \begin{remark}\label{remWPD1}
A few remarks are in order:
\begin{enumerate}[label=(\alph*), ref=(\alph*), leftmargin=*, widest=e]
           \item\label{rem:D1horagru} 
  Existence of global conservative solutions to the two-component Camassa--Holm system~\eqref{eqnOurCH} with initial data $(u,\mu)$ such that 
  \begin{align}
     u(x)  = \bar{u}(x) + c_-\chi(-x) + c_+\chi(x)
 \end{align}
  for some constants $c_\pm\in\R$, where $\bar{u}\in H^1(\R)$ and $\chi$ is a smooth non-decreasing function on $\R$ with support in $[0,\infty)$ such that $\chi(x)=1$ for $x\geq 1$, has been established in~\cite{grhora12b} by means of a transformation from Eulerian to Lagrangian variables. 
  These solutions are continuous with respect to a metric that is defined in terms of Lagrangian variables, which makes it somewhat involved in Eulerian variables. 
  However, its relation to more standard topologies given in~\cite[Proposition~5.2]{hora07} shows that it implies uniform convergence of the function $u$ as well as vague convergence of the measure $\mu$, which indicates that it may be close to our topology (compare with Lemma~\ref{lem:convergenceD}).   
\item\label{rem:D1integrable}
 Our results in the present section show that under a stronger decay assumption at $-\infty$ (compare the set $\CHdom^1$ with the phase space in \cite{grhora12b} when $c_-=0$; this case can be viewed as a regime with dispersion at one end), conservative solutions can be integrated by means of the inverse spectral transform. 
 We expect that this will allow to deduce qualitative properties of such solutions. 
\item\label{rem:LaguerreMP}
   The results in Section~\ref{secKreinLanger} characterize certain classes of pairs in $\CHdom^1$ that are made up of infinitely many peakons, where the positions of the peakons may only accumulate at $+\infty$. 
   One example of such a pair $(u,\mu)$ is given by 
   \begin{align}
   u(x) & = \frac{1}{2}\sum_{n\in\N} p_n \E^{-|x-x_n|}, & \dip & = 0, 
   \end{align}
   with the peakons' positions $x_n$ and their weights $p_n$ explicitly expressed via the Laguerre polynomials by
   \begin{align} 
   x_n & = \log\Biggl(\sum_{k=0}^{n-1} L_{k}(-1/4)^2\Biggr), & p_n & = \frac{\sum_{k=0}^{n-1} L_{k}(-1/4)^2}{nL_{n-1}(-1/4)L_{n}(-1/4)} .
   \end{align}
   In fact, this pair $(u,\mu)$ belongs to $\CHdom^1$ and corresponds (up to a negative point mass at zero) to the Krein--Stieltjes string in Example~\ref{eq:LaguerreCoeff} with $\alpha=\nicefrac{1}{4}$.   
   Using this connection and the asymptotics in~\eqref{eqnLagCoeff01}, one sees that 
     \begin{align} 
   x_n & = 2\sqrt{n} + \OO(1), & p_n & = \frac{1}{\sqrt{n}} + \OO(n^{-1}), 
   \end{align} 
   as $n\rightarrow\infty$. 
   We will not pursue an exploration of corresponding solutions to the two-component Camassa--Holm system here. 
   However, let us mention that even though $\CHdom^1$ does not contain any multi-peakon profiles (despite the fact that multi-peakon profiles are dense in $\CHdom$), one can show that pairs in $\CHdom^1$ that are made up of infinitely many peakons (with positions only accumulating at $+\infty$) are dense in $\CHdom^1$ (with respect to the $\CHdom^1$ topology).  
   \item  When restricted to the invariant subset $\CHdom^1\cap\CHdom^+$, the additional measure $\mu$ becomes superfluous and our global conservative solutions can be understood as weak solutions to the Camassa--Holm equation~\eqref{eqnCH}.
   \end{enumerate}
  \end{remark}
 
We conclude this section with a brief outline of further results that can be obtained in a similar way from condition~\eqref{eqnCHubeta} in Theorem~\ref{thm:SzegoStepLikeII} with $\beta=\nicefrac{1}{2}$.  
The corresponding phase space consists of all pairs $(u,\mu)$ in $\CHdom$ such that 
   \begin{align}
     \int_{0}^{\infty} u(x)^2 + u'(x)^2\, dx +  \int_{0}^{\infty} (\varrho(x)-1)^2dx + \int_{[0,\infty)} d\dip_{\sing} & < \infty, 
   \end{align}
  where $\varrho$ is the (positive) square root of the Radon--Nikod\'ym derivative of $\dip$ with respect to the Lebesgue measure and $\dip_\sing$ is the singular part of $\dip$.
A complete characterization of the spectral measures corresponding to this phase space is given by Theorem~\ref{thm:SzegoStepLikeII} with $\beta=\nicefrac{1}{2}$.
Furthermore, the trace formula in Corollary~\ref{coreqnTFbetaWT} readily translates to such pairs $(u,\mu)$ with spectral measure $\rho$ and gives 
\begin{align}\begin{split}\label{eqnTFchII}
  & \int_\R (1+\E^{-x})(u(x)+u'(x))^2 dx \\
  & \qquad + \int_\R (1+\E^{-x})\biggl(\varrho(x)-\frac{1}{1+\E^{-x}}\biggr)^2 dx + \int_{\R} (1+\E^{-x})d\dip_{\sing}(x)  \\
&  \qquad = \int_\R \frac{d\rho(\lambda)}{\lambda^2} - \frac{1}{2\pi} \int_{\R\backslash(-\frac{1}{2},\frac{1}{2})} \frac{\sqrt{4\lambda^2-1}}{|\lambda|^3} d\lambda +  2 \sum_{\lambda\in\supp(\rho)\atop|\lambda|<\nicefrac{1}{2}} \cF_2\Biggl(\sqrt{\frac{1-2\lambda}{1+2\lambda}}\Biggr) \\
&  \qquad\qquad - \frac{1}{2\pi} \int_{\R\backslash(-\frac{1}{2},\frac{1}{2})}  \frac{\sqrt{4\lambda^2-1}}{|\lambda|^3} \log\Biggl(\frac{2\pi|\lambda|}{\sqrt{4\lambda^2-1}} \frac{d\rho_\ac(\lambda)}{d\lambda}\Biggr) d\lambda.
\end{split}\end{align}
With the help of these two main ingredients, one can proceed along the same lines as in the discussion of the phase space $\CHdom^1$ above.  


\section{Applications to other operators}\label{secCO}

The aim of this final section is to point out possible applications of our results to other classical one-dimensional models including Dirac operators.

\subsection{Canonical systems}\label{ss:KSforCS}
 
 Let $h$ be a Hamiltonian according to Definition~\ref{def:Hamilt} (for a brief account on associated canonical systems and Weyl--Titchmarsh functions we refer to Appendix~\ref{app:D}) and define the function $H_{22}$ on $[0,\infty)$ by 
 \begin{align}\label{eq:defH22}
H_{22}(x) = \int_0^x h_{22}(s)ds.
 \end{align}
If the Hamiltonian $h$ is trace normalized, then $H_{22}$ is nothing but the function $\xi$ defined by~\eqref{eqnIPSurDefxi}.
Furthermore, we introduce the subsets 
 \begin{align}\label{eq:setOmega}
 \Omega &  = \{x\in [0,\infty)\,|\, h_{22}(x) = 0\}, & \Omega^c &  = [x_0+1,\infty)\backslash\Omega,
 \end{align}
  where the constant $x_0$ is defined by~\eqref{eqnC1fromH}, so that $H_{22}(x)>0$ when $x>x_0$. 
 We shall also use the logarithmic derivative of $H_{22}$ denoted by $\gh$, that is, we set 
 \begin{align}\label{eq:frakH}
\gh(x)  = \frac{d}{dx} \log H_{22}(x) = \frac{h_{22}(x)}{H_{22}(x)},
 \end{align}
 which is well-defined for almost all $x>x_0$. 
 Our first result will follow by applying Theorem~\ref{thm:KSforExB1} to canonical systems, where $\alpha$ is once more a positive constant.

\begin{theorem}\label{thm:KSforCS-I}
A Herglotz--Nevanlinna function $m$ is the Weyl--Titchmarsh function of a Hamiltonian $h$ with $h_{22}\notin L^1[0,\infty)$ and
\begin{align}\label{eqnKSforCS-I}
\int_\Omega H_{22}(x)h_{11}(x)dx + \int_{\Omega^c}  \frac{\det h(x)}{\gh(x)} dx + \int_{\Omega^c} \biggl(\frac{h_{12}(x)}{\gh(x)} + \frac{1}{4\alpha}\biggr)^2 \gh(x) dx<\infty, 
 \end{align}
 if and only if all the following conditions hold:
    \begin{enumerate}[label=(\roman*), ref=(\roman*), leftmargin=*, widest=iii]
    \item\label{itmKSc1CS} The function $m$ has a meromorphic extension to $\C\backslash[\alpha,\infty)$ that is analytic at zero with $m(0) = 0$.
 \item The negative poles $\sigma_-$ and the positive poles $\sigma_+$ of $m$ in $(-\infty,\alpha)$ satisfy~\eqref{eqnLT-I}.
\item The boundary values of the function $m$ satisfy~\eqref{eqnSzego-I}.
    \end{enumerate}
\end{theorem}  
  
 \begin{proof}
 We just need to use Theorem~\ref{thm:KSforExB1} together with the transformation between canonical systems and generalized indefinite strings as described in Appendix~\ref{app:D}. 
To this end, suppose first that $h$ is a trace normalized Hamiltonian and let $(L,\omega,\dip)$ be the corresponding generalized indefinite string with the same Weyl--Titchmarsh function $m$. 
 Since we then have the relation 
 \begin{align*}
    L = \int_0^\infty h_{22}(x)dx, 
 \end{align*}
 one sees that the condition $h_{22}\notin L^1[0,\infty)$ is equivalent to $L=\infty$.
  In this case, we next notice that the first integral in~\eqref{eqnCondS1} with $c= -(2\sqrt{\alpha})^{-1}$ is finite if and only if 
 \begin{align*}
    \int_\eps^\infty \biggl(\Wr(x) + \frac{1}{4\alpha x}\biggr)^2 x\, dx < \infty
 \end{align*} 
 for some $\eps>0$. 
 After a substitution (using~\cite[Corollary~5.4.4]{bo07}), we get
  \begin{align*}
\int_\eps^\infty \biggl(\Wr(x) + \frac{1}{4\alpha x}\biggr)^2 x\, dx  & =  \int_{x_0+1}^\infty \biggl(\Wr(\xi(x)) + \frac{1}{4\alpha \xi(x)}\biggr)^2 \xi(x)\xi'(x) dx \\
& = \int_{\Omega^c} \biggl(\Wr(\xi(x))\xi(x) + \frac{1}{4\alpha}\biggr)^2 \frac{\xi'(x)}{\xi(x)} dx \\
& = \int_{\Omega^c} \biggl(\xi'(x)\Wr(\xi(x))\frac{\xi(x)}{\xi'(x)} + \frac{1}{4\alpha}\biggr)^2 \frac{\xi'(x)}{\xi(x)} dx \\
& = \int_{\Omega^c} \biggl(\frac{h_{12}(x)}{\gh(x)} + \frac{1}{4\alpha}\biggr)^2 \gh(x) dx 
 \end{align*} 
 upon setting $\eps = \xi(x_0+1)>0$. 
 In order to treat the second integral in~\eqref{eqnCondS1}, we employ~\eqref{appD:dipviaCS} to obtain that  
 \begin{align}\begin{split}\label{eqnintxdip}
 \int_{[0,\infty)}x\, d\dip(x) &= \int_0^\infty \xi(x)(1-\xi'(x) - \xi'(x)\Wr(\xi(x))^2)dx \\
 &= \int_{\Omega} \xi(x) dx + \int_{[0,\infty)\backslash\Omega}\xi(x)(1-\xi'(x) - \xi'(x)\Wr(\xi(x))^2)dx \\
 &=   \int_\Omega H_{22}(x)dx + \int_{[0,\infty)\backslash\Omega} \frac{H_{22}(x)}{h_{22}(x)} \det h(x) dx.
 \end{split}\end{align}
From these equalities, we infer that condition~\eqref{eqnCondS1} with $c= -(2\sqrt{\alpha})^{-1}$ is equivalent to condition~\eqref{eqnKSforCS-I}.
Since the correspondence between trace normalized Hamiltonians and generalized indefinite strings is bijective, we may conclude from Theorem~\ref{thm:KSforExB1} that the conditions on the function $m$ in the claim are sufficient.
 
 Let us now consider an arbitrary Hamiltonian $h$ and define the trace normalized Hamiltonian $\tilde{h}$ as in Remark~\ref{remTraceNormal} (associated quantities will be denoted with a tilde sign), so that the corresponding Weyl--Titchmarsh functions coincide. 
 We then observe that according to~\eqref{eqnTracedH} one has 
 \begin{align*}
 \tilde{H}_{22}(x) = \int_0^x \frac{h_{22}(\gx^{-1}(s))}{\gx'(\gx^{-1}(s))} ds = H_{22}(\gx^{-1}(x)), 
 \end{align*}
 which shows that $\tilde{h}_{22}\notin L^1[0,\infty)$ if and only if $h_{22}\notin L^1[0,\infty)$. 
 Moreover, from the relations  
 \begin{align*}
 \tilde{\gh}(x) & = \frac{\gh(\gx^{-1}(x))}{\gx'(\gx^{-1}(x))}, & \det\tilde{h}(x) & = \frac{\det h(\gx^{-1}(x))}{\gx'(\gx^{-1}(x))^2}, 
 \end{align*}
  it follows that condition~\eqref{eqnKSforCS-I} is equivalent to the same condition for $\tilde{h}$. 
  Together with our considerations above and Theorem~\ref{thm:KSforExB1}, this proves that the conditions on the function $m$ are also necessary. 
 \end{proof}
 
 \begin{remark}
  We could have used any other real normalization of the function $m$ at zero in condition~\ref{itmKSc1CS} of Theorem~\ref{thm:KSforCS-I} instead, which then slightly changes the conditions on the Hamiltonian accordingly.
   Indeed, this more general claim can be deduced from Theorem~\ref{thm:KSforCS-I} by using the following simple fact: 
   If $h$ is a Hamiltonian with Weyl--Titchmarsh function $m$ and one defines another Hamiltonian by 
   \begin{align}\label{eqnRotH}
     \tilde{h} = \begin{pmatrix} 1 & -c \\ 0 & 1 \end{pmatrix} h \begin{pmatrix} 1 & 0 \\ -c & 1 \end{pmatrix} = \begin{pmatrix} h_{11} - 2ch_{12} + c^2 h_{22} & h_{12} - ch_{22} \\ h_{12} - ch_{22} & h_{22} \end{pmatrix}
   \end{align}
    for a given constant $c\in\R$, then the corresponding Weyl--Titchmarsh function $\tilde{m}$ is given by 
   \begin{align}
     \tilde{m}(z) = m(z) - c.
   \end{align}
   Since the bottom-right entries of $h$ and $\tilde{h}$ coincide, only the last integral in condition~\eqref{eqnKSforCS-I} will change under this transformation.
   Moreover, it is also possible to handle the case when the function $m$ is required to have a pole at zero instead of a finite value. 
   This can be deduced from Theorem~\ref{thm:KSforCS-I} again by using that the Weyl--Titchmarsh function of the Hamiltonian 
   \begin{align}\label{eqnHinv}
     \tilde{h} = \begin{pmatrix} 0 & -1\\ 1 & 0 \end{pmatrix}h\begin{pmatrix} 0 & 1\\ -1 & 0 \end{pmatrix} = \begin{pmatrix} h_{22} & -h_{12} \\ -h_{12} & h_{11} \end{pmatrix}
   \end{align}
   can be expressed in terms of the Weyl--Titchmarsh function $m$ via 
   \begin{align}
     \tilde{m}(z) = - \frac{1}{m(z)},
   \end{align}
   where we have to suppose that the function $h_{11}$ is not identically zero almost everywhere on $[0,\infty)$, so that $m$ is not identically zero. 
 \end{remark}
 
 \begin{remark}
  The condition that $h_{22}\notin L^1[0,\infty)$ in Theorem~\ref{thm:KSforCS-I} can be replaced with the stronger condition that $h_{11}\in L^1[0,\infty)$. 
  In fact, the function $h_{11}$ being integrable always implies that $h_{22}$ is not because the Hamiltonian is assumed to be in the limit-point case; see~\eqref{appD:lpcase}. 
  On the other hand, the condition that the Weyl--Titchmarsh function $m$ has an analytic extension to zero with $m(0)=0$ implies that zero is an eigenvalue for the Hamiltonian in~\eqref{eqnHinv}.
  Since the corresponding eigenfunctions are constant and identically zero in the first component, this guarantees that 
  \begin{align}
    \int_0^\infty h_{11}(x)dx = \int_0^x \begin{pmatrix} 0 & 1 \end{pmatrix} \begin{pmatrix} h_{22}(x) & -h_{12}(x) \\ -h_{12}(x) & h_{11}(x) \end{pmatrix} \begin{pmatrix} 0 \\ 1 \end{pmatrix} dx < \infty. 
  \end{align}
  Under this assumption, one also obtains the limit 
   \begin{align}\label{eqn?}
 \lim_{x\rightarrow\infty}\frac{\int_0^{x} |h_{12}(s)|ds}{\int_0^{x} h_{22}(s)ds} =  0
 \end{align}
  by using that $h_{12}^2\leq h_{11}h_{22}$ and applying the Cauchy--Schwarz inequality.
 \end{remark}
 
 In a similar way, we are able to derive another result from Theorem~\ref{thm:KSforExB2} for a positive constant $\beta$. 
 As the proof is not too different from the one for Theorem~\ref{thm:KSforCS-I}, we are going to omit most details and calculations.
 
 \begin{theorem}\label{thm:KSforCS-II}
A Herglotz--Nevanlinna function $m$ is the Weyl--Titchmarsh function of a Hamiltonian $h$ with $h_{22}\notin L^1[0,\infty)$ and
 \begin{align}\label{eqnKSforCS-II}
\int_\Omega H_{22}(x) h_{11}(x)dx + \int_{\Omega^c} \frac{h_{12}(x)^2}{\gh(x)}dx + \int_{\Omega^c} \biggl( \frac{\sqrt{\det h(x)}}{\gh(x)} - \frac{1}{2\beta} \biggr)^2 \gh(x) dx<\infty 
\end{align}
 if and only if all the following conditions hold:
    \begin{enumerate}[label=(\roman*), ref=(\roman*), leftmargin=*, widest=iii]
\item \label{itmKSc2CS} The function $m$ has a meromorphic extension to $\C_+\cup(-\beta,\beta)\cup\C_-$ that is analytic at zero with $m(0)=0$. 
\item The poles $\sigma_\dis$ of $m$ in $(-\beta,\beta)$ satisfy~\eqref{eqnLT-II}.
\item The boundary values of the function $m$ satisfy~\eqref{eqnSzego-II}.
    \end{enumerate}
\end{theorem}
    
  \begin{proof}   
  We are only going to explain briefly how condition~\eqref{eqnKSforCS-II} for a trace normalized Hamiltonian $h$ is related to condition~\eqref{eqnCondS2} with $c=0$ for the corresponding generalized indefinite string $(L,\omega,\dip)$. 
  Firstly, performing a substitution gives     
    \begin{align*}
\int_{\xi(x_0+1)}^\infty \Wr(x)^2 x\, dx  = \int_{x_0+1}^\infty \Wr(\xi(x))^2 \xi(x)\xi'(x) dx = \int_{\Omega^c} \frac{h_{12}(x)^2}{\gh(x)} dx. 
        \end{align*}
Through another substitution and the identity~\eqref{eqnxixi}, one also obtains
 \begin{align*}
   \int_{\xi(x_0+1)}^\infty   \biggl( \varrho(x) - \frac{1}{2\beta x} \biggr)^2 x\, dx & 
   = \int_{x_0+1}^\infty   \biggl( \varrho(\xi(x)) - \frac{1}{2\beta \xi(x)} \biggr)^2 \xi(x)\xi'(x) dx \\ 
   & = \int_{\Omega^c}   \biggl( \xi'(x)\varrho(\xi(x)) \frac{\xi(x)}{\xi'(x)} - \frac{1}{2\beta} \biggr)^2 \frac{\xi'(x)}{\xi(x)} dx \\
   & = \int_{\Omega^c}   \biggl( \frac{\sqrt{\det h(x)}}{\gh(x)} - \frac{1}{2\beta} \biggr)^2 \gh(x) dx.
    \end{align*}
  Finally, it remains to mention that the equality
    \begin{align*}
    \int_{[0,\infty)} x\, d\dip_\sing(x) = \int_\Omega H_{22}(x)dx
    \end{align*}
    follows from~\eqref{eqnintxdip} in the proof of Theorem~\ref{thm:KSforCS-I} because 
    \begin{align*}
        \int_{[0,\infty)\backslash\Omega} \frac{H_{22}(x)}{h_{22}(x)} \det h(x) dx & =  \int_{0}^{\infty} \varrho(\xi(x))^2 \xi(x) \xi'(x) dx = \int_0^\infty \varrho(x)^2 x\, dx,
    \end{align*}
    where we used~\eqref{eqnxixi} once more in the first step. 
    \end{proof}

\subsection{Dirac operators}\label{ss:Dirac}

Let us consider the one-dimensional Dirac system of the form (called {\em canonical} in~\cite[Section~7.1]{lesa91})
\begin{align}\label{eq:1dDirac}
\begin{pmatrix} 0 & 1\\ -1 & 0\end{pmatrix}f' + \begin{pmatrix} p & q\\ q & - p\end{pmatrix}f = zf.
\end{align}
We will assume that $p$ and $q$ are real-valued and locally square integrable functions on $[0,\infty)$. 
Together with the boundary condition $f_1(0)=0$ at zero, it is well known that the corresponding maximally defined operator in $L^2([0,\infty);\C^2)$ is self-adjoint; see~\cite[Section~8.6]{lesa91} for example. 
 Consequently, for all $z\in\C\backslash\R$ there is a unique (up to constant multiples) {\em Weyl solution} $\psi(z,\redot)$ of the system~\eqref{eq:1dDirac} that is square integrable.  
This allows one to define the Weyl--Titchmarsh function $m$ on $\C\backslash\R$ by 
\begin{align}\label{eq:m-Dirac}
m(z) = - \frac{\psi_2(z,0)}{\psi_1(z,0)}, 
\end{align}
which is a Herglotz–Nevanlinna function. 
Under the current assumptions on the potential, the function $m$ obeys the high energy asymptotic behavior
\begin{align}\label{eq:m-Dirac-asymp}
m(z) = \I+\oo(1)
\end{align}
 as $z\rightarrow\infty$ in any non-real sector in the upper complex half-plane $\C_+$. 
 In particular, these asymptotics imply that the integral representation of $m$ takes the form 
 \begin{align}\label{eqnWTRepDirac}
 m(z) =  c +  \int_\R \frac{1}{\lambda-z} - \frac{\lambda}{1+\lambda^2}\, d\rho(\lambda) 
\end{align}
for some real constant $c$ and a positive Borel measure $\rho$ on $\R$ with~\eqref{eqnWTrhoPoisson}.
Let us point out that unlike in our definition of the measure $\rho$ for generalized indefinite strings, we do not exclude possible point masses at zero here.
However, as for generalized indefinite strings, the measure $\rho$ defined in this way is a spectral measure for the previously mentioned self-adjoint operator associated with~\eqref{eq:1dDirac}.
 
\begin{remark} 
 In contrast to generalized indefinite strings and canonical systems, not every positive Borel measure $\rho$ on $\R$ with~\eqref{eqnWTrhoPoisson} is the spectral measure of a Dirac system (see~\cite{kre55} and also~\cite{den} for example). 
\end{remark}

\begin{example}\label{xmpl:Diracbeta}
Let $p= 0$ and $q= -\beta$ for some real constant $\beta\neq 0$. 
Under these assumptions, the system~\eqref{eq:1dDirac} simplifies to 
\begin{align}\begin{split}
 f_2' - \beta f_2 & = z f_1, \\ -f_1' - \beta f_1 & = zf_2.
\end{split}\end{align}
 For $z\in\C_+$ and $k\in\C_+$ with $k^2 = z^2 - \beta^2$, we find a Weyl solution $\psi(z,\redot)$ explicitly given by 
\begin{align}
\psi(z,x) = \begin{pmatrix} -z\\  \beta + \I k \end{pmatrix}\E^{\I k x},
\end{align}
 so that the corresponding Weyl--Titchmarsh function $m$ can be written as 
\begin{align}
m(z) = \frac{\beta + \I k}{z} = \frac{z}{\beta - \I k}. 
\end{align}
It is not at all surprising that this function is closely related to the Weyl--Titchmarsh function of the generalized indefinite string in Example~\ref{exa2alpha}. 
Namely, one has 
\begin{align}
m(z) = \begin{cases} m_\beta(z), & \beta>0,\\ \frac{2\beta}{z} + m_{-\beta}(z), & \beta<0, \end{cases}
\end{align}
where $m_\beta$ is the Weyl--Titchmarsh function in Example~\ref{exa2alpha}. 
Of course, this also gives a simple connection between the corresponding spectral measures.
\end{example}

\begin{example}[Dirac system with positive mass]
Let $p= \beta$ and $q= 0$ for some positive constant $\beta>0$.
Under these assumptions, the system~\eqref{eq:1dDirac} simplifies to 
\begin{align}\begin{split}
   f_2' + \beta f_1 & = z f_1, \\ -f_1' - \beta f_2 & = zf_2. 
\end{split}\end{align}
For $z\in\C_+$ and $k\in\C_+$ with $k^2 = z^2 - \beta^2$, we find a Weyl solution $\psi(z,\redot)$ explicitly given by 
\begin{align}
\psi(z,x) = \begin{pmatrix} \I k\\ z - \beta \end{pmatrix}\E^{\I kx},
\end{align}
so that the corresponding Weyl--Titchmarsh function $m$ can be written as 
\begin{align}
m(z) = \frac{\I k}{\beta + z} = \frac{\beta-z}{\I k}.
\end{align}
We note that $m$ has an analytic extension to zero with $m(0) = -1$ and that the corresponding spectral measure $\rho$ is given by
\begin{align}
\rho(B) = \frac{1}{\pi}\int_{B\backslash (-\beta,\beta)}\sqrt{\frac{\lambda - \beta}{\lambda+\beta}} d\lambda.
\end{align}
\end{example}

By transforming the Dirac system~\eqref{eq:1dDirac} to a canonical system~\eqref{appD:CS}, we are able to apply Theorem~\ref{thm:KSforCS-II} to Dirac operators. 
For the sake of keeping calculations simple, we will restrict our considerations to the cases of diagonal and off-diagonal potential matrices, that is, when either $q= 0$ or $p= 0$. 
We begin with the latter case, so that the potential matrix has the form
\begin{align}\label{eq:DirQss}
  \begin{pmatrix} 0 & q \\ q & 0 \end{pmatrix}.
\end{align}
It is not difficult to see that the potential matrix of a Dirac system is of this form if and only if the Weyl--Titchmarsh function $m$ is odd, which in turn is equivalent to the corresponding spectral measure $\rho$ being even. 
In order to state the main result of this section, we also introduce the functions
\begin{align}
\Qr(x) & = \exp\biggl(-\int_0^x q(s)ds\biggr), & \mathcal{Q}(x) & = \frac{\Qr(x)}{(\int_0^x \Qr(s)^2ds)^{\nicefrac{1}{2}}},
\end{align}
 and note that we continue to use $\beta$ for an arbitrary positive constant. 

\begin{theorem}\label{thm:KSforDir01}
An even positive Borel measure $\rho$ on $\R$ with~\eqref{eqnrhoPoisson} is the spectral measure of a Dirac system with potential matrix of the form~\eqref{eq:DirQss} satisfying $1/\Qr\in L^2[0,\infty)$ and 
\begin{align}\label{eq:DiracTrace01}
\int_1^\infty \biggl(\mathcal{Q}(x) - \frac{2\beta}{\mathcal{Q}(x)}\biggr)^2 dx < \infty
\end{align}
 if and only if all the following conditions hold:
    \begin{enumerate}[label=(\roman*), ref=(\roman*), leftmargin=*, widest=iii]
\item\label{itmKSDirac1} The transfer function $\Phi_\rho$ is well-defined on $[0,\infty)$ by 
\begin{align}
\Phi_\rho(x) = \int_{[0,\infty)} \frac{1-\cos(x \lambda)}{\lambda^2} d\rho(\lambda)
\end{align} 
  and belongs to $H^2_{\loc}[0,\infty)$.    
\item\label{itmKSDirac2} The support of $\rho$ is discrete in $(-\beta,\beta)$, does not contain zero and satisfies~\eqref{eqnLT-IIrho}. 
\item\label{itmKSDirac3} The absolutely continuous part $\rho_\ac$ of $\rho$ on $(-\infty,-\beta)\cup(\beta,\infty)$ satisfies~\eqref{eqnSzego-IIrho}. 
    \end{enumerate}
\end{theorem}

\begin{proof}
For a Dirac system with potential matrix of the form~\eqref{eq:DirQss} we define the matrix-valued function $U_0$ on $[0,\infty)$ by
\begin{align*}
U_0(x) = \begin{pmatrix} \E^{\int_0^x q(s)ds} & 0\\ 0 & \E^{-\int_0^x q(s)ds}\end{pmatrix} =  \begin{pmatrix} \Qr(x)^{-1} & 0\\ 0 & \Qr(x)\end{pmatrix}, 
\end{align*}
 which solves the system~\eqref{eq:1dDirac} with $z=0$. 
 Using this solution, we introduce the Hamiltonian $h$ by setting 
\begin{align*}
h(x) = U_0(x)^\ast U_0 (x) = \begin{pmatrix} \Qr(x)^{-2} & 0\\ 0 & \Qr(x)^2\end{pmatrix}.
\end{align*}
It is straightforward to verify that the Weyl--Titchmarsh function of the Dirac system then coincides with the Weyl--Titchmarsh function of the Hamiltonian $h$. 
In fact, the solution $\psi_h(z,\redot)$ for the corresponding canonical system defined as in~\eqref{eqnCSWeylSol} gives a Weyl solution $\psi(z,\redot)$ for the Dirac system by setting 
\begin{align*}
  \psi(z,x) = U_0(x)\psi_h(z,x);
\end{align*} 
 see~\cite[Section~6.4]{rem} for example. 
 Applying Theorem~\ref{thm:KSforCS-II} then shows that the corresponding spectral measure $\rho$ satisfies conditions~\ref{itmKSDirac2} and~\ref{itmKSDirac3} if and only if the Hamiltonian $h$ satisfies $h_{11}\in L^1[0,\infty)$ and~\eqref{eqnKSforCS-II}. 
 Taking into account that $\det h = 1$ and that the set $\Omega$ defined in~\eqref{eq:setOmega} is empty in this case, the latter is equivalent to $1/\Qr\in L^2[0,\infty)$ and~\eqref{eq:DiracTrace01}.
 Finally, in view of~\cite[Theorem~6.3, Theorem~14.4 and Equations~(2.23) and~(4.2)]{den}, it suffices to mention that condition~\ref{itmKSDirac1} is nothing but the local solvability condition ensuring that $\rho$ is indeed the spectral measure of a Dirac system (the latter goes back to work of M.\ G.\ Krein on Krein systems~\cite{kre55}; see also~\cite[Section~12]{den} for a detailed exposition of this theory) with potential matrix of the form~\eqref{eq:DirQss} since $\rho$ is even.
\end{proof}


\begin{remark}\label{rem:DiracTransf}
  It is also possible to obtain an analogue statement for Dirac systems with diagonal potential matrices from Theorem~\ref{thm:KSforDir01} by using the following simple observation:
  If $m$ is the Weyl--Titchmarsh function of the Dirac system
  \begin{align}
\begin{pmatrix} 0 & 1\\ -1 & 0\end{pmatrix}f' + \begin{pmatrix} p & 0\\ 0 & - p\end{pmatrix}f = zf
\end{align}
  and $\tilde{m}$ is the Weyl--Titchmarsh function of the Dirac system
 \begin{align}
\begin{pmatrix} 0 & 1\\ -1 & 0\end{pmatrix}f' + \begin{pmatrix} 0 & p\\ p & 0\end{pmatrix}f = zf,
\end{align}
 then these functions are related via   
\begin{align}
\tilde{m}(z) = \frac{m(z)-1}{m(z) + 1}.
\end{align} 
 The condition that the Weyl--Titchmarsh function $\tilde{m}$ is odd (that the corresponding spectral measure $\tilde{\rho}$ is even) turns into the condition 
\begin{align}
  m(-z)m(z) = 1
\end{align}
on the Weyl--Titchmarsh function $m$ in the diagonal case. 
\end{remark}

\begin{remark}
 As a final remark, let us mention that one may also apply Theorem~\ref{thm:KSforDir01} to one-dimensional Schr\"odinger operators by using the well-known supersymmetry relations (see~\cite[Section~15]{den} for example). 
  \end{remark}

\appendix

\section{Meromorphic functions and trace formulas, \ref*{thm:KSforExB1}}\label{appMeroMain}

Our main goal here is to prove Theorem~\ref{thm:FactorMain}, an auxiliary technical result for application in the proof of the relative trace formula in  Section~\ref{secSbSsumrule}. 
To this end, let us first recall (see \cite[Lecture~16]{le96} for example) that an entire function $\Phi$ belongs to the {\em Cartwright class} if it is of finite exponential type and satisfies 
  \begin{align}\label{eqnCondCart}
  \int_\R \frac{\log_+ |\Phi(\lambda)|}{1+\lambda^2}d\lambda <\infty,
  \end{align} 
  where $\log_+$ is the positive part of the logarithm. 
  We are also going to use the two functions $\cF_1$ and $\cF_2$ defined on $(0,1)\cup(1,\infty)$ by~\eqref{eq:Fsdef} in Section~\ref{secSbSsumrule} (their crucial properties are collected in Appendix~\ref{app:F1F2}).   
 
\begin{theorem}\label{thm:FactorMain}
 Let $\alpha>0$ and let $a$ be a meromorphic function on $\C_+$ with $a(\I\sqrt{\alpha})=1$ and only simple zeros and poles. 
 Suppose that $a$ can be written as 
 \begin{align}\label{eqnGprodF}
 a(k) = \frac{\Phi_+(k^2+\alpha)}{\Phi_-(k^2+\alpha)} \prod_{n=1}^N \frac{G_{+,n}(k)}{G_{-,n}(k)}
 \end{align}
 for some $N\in\N$ and functions $\Phi_\pm$, $G_{\pm,1},\ldots,G_{\pm,N}$, where: 
\begin{enumerate}[label=(\roman*), ref=(\roman*), leftmargin=*, widest=ii]
 \item The real entire functions $\Phi_\pm$ are of Cartwright class with $\Phi_\pm(0)=1$ and only real and simple zeros.
 \item The functions $G_{\pm,1},\ldots,G_{\pm,N}$ are meromorphic on $\C_+$ but not identically zero and such that $\im\,G_{\pm,n}(k)\geq0$ when $\re\,k>0$ and $\im\, G_{\pm,n}(k)\leq0$ when $\re\,k<0$ for all $n=1,\ldots,N$.
\end{enumerate}
 Then the limit $a(\xi)=\lim_{\varepsilon\downarrow0} a(\xi+\I\varepsilon)$ exists and is nonzero for almost all $\xi\in\R$, satisfies 
  \begin{align}\label{eq:FlogSzego}
      \int_{\R} \frac{|\log|a(\xi)||}{1+\xi^4}d\xi <\infty,
   \end{align} 
 and the identities 
   \begin{align}\label{eq:Gtrace01}
   \begin{split}
   &  \dot{a}(\I\sqrt{\alpha})  \\
    & \qquad =  \frac{2\I \alpha}{\pi} \int_{\R} \frac{\log|a(\xi)|}{(\xi^2+\alpha)^2}d\xi +  \frac{\I}{\sqrt{\alpha}}\lim_{\delta\downarrow0} \sum_{\kappa\in{\rm Z}\atop\delta<\kappa<1/\delta} \cF_1\biggl(\frac{\kappa}{\sqrt{\alpha}}\biggr) - \sum_{\eta\in{\rm P}\atop\delta<\eta<1/\delta}\cF_1\biggl(\frac{\eta}{\sqrt{\alpha}}\biggr), 
     \end{split}
 \\ \label{eq:Gtrace02}
   \begin{split}
& \ddot{a}(\I\sqrt{\alpha}) - \dot{a}(\I\sqrt{\alpha})^2 - \frac{\I}{\sqrt{\alpha}} \dot{a}(\I\sqrt{\alpha}) \\
& \qquad = \frac{8\sqrt{\alpha}}{\pi} \int_{\R} \frac{\xi^2 \log|a(\xi)|}{(\xi^2+\alpha)^3}d\xi  +  \frac{1}{\alpha}\lim_{\delta\downarrow0}  \sum_{\kappa\in{\rm Z}\atop\delta<\kappa<1/\delta} \cF_2\biggl(\frac{\kappa}{\sqrt{\alpha}}\biggr) - \sum_{\eta\in{\rm P}\atop\delta<\eta<1/\delta} \cF_2\biggl(\frac{\eta}{\sqrt{\alpha}}\biggr),
     \end{split}
   \end{align}
   hold, where ${\rm Z} = \{\kappa>0\,|\, a(\I\kappa)=0\}$ and ${\rm P} = \{\eta>0\,|\, a(\I\eta)=\infty\}$.
\end{theorem}

The proof of this theorem will rely on several ingredients.  
We begin with a factorization of some meromorphic functions that is essentially due to B.\ Simon~\cite{sim04}.

\begin{theorem}\label{th:SimonFactor}
   Let $\kappa_0>0$ and let $G$ be a meromorphic function on $\C_+$ with $\pm G(\I\kappa_0)>0$ and such that $\im\,G(k)\geq0$ when $\re\,k>0$ and $\im\, G(k)\leq0$ when $\re\,k<0$. 
 Then the limit $G(\xi)=\lim_{\varepsilon\downarrow0} G(\xi+\I\varepsilon)$ exists and is nonzero for almost all $\xi\in\R$, satisfies  
  \begin{align}\label{eq:SimonLog}
      \int_{\R} \frac{|\log|G(\xi)||}{1+\xi^2}d\xi <\infty,
   \end{align} 
 and the function $G$ admits the factorization 
   \begin{align}\label{eq:SimonFactor}
     G(k) =\pm B(k) \exp\biggl(\frac{\I}{\pi} \int_{\R} \frac{1+k\xi}{k-\xi}\frac{\log|G(\xi)|}{1+\xi^2}d\xi\biggr), 
   \end{align}
where the meromorphic function $B$ is the Blaschke-type product given by
   \begin{align}\label{eq:SimonBlaschke}
B(k)  & =   \lim_{\delta\downarrow0} \prod_{\kappa\in{\rm Z}\atop\delta<\kappa<1/\delta} \frac{\kappa_0-\kappa}{|\kappa_0-\kappa|} \frac{k- \I\kappa}{k+\I\kappa} \prod_{\eta\in{\rm P}\atop\delta<\eta<1/\delta}  \frac{\kappa_0-\eta}{|\kappa_0-\eta|} \frac{k+\I \eta}{k-\I \eta}  
   \end{align} 
  with ${\rm Z} = \{\kappa>0\,|\, G(\I\kappa)=0\}$ and ${\rm P} = \{\eta>0\,|\, G(\I\eta)=\infty\}$, and the convergence holds locally uniformly on $\C_+$ away from the poles of $G$. 
 \end{theorem}
  
  \begin{proof}
    This follows from~\cite[Theorem~1.1]{sim04} applied to the function $G\circ\varphi$ on $\mathbb{D}$, where $\mathbb{D}$ is the open unit disc in the complex plane and $\varphi:\mathbb{D}\rightarrow\C_+$ is given by 
    \begin{align*}
      \varphi(z) & = \I\kappa_0\frac{1- z}{1+z} = \kappa_0\frac{2\,\im\,z}{1+|z|^2+2\,\re\,z} + \I\kappa_0 \frac{1-|z|^2}{1+|z|^2+2\,\re\,z}.
    \end{align*}
    Note that the limit in our Blaschke-type product in~\eqref{eq:SimonBlaschke} is slightly different to the one in~\cite[Theorem~1.1]{sim04}. 
    However, as also mentioned in the proof of~\cite[Theorem~1.1]{sim04}, they are indeed the same in view of~\cite[Theorem~2.1]{sim04}.
  \end{proof}   

  We will also need some more results about real entire functions of Cartwright class.
  The first one is a technical fact, which is probably known to experts, but we were not able to locate it explicitly in the literature. 
   
  \begin{lemma}\label{lem:phi-estimate}
  Let $\Phi$ be a real entire function of Cartwright class with $\Phi(0)=1$ and only real and simple zeros. 
  For each $r>0$, define the polynomial $p_r$ by 
  \begin{align}\label{eqn:prodPR}
   p_r(z) = \prod_{\lambda\in\Lambda\atop|\lambda|< r} \biggl(1-\frac{z}{\lambda}\biggr), 
  \end{align} 
  where $\Lambda$ is the set of zeros of $\Phi$. 
  Then there is a constant $C>0$ such that
  \begin{align}\label{eq:phi-estimate}
    |\log|p_r(z)||  \leq C |z| (\log(1+|z|)+1) + |\log|\Phi(z)||
  \end{align}
  for all $r>0$ and all $z\in\R\backslash\Lambda$. 
  \end{lemma}
   
  \begin{proof}
Since  $\Phi$ is an entire function of Cartwright class, one has 
\begin{align*}
 M & =\lim_{r\rightarrow\infty} \sum_{\lambda\in\Lambda\atop|\lambda|<r}\frac{1}{\lambda}, &  \sup_{r>0}\,\biggl|\sum_{\lambda\in\Lambda\atop|\lambda|<r}\frac{1}{\lambda}\biggr| & \leq C_0,
\end{align*}
for some constants $M$, $C_0\in\R$; see~\cite[Theorem~1 in Lecture~17]{le96}. 
Moreover, by the Hadamard theorem, the function $\Phi$ can be written as 
\begin{align}\label{eq:PhiProd}
  \Phi(z) = \E^{-Mz}\prod_{\lambda\in\Lambda} \biggl(1-\frac{z}{\lambda}\biggr) \E^{\frac{z}{\lambda}}.
\end{align}
 Because the polynomials $p_r$ admit the similar factorization 
  \begin{align}\label{eqnprFac}
   p_r(z) = \E^{-\sum_{\lambda\in\Lambda\atop|\lambda|<r}\frac{z}{\lambda}} \prod_{\lambda\in\Lambda\atop|\lambda|<r} \biggl(1-\frac{z}{\lambda}\biggr)\E^{\frac{z}{\lambda}}, 
 \end{align}
  one also gets   
 \begin{align*}
   \frac{\Phi(z)}{p_r(z)} = \E^{-Mz+\sum_{\lambda\in\Lambda\atop|\lambda|<r}\frac{z}{\lambda}} \prod_{\lambda\in\Lambda\atop|\lambda|\geq r} \biggl(1-\frac{z}{\lambda}\biggr)\E^{\frac{z}{\lambda}}
 \end{align*}
 away from the set $\Lambda$.
Applying~\cite[Theorem~2 in Lecture~4]{le96} to estimate the canonical product on the right-hand side then yields 
 \begin{align*}
  \log\biggl|\frac{\Phi(z)}{p_r(z)}\biggr| & \leq 2C_0 |z| + C_1 |z| \biggl(\int_{0}^{|z|} \frac{n_r(t)}{t^2}dt + |z|\int_{|z|}^\infty \frac{n_r(t)}{t^3}dt\biggr)
 \end{align*}
 for all $z\in\R\backslash\Lambda$, where $C_1 = 12\E$ is an absolute constant and 
 \begin{align*}
   n_r(t) = \#\{\lambda\in\Lambda\,|\, r\leq|\lambda|\leq t\} \leq \#\{\lambda\in\Lambda\,|\, |\lambda|\leq t\} = n(t). 
 \end{align*}
However, as the function $\Phi$ is of Cartwright class, one has $n(t) \leq C_2t$ for some positive constant $C_2$; see~\cite[Theorem~1 in Lecture~17]{le96}. 
From this we get 
 \begin{align}\label{eqnprminus}\begin{split}
  -\log|p_r(z)| & = \log\biggl|\frac{\Phi(z)}{p_r(z)}\biggr| - \log|\Phi(z)| \\
   &  \leq 2C_0|z|  + C_1 C_2 |z| (\log(1+|z|) + |\log\lambda_0| + 1) + |\log|\Phi(z)||,
 \end{split}\end{align}
 where $\lambda_0=\min_{\lambda\in\Lambda}|\lambda|>0$, so that $n(t)=0$ when $t<\lambda_0$. 
 The same estimate for the canonical product representing $p_r$ in~\eqref{eqnprFac} gives  
 \begin{align*} 
   \log|p_r(z)| \leq C_0|z| + C_1 |z| \biggl(\int_{0}^{|z|} \frac{m_r(t)}{t^2}dt + |z|\int_{|z|}^\infty \frac{m_r(t)}{t^3}dt\biggr)
 \end{align*}
 for all $z\in\R\backslash\Lambda$, where 
 \begin{align*}
   m_r(t) = \#\{\lambda\in\Lambda\,|\, |\lambda|<r,\,|\lambda|\leq t\} \leq n(t) \leq C_2t.
 \end{align*}
 In a similar way as above, this results in the bound  
 \begin{align*}
   \log|p_r(z)| \leq C_0|z| + C_1C_2 |z| (\log(1+|z|) + |\log\lambda_0| + 1), 
 \end{align*}
which, together with~\eqref{eqnprminus}, implies~\eqref{eq:phi-estimate} for some constant $C>0$.
\end{proof}
  
 With this auxiliary result, we are now able to establish some identities for real entire functions of Cartwright class with only real and simple zeros.
  
\begin{theorem}\label{thm:CartwrightRepr}
  Let $\Phi$ be a real entire function of Cartwright class with $\Phi(0)=1$ and only real and simple zeros. 
  Then for every $\alpha>0$ one has 
   \begin{align}
    \label{eq:CartTrace01}    \dot{\Phi}(0) & = \frac{\sqrt{\alpha}}{\pi} \int_{\R} \frac{\log|\Phi(\xi^2+\alpha)|}{(\xi^2+\alpha)^2}d\xi +  \frac{1}{2\alpha} \sum_{\kappa\in{\rm Z}} \cF_1\biggl(\frac{\kappa}{\sqrt{\alpha}}\biggr),  \\
    \label{eq:CartTrace02}   \dot{\Phi}(0) - \alpha \ddot{\Phi}(0) + \alpha\dot{\Phi}(0)^2  
         & = \frac{2\sqrt{\alpha}}{\pi} \int_{\R} \frac{\xi^2\log|\Phi(\xi^2+\alpha)|}{(\xi^2+\alpha)^3}d\xi +  \frac{1}{4\alpha}\sum_{\kappa\in{\rm Z}} \cF_2\biggl(\frac{\kappa}{\sqrt{\alpha}}\biggr),
   \end{align} 
   where ${\rm Z} = \{\kappa>0\,|\, \Phi(-\kappa^2+\alpha)=0\}$. 
  \end{theorem}
   
   \begin{proof}
    For each fixed $\alpha>0$, the entire function $\Phi$ admits the representation
   \begin{align}\label{eq:CartwrightRepr}
     \Phi(k^2+\alpha) = \lim_{r\rightarrow\infty} C_r   \exp\biggl(\frac{\I}{\pi} \int_{\R} \frac{1+k\xi}{k-\xi}\frac{\log|p_r(\xi^2+\alpha)|}{1+\xi^2}d\xi\biggr)\prod_{\kappa\in{\rm Z}\atop\kappa<\sqrt{r+\alpha}} \frac{k-\I\kappa}{k+\I\kappa},
   \end{align}
 where $C_r$ are some complex constants with $|C_r|=1$, the polynomials $p_r$ are defined in~\eqref{eqn:prodPR} and the convergence holds locally uniformly for all $k\in\C_+$.  
 More specifically, since the zero set $\{k\in\C_+\,|\, p_r(k^2+\alpha) =0\}$ coincides with the set $\{\I\kappa\,|\, \kappa\in{\rm Z},~\kappa<\sqrt{r+\alpha}\}$ whenever $r>\alpha$, for sufficiently large $r$ one has the Nevanlinna factorization~\cite[Theorem~6.13]{roro94} 
     \begin{align}\label{eqnprNevFac}
       p_r(k^2+\alpha) & = C_r  \exp\biggl(\frac{\I}{\pi} \int_{\R} \frac{1+k\xi}{k-\xi}\frac{\log|p_r(\xi^2+\alpha)|}{1+\xi^2}d\xi\biggr) \prod_{\kappa\in{\rm Z}\atop\kappa<\sqrt{r+\alpha}} \frac{k-\I\kappa}{k+\I\kappa}
     \end{align}
     on the upper complex half-plane, where $C_r$ is some complex constant with $|C_r|=1$. 
      In view of~\cite[Remark~2 in Lecture 17]{le96}, this gives the claimed representation. 

 By taking the logarithm of the absolute value of~\eqref{eqnprNevFac} at $k=\I\sqrt{\alpha}$, we first get the identity 
\begin{align}\label{eqnPhir}
 0 =   \frac{\sqrt{\alpha}}{\pi}\int_\R \frac{\log|p_r(\xi^2+\alpha)|}{\xi^2+\alpha}d\xi +  \sum_{\kappa\in{\rm Z}\atop\kappa<\sqrt{r+\alpha}} \log\biggl| \frac{\kappa-\sqrt{\alpha}}{\kappa+\sqrt{\alpha}}\biggr|.
\end{align}
Furthermore, taking the logarithmic derivative of~\eqref{eq:CartwrightRepr} gives 
     \begin{align}\label{eqnPhiLogDev} 
 2k \frac{\dot{\Phi}(k^2+\alpha)}{\Phi(k^2+\alpha)} & = \lim_{r\rightarrow\infty}  - \frac{\I}{\pi} \int_{\R} \frac{\log|p_r(\xi^2+\alpha)|}{(k-\xi)^2}d\xi + \sum_{\kappa\in{\rm Z}\atop\kappa<\sqrt{r+\alpha}} \frac{2\I\kappa}{k^2+\kappa^2}
\end{align}
for all $k$ near $\I\sqrt{\alpha}$ (such that $\Phi(k^2+\alpha)\not=0$) and evaluating at $k=\I\sqrt{\alpha}$ yields (take into account that $\Phi$ is a real entire function)  
     \begin{align*}
 2\sqrt{\alpha} \dot{\Phi}(0) & = \lim_{r\rightarrow\infty}  - \frac{1}{\pi} \int_{\R} \frac{(\xi^2-\alpha) \log|p_r(\xi^2+\alpha)|}{(\xi^2+\alpha)^2}d\xi + \sum_{\kappa\in{\rm Z}\atop\kappa<\sqrt{r+\alpha}} \frac{2\kappa}{\kappa^2 - \alpha}.
\end{align*}
 After adding~\eqref{eqnPhir} divided by $\sqrt{\alpha}$ to the limit, we end up with 
     \begin{align}\label{eqndPhirfin}
 2\sqrt{\alpha} \dot{\Phi}(0) & = \lim_{r\rightarrow\infty}  \frac{2\alpha}{\pi} \int_{\R} \frac{\log|p_r(\xi^2+\alpha)|}{(\xi^2+\alpha)^2}d\xi +  \frac{1}{\sqrt{\alpha}} \sum_{\kappa\in{\rm Z}\atop\kappa<\sqrt{r+\alpha}} \cF_1\biggl(\frac{\kappa}{\sqrt{\alpha}}\biggr).
\end{align}
Differentiating~\eqref{eqnPhiLogDev} once more and evaluating at $k=\I\sqrt{\alpha}$ yields 
 \begin{align*}
  & 2 \dot{\Phi}(0) - 4\alpha \ddot{\Phi}(0) + 4\alpha \dot{\Phi}(0)^2 \\
  & \qquad\qquad =  \lim_{r\rightarrow\infty}  \frac{2\sqrt{\alpha}}{\pi} \int_{\R} \frac{(3\xi^2-\alpha)\log|p_r(\xi^2+\alpha)|}{(\xi^2+\alpha)^3}d\xi + \sum_{\kappa\in{\rm Z}\atop\kappa<\sqrt{r+\alpha}} \frac{4\sqrt{\alpha}\kappa}{(\kappa^2-\alpha)^2},
 \end{align*}
which, after adding~\eqref{eqndPhirfin} divided by $\sqrt{\alpha}$, becomes  
 \begin{align}\label{eqnddPhir}\begin{split}
  & 4 \dot{\Phi}(0) - 4\alpha \ddot{\Phi}(0) + 4\alpha \dot{\Phi}(0)^2 \\
  & \qquad\qquad =  \lim_{r\rightarrow\infty}  \frac{8\sqrt{\alpha}}{\pi} \int_{\R} \frac{\xi^2 \log|p_r(\xi^2+\alpha)|}{(\xi^2+\alpha)^3}d\xi + \frac{1}{\alpha} \sum_{\kappa\in{\rm Z}\atop\kappa<\sqrt{r+\alpha}}  \cF_2\biggl(\frac{\kappa}{\sqrt{\alpha}}\biggr).
\end{split}\end{align}
Finally, because of the asymptotics of the functions $\cF_1$ and $\cF_2$ at $\infty$ (see Appendix~\ref{app:F1F2}) 
 and the uniform bound on the polynomials $p_r$ in Lemma~\ref{lem:phi-estimate}, we can pass to the limit $r\rightarrow\infty$ in~\eqref{eqndPhirfin} and~\eqref{eqnddPhir} to arrive at the claimed identities. 
   \end{proof}   

We are now in position to prove the main result of this appendix.

\begin{proof}[Proof of Theorem~\ref{thm:FactorMain}]
 Choose a $\kappa_0>0$ such that $\I\kappa_0$ is neither a zero nor a pole of $G_{\pm,n}$ for all $n=1,\ldots,N$ and define the meromorphic function $Q$ on $\C_+$ by 
 \begin{align}\label{eqnDefQ}
   Q(k) = \prod_{n=1}^N \frac{G_{+,n}(k)}{G_{-,n}(k)}.
 \end{align} 
 Since the functions $G_{\pm,1},\ldots,G_{\pm,N}$ all satisfy the requirements of Theorem~\ref{th:SimonFactor}, it follows readily that the limit $Q(\xi) = \lim_{\varepsilon\downarrow0}Q(\xi+\I\varepsilon)$ exists and is nonzero for almost all $\xi\in\R$ and that this limit satisfies 
   \begin{align}\label{eqnQint}
      \int_{\R} \frac{|\log|Q(\xi)||}{1+\xi^2}d\xi <\infty.
   \end{align} 
  Moreover, Theorem~\ref{th:SimonFactor} guarantees that the function $Q$ admits the factorization
  \begin{align}\label{eqnQfac}
    Q(k) = \sgn(Q(\I\kappa_0))  \exp\biggl(\frac{\I}{\pi} \int_{\R} \frac{1+k\xi}{k-\xi}\frac{\log|Q(\xi)|}{1+\xi^2}d\xi\biggr)  \lim_{\delta\downarrow0} \prod_{n=1}^N \frac{B_{+,n,\delta}(k)}{B_{-,n,\delta}(k)},
  \end{align}
  where the rational functions $B_{\pm,n,\delta}$ are given by
   \begin{align*}
B_{\pm,n,\delta}(k)  & =    \prod_{\kappa\in{\rm Z}_{\pm,n}\atop\delta<\kappa<1/\delta} \frac{\kappa_0-\kappa}{|\kappa_0-\kappa|} \frac{k- \I\kappa}{k+\I\kappa} \prod_{\eta\in{\rm P}_{\pm,n}\atop\delta<\eta<1/\delta}  \frac{\kappa_0-\eta}{|\kappa_0-\eta|} \frac{k+\I \eta}{k-\I \eta}  
   \end{align*} 
  with ${\rm Z}_{\pm,n} = \{\kappa>0\,|\, G_{\pm,n}(\I\kappa)=0\}$ and ${\rm P}_{\pm,n} = \{\eta>0\,|\, G_{\pm,n}(\I\eta)=\infty\}$, and the convergence holds locally uniformly on $\C_+$ away from the zeros and poles of the functions $G_{\pm,1},\ldots,G_{\pm,N}$. 
  In fact, because $\I\sqrt{\alpha}$ is neither a zero nor a pole of $Q$ by our assumptions, we can remove $\sqrt{\alpha}$ from all the sets ${\rm Z}_{\pm,1},\ldots,{\rm Z}_{\pm,N}$ and ${\rm P}_{\pm,1},\ldots,{\rm P}_{\pm,N}$, since the corresponding factors just cancel out anyway. 
  As a consequence, the convergence also holds locally uniformly near $\I\sqrt{\alpha}$. 
 This allows us to take the logarithmic derivative of~\eqref{eqnQfac} near $\I\sqrt{\alpha}$ and arrive at  
 \begin{align}\label{eqnQTF01}
   \begin{split}\dot{Q}(\I\sqrt{\alpha}) & = \frac{2\I\alpha}{\pi} \int_\R \frac{\log|Q(\xi)|}{(\xi^2+\alpha)^2} d\xi + \lim_{\delta\downarrow0} \frac{\I}{\sqrt{\alpha}}\sum_{n=1}^N  \sum_{\kappa\in{\rm Z}_{+,n}\atop\delta<\kappa<1/\delta}  \cF_1\biggl(\frac{\kappa}{\sqrt{\alpha}}\biggr)  \\
    & \qquad  - \sum_{\eta\in{\rm P}_{+,n}\atop\delta<\eta<1/\delta} \cF_1\biggl(\frac{\eta}{\sqrt{\alpha}}\biggr) - \sum_{\kappa\in{\rm Z}_{-,n}\atop\delta<\kappa<1/\delta} \cF_1\biggl(\frac{\kappa}{\sqrt{\alpha}}\biggr) + \sum_{\eta\in{\rm P}_{-,n}\atop\delta<\eta<1/\delta}  \cF_1\biggl(\frac{\eta}{\sqrt{\alpha}}\biggr)  
 \end{split}
 \end{align}
 in the same way as in the proof of Theorem~\ref{thm:CartwrightRepr} (by also using the identity 
  \begin{align}\label{eqnQTF00}
  \begin{split}
   0 & = \frac{\sqrt{\alpha}}{\pi} \int_\R \frac{\log|Q(\xi)|}{\xi^2+\alpha} d\xi + \lim_{\delta\downarrow0} \sum_{n=1}^N  \sum_{\kappa\in{\rm Z}_{+,n}\atop\delta<\kappa<1/\delta}  \log\biggl|\frac{\kappa-\sqrt{\alpha}}{\kappa+\sqrt{\alpha}}\biggr|  \\
    & \qquad - \sum_{\eta\in{\rm P}_{+,n}\atop\delta<\eta<1/\delta} \log\biggl|\frac{\eta-\sqrt{\alpha}}{\eta+\sqrt{\alpha}}\biggr|  - \sum_{\kappa\in{\rm Z}_{-,n}\atop\delta<\kappa<1/\delta} \log\biggl|\frac{\kappa-\sqrt{\alpha}}{\kappa+\sqrt{\alpha}}\biggr| + \sum_{\eta\in{\rm P}_{-,n}\atop\delta<\eta<1/\delta} \log\biggl|\frac{\eta-\sqrt{\alpha}}{\eta+\sqrt{\alpha}}\biggr| 
    \end{split}
 \end{align}
 obtained from~\eqref{eqnQfac} after taking the logarithm of the absolute value at $k=\I\sqrt{\alpha}$). 
 Differentiating the logarithmic derivative of $Q$ once more and evaluating at $k=\I\sqrt{\alpha}$, as in the proof of Theorem~\ref{thm:CartwrightRepr}, yields  
 \begin{align}\label{eqnQTF02}
 \begin{split}
 &  \ddot{Q}(\I\sqrt{\alpha}) - \dot{Q}(\I\sqrt{\alpha})^2 - \frac{\I}{\sqrt{\alpha}} \dot{Q}(\I\sqrt{\alpha}) \\
 & \qquad = \frac{8\sqrt{\alpha}}{\pi} \int_\R \frac{\xi^2 \log|Q(\xi)|}{(\xi^2+\alpha)^3} d\xi+  \lim_{\delta\downarrow0} \frac{1}{\alpha}\sum_{n=1}^N  \sum_{\kappa\in{\rm Z}_{+,n}\atop\delta<\kappa<1/\delta}  \cF_2\biggl(\frac{\kappa}{\sqrt{\alpha}}\biggr)  \\
  & \qquad\qquad - \sum_{\eta\in{\rm P}_{+,n}\atop\delta<\eta<1/\delta} \cF_2\biggl(\frac{\eta}{\sqrt{\alpha}}\biggr)  -  \sum_{\kappa\in{\rm Z}_{-,n}\atop\delta<\kappa<1/\delta} \cF_2\biggl(\frac{\kappa}{\sqrt{\alpha}}\biggr) + \sum_{\eta\in{\rm P}_{-,n}\atop\delta<\eta<1/\delta}  \cF_2\biggl(\frac{\eta}{\sqrt{\alpha}}\biggr).
 \end{split}
 \end{align}
 It remains to notice that the limit $a(\xi)=\lim_{\varepsilon\downarrow0}a(\xi+\I\varepsilon)$ also exists and is nonzero for almost all $\xi\in\R$ because $\Phi_\pm$ are entire functions.
The limit satisfies~\eqref{eq:FlogSzego} due to~\eqref{eqnQint} and condition~\eqref{eqnCondCart} for entire functions of Cartwright class. 
 By combining the identities we proved for $Q$ with the ones we get from Theorem~\ref{thm:CartwrightRepr} applied to the functions $\Phi_\pm$, one obtains the claimed identities after taking into account cancellations in the sums and that $a$ has only simple zeros and poles. 
\end{proof}

\section{Meromorphic functions and trace formulas, \ref*{thm:KSforExB2}}\label{appMeroMainII}

For the proof of the relative trace formula in Section~\ref{secSbSsumrule2} we need another auxiliary technical result that is somewhat different to the one in Appendix~\ref{appMeroMain}. 
To this end, we are first going to establish some more identities for a real entire function $\Phi$ of Cartwright class that also involve its exponential type $\tau$, which is given by 
\begin{align}\label{eq:typePhi}
    \tau  =  \limsup_{s\rightarrow\infty} \frac{\log|\Phi(\I s)|}{s}
\end{align}
 under our assumptions on $\Phi$ (see \cite[Lecture~16]{le96}, \cite[Theorem~6.18]{roro94} for example).
 As in Section~\ref{secSbSsumrule2}, we fix $\beta>0$ and let $\zeta$ be the rational function given by 
 \begin{align}
   \zeta(k) & = \beta \frac{1+k^2}{1-k^2},
 \end{align}
 mapping the upper complex half-plane $\C_+$ conformally onto $\C_+\cup(-\beta,\beta)\cup\C_-$. 
 We also continue to use the functions $\cF_1$ and $\cF_2$ defined by~\eqref{eq:Fsdef}. 
 
\begin{theorem}\label{thm:CartwrightB}
  Let $\Phi$ be a real entire function of Cartwright class with $\Phi(0)=1$ and only real and simple zeros. 
  Then for every $\beta>0$ one has 
   \begin{align}
     \label{eq:CartTrace00B}    -\tau & = \frac{1}{\pi\beta} \int_{\R} \frac{\log|\Phi(\zeta(\xi))|}{1+\xi^2}d\xi + \frac{1}{\beta} \sum_{\kappa\in{\rm Z}} \log\biggl|\frac{\kappa-1}{\kappa+1}\biggr|, \\
    \label{eq:CartTrace01B}    \dot{\Phi}(0) - \tau & = \frac{2}{\pi\beta} \int_{\R} \frac{\log|\Phi(\zeta(\xi))|}{(1+\xi^2)^2}d\xi + \frac{1}{\beta} \sum_{\kappa\in{\rm Z}} \cF_1(\kappa),  \\
    \label{eq:CartTrace02B} - \beta\ddot{\Phi}(0) + \beta\dot{\Phi}(0)^2 -2\tau   & = \frac{8}{\pi\beta} \int_{\R} \frac{\xi^2 \log|\Phi(\zeta(\xi))|}{(1+\xi^2)^3}d\xi +  \frac{1}{\beta} \sum_{\kappa\in{\rm Z}} \cF_2(\kappa), 
   \end{align} 
   where ${\rm Z} = \{\kappa>0\,|\, \Phi\circ\zeta(\I\kappa)=0\}$ and $\tau$ is the exponential type of $\Phi$. 
  \end{theorem}
   
   \begin{proof}  
     The polynomials $p_r$ defined in Lemma~\ref{lem:phi-estimate} converge locally uniformly to $\Phi$ as $r\rightarrow\infty$ in view of~\cite[Remark~2 in Lecture 17]{le96}. 
     For large enough $r>\beta$, the zero set ${\rm Z}$ and $\{\kappa>0\,|\, p_r\circ\zeta(\I\kappa)=0\}$ 
      coincide and hence the rational function $p_r\circ\zeta$ admits the Nevanlinna factorization~\cite[Theorem~6.13]{roro94} 
     \begin{align}\label{eqnprNevFacB}
       p_r(\zeta(k)) & = C_r  \exp\biggl(\frac{\I}{\pi} \int_{\R} \frac{1+k\xi}{k-\xi}\frac{\log|p_r(\zeta(\xi))|}{1+\xi^2}d\xi\biggr) \prod_{\kappa\in{\rm Z}} \frac{k-\I\kappa}{k+\I\kappa}
     \end{align}
     on the upper complex half-plane, where $C_r$ is some complex constant with $|C_r|=1$. 
   Taking the logarithm of the absolute value of~\eqref{eqnprNevFacB} at $k=\I$ yields the identity 
\begin{align}\label{eqnPhirB}
 0 =   \frac{1}{\pi}\int_\R \frac{\log|p_r(\zeta(\xi))|}{1+\xi^2}d\xi +  \sum_{\kappa\in{\rm Z}} \log\biggl| \frac{\kappa-1}{\kappa+1}\biggr|.
\end{align}
Now applying~\cite[Theorem~8.2.1]{bo54} to the function $\tilde{\Phi}(z) = \E^{Mz}\Phi(z)$, which is the canonical product in~\eqref{eq:PhiProd}, and taking into account that $\Phi$ is a real entire function of Cartwright class, the exponential type of $\Phi$ is given by  
    \begin{align}\label{eqnExpTypePhi}
    \begin{split}
      - \tau & = \lim_{R\to \infty} \frac{1}{\pi} \int_{-R}^R \frac{\log|\tilde{\Phi}(\lambda)|}{\lambda^2}d\lambda\\
      &= \frac{1}{\pi} \int_{-\beta}^\beta \frac{\log|\Phi(\lambda)|-\dot{\Phi}(0)\lambda}{\lambda^2}d\lambda + \frac{1}{\pi} \int_{\R\backslash[-\beta,\beta]} \frac{\log|\Phi(\lambda)|}{\lambda^2}d\lambda.
      \end{split}
    \end{align}
    Due to locally uniform convergence of the polynomials $p_r$ to $\Phi$ as well as the uniform bound in Lemma~\ref{lem:phi-estimate}, the first integral in~\eqref{eqnExpTypePhi} can be written as  
     \begin{align*}
         \int_{-\beta}^\beta \frac{\log|\Phi(\lambda)|-\dot{\Phi}(0)\lambda}{\lambda^2}d\lambda = \lim_{r\rightarrow\infty}  \int_{-\beta}^\beta \frac{\log|p_r(\lambda)|-\dot{p}_r(0)\lambda}{\lambda^2}d\lambda.
    \end{align*}
Notice that the uniform estimate in the above integral near zero is ensured by the normalization $\Phi(0) = p_r(0) = 1$ and the locally uniform convergence of the corresponding second derivatives.     
    The second integral in~\eqref{eqnExpTypePhi} exists because $\Phi$ is of Cartwright class and becomes  
        \begin{align*}
       & \int_{\R\backslash[-\beta,\beta]} \frac{\log|\Phi(\lambda)|}{|\lambda|\sqrt{\lambda^2-\beta^2}}d\lambda + \int_{\R\backslash[-\beta,\beta]} \log|\Phi(\lambda)|\biggl(\frac{1}{\lambda^2}-\frac{1}{|\lambda|\sqrt{\lambda^2-\beta^2}}\biggr)d\lambda \\
         & \qquad =  \frac{1}{\beta} \int_{\R} \frac{\log|\Phi(\zeta(\xi))|}{1+\xi^2}d\xi + \lim_{r\rightarrow\infty}  \int_{\R\backslash[-\beta,\beta]} \log|p_r(\lambda)|\biggl(\frac{1}{\lambda^2}-\frac{1}{|\lambda|\sqrt{\lambda^2-\beta^2}}\biggr)d\lambda,
    \end{align*}
    where we performed a transformation $\lambda=\zeta(\xi)$ and passing to the limit is justified by the uniform bound in Lemma~\ref{lem:phi-estimate}. 
    After adding up in~\eqref{eqnExpTypePhi}, we get 
    \begin{align*}
      -\tau = \frac{1}{\pi\beta}\int_\R \frac{\log|\Phi(\zeta(\xi))|}{1+\xi^2}d\xi - \lim_{r\rightarrow\infty} \frac{1}{\pi}\int_{\R\backslash[-\beta,\beta]} \frac{\log|p_r(\lambda)|}{|\lambda|\sqrt{\lambda^2-\beta^2}}d\lambda,
    \end{align*}
    where we also took into account that $p_r$ has zero exponential type (apply~\cite[Theorem~8.2.1]{bo54} to $p_r$).
    It remains to perform a transformation $\lambda=\zeta(\xi)$ and use~\eqref{eqnPhirB} to arrive at the first identity~\eqref{eq:CartTrace00B}. 
    
   Taking the logarithmic derivative of~\eqref{eqnprNevFacB} gives 
     \begin{align}\label{eqnPhiLogDevB} 
  \frac{\dot{\zeta}(k)\dot{p}_r(\zeta(k))}{p_r(\zeta(k))} & =  - \frac{\I}{\pi} \int_{\R} \frac{\log|p_r(\zeta(\xi))|}{(k-\xi)^2}d\xi + \sum_{\kappa\in{\rm Z}} \frac{2\I\kappa}{k^2+\kappa^2}
\end{align}
for all $k$ near $\I$ (such that $p_r\circ \zeta(k)\not=0$) and evaluating at $k=\I$ yields (take into account that $p_r$ is a real polynomial)  
     \begin{align}\label{eqndPhirfinB}
    \dot{p}_r(0) & =   \frac{1}{\pi\beta} \int_{\R} \frac{(1-\xi^2) \log|p_r(\zeta(\xi))|}{(1+\xi^2)^2}d\xi + \frac{1}{\beta} \sum_{\kappa\in{\rm Z}} \frac{2\kappa}{\kappa^2 - 1}.
\end{align}
  Because the polynomials $p_r$ obey the uniform bound in Lemma~\ref{lem:phi-estimate}, we are able to pass to the limit $r\rightarrow\infty$ to obtain  
      \begin{align*}
    \dot{\Phi}(0) & =   \frac{1}{\pi\beta} \int_{\R} \frac{(1-\xi^2) \log|\Phi(\zeta(\xi))|}{(1+\xi^2)^2}d\xi + \frac{1}{\beta} \sum_{\kappa\in{\rm Z}} \frac{2\kappa}{\kappa^2 - 1}
\end{align*}
and after adding~\eqref{eq:CartTrace00B} we end up with the second identity~\eqref{eq:CartTrace01B}. 
Differentiating~\eqref{eqnPhiLogDevB}  once more and evaluating at $k=\I$ yields 
 \begin{align*}
   - \dot{p}_r(0) - \beta\ddot{p}_r(0) + \beta\dot{p}_r(0)^2  =   \frac{2}{\pi\beta} \int_{\R} \frac{(3\xi^2-1)\log|p_r(\zeta(\xi))|}{(1+\xi^2)^3}d\xi + \frac{1}{\beta} \sum_{\kappa\in{\rm Z}} \frac{4\kappa}{(\kappa^2-1)^2},
 \end{align*}
which, after adding~\eqref{eqndPhirfinB} and subtracting~\eqref{eqnPhirB} divided by $\beta$, becomes  
 \begin{align*}
     & - \beta \ddot{p}_r(0) + \beta \dot{p}_r(0)^2  \\
     & \qquad =   - \frac{2}{\pi\beta} \int_{\R} \frac{(1-\xi^2)^2 \log|p_r(\zeta(\xi))|}{(1+\xi^2)^3}d\xi + \frac{1}{\beta} \sum_{\kappa\in{\rm Z}}   \frac{2\kappa^3 + 2\kappa}{(\kappa^2-1)^2} -\log\biggl| \frac{\kappa-1}{\kappa+1}\biggr|.
 \end{align*}
Due to the bound in Lemma~\ref{lem:phi-estimate}, we may pass to the limit $r\rightarrow\infty$ and obtain 
 \begin{align*}
     & - \beta \ddot{\Phi}(0) + \beta \dot{\Phi}(0)^2  \\
     & \qquad =   - \frac{2}{\pi\beta} \int_{\R} \frac{(1-\xi^2)^2 \log|\Phi(\zeta(\xi))|}{(1+\xi^2)^3}d\xi + \frac{1}{\beta} \sum_{\kappa\in{\rm Z}}   \frac{2\kappa^3 + 2\kappa}{(\kappa^2-1)^2} -\log\biggl| \frac{\kappa-1}{\kappa+1}\biggr|.
 \end{align*}
 It remains to add~\eqref{eq:CartTrace00B} twice to get the last identity~\eqref{eq:CartTrace02B}. 
   \end{proof}  
  
  The main result of this appendix now follows in a similar way as Theorem~\ref{thm:FactorMain}. 
  
\begin{theorem}\label{thm:FactorMainB}
 Let $\beta>0$ and let $a$ be a meromorphic function on $\C_+$ with $a(\I)=1$ and only simple zeros and poles. 
 Suppose that $a$ can be written as 
 \begin{align}
 a(k) = \frac{\Phi_+(\zeta(k))}{\Phi_-(\zeta(k))} \prod_{n=1}^N \frac{G_{+,n}(k)}{G_{-,n}(k)}
 \end{align}
 for some $N\in\N$ and functions $\Phi_\pm$, $G_{\pm,1},\ldots,G_{\pm,N}$, where: 
\begin{enumerate}[label=(\roman*), ref=(\roman*), leftmargin=*, widest=ii]
 \item The real entire functions $\Phi_\pm$ are of Cartwright class with $\Phi_\pm(0)=1$ and only real and simple zeros.
 \item The functions $G_{\pm,1},\ldots,G_{\pm,N}$ are meromorphic on $\C_+$ but not identically zero and such that $\im\,G_{\pm,n}(k)\geq0$ when $\re\,k>0$ and $\im\, G_{\pm,n}(k)\leq0$ when $\re\,k<0$ for all $n=1,\ldots,N$.
\end{enumerate}
 Then the limit $a(\xi)=\lim_{\varepsilon\downarrow0} a(\xi+\I\varepsilon)$ exists and is nonzero for almost all $\xi\in\R$, satisfies 
  \begin{align}\label{eq:FlogSzegoB}
      \int_{\R} \frac{|\log|a(\xi)||}{1+\xi^2}d\xi <\infty,
   \end{align} 
 and the identities 
   \begin{align}
   \label{eq:Gtrace01B}
   \begin{split}
     & \beta(\tau_- - \tau_+)  \\
     & \qquad = \frac{1}{\pi} \int_{\R} \frac{\log|a(\xi)|}{1+\xi^2}d\xi +  \lim_{\delta\downarrow0} \sum_{\kappa\in{\rm Z}\atop\delta<\kappa<1/\delta} \log\biggl|\frac{\kappa-1}{\kappa+1}\biggr| - \sum_{\eta\in{\rm P}\atop\delta<\eta<1/\delta}\log\biggl|\frac{\eta-1}{\eta+1}\biggr|,
   \end{split}
  \\ \label{eq:Gtrace02B}
   \begin{split}
   &  \dot{a}(\I) + \I\beta(\tau_- - \tau_+) \\
    & \qquad =  \frac{2\I}{\pi} \int_{\R} \frac{\log|a(\xi)|}{(1+\xi^2)^2}d\xi +  \I \lim_{\delta\downarrow0} \sum_{\kappa\in{\rm Z}\atop\delta<\kappa<1/\delta} \cF_1(\kappa) - \sum_{\eta\in{\rm P}\atop\delta<\eta<1/\delta}\cF_1(\eta), 
     \end{split}
 \\ \label{eq:Gtrace03B}
   \begin{split}
& \ddot{a}(\I) - \dot{a}(\I)^2 - \I \dot{a}(\I) + 2\beta(\tau_- - \tau_+) \\
& \qquad = \frac{8}{\pi} \int_{\R} \frac{\xi^2 \log|a(\xi)|}{(1+\xi^2)^3}d\xi  +  \lim_{\delta\downarrow0}  \sum_{\kappa\in{\rm Z}\atop\delta<\kappa<1/\delta} \cF_2(\kappa) - \sum_{\eta\in{\rm P}\atop\delta<\eta<1/\delta} \cF_2(\eta),
     \end{split}
   \end{align}
   hold, where ${\rm Z} = \{\kappa>0\,|\, a(\I\kappa)=0\}$, ${\rm P} = \{\eta>0\,|\, a(\I\eta)=\infty\}$ and $\tau_\pm$ are the exponential types of $\Phi_\pm$.
\end{theorem} 

\begin{proof}
  We have seen in the proof of Theorem~\ref{thm:FactorMain} that for the meromorphic function $Q$  defined on $\C_+$ by~\eqref{eqnDefQ} the limit $Q(\xi)=\lim_{\varepsilon\downarrow0}Q(\xi+\I\varepsilon)$ exists and is nonzero for almost all $\xi\in\R$ and satisfies~\eqref{eqnQint}. 
  Moreover, the identities~\eqref{eqnQTF01}, \eqref{eqnQTF00} and~\eqref{eqnQTF02} hold with $\alpha=1$. 
 It remains to notice that the limit $a(\xi)=\lim_{\varepsilon\downarrow0}a(\xi+\I\varepsilon)$ also exists and is nonzero for almost all $\xi\in\R$ because $\Phi_\pm$ are entire functions.
The limit satisfies~\eqref{eq:FlogSzegoB} due to~\eqref{eqnQint} and condition~\eqref{eqnCondCart} for entire functions of Cartwright class. 
 By combining the identities~\eqref{eqnQTF01}, \eqref{eqnQTF00} and~\eqref{eqnQTF02} for $Q$ with the ones we get from Theorem~\ref{thm:CartwrightB} applied to the functions $\Phi_\pm$, one obtains the claimed identities after taking into account cancellations in the sums and that $a$ has only simple zeros and poles. 
\end{proof}

\section{Properties of \texorpdfstring{$\cF_1$}{F1} and \texorpdfstring{$\cF_2$}{F2}}\label{app:F1F2}

The purpose of this appendix is to collect a few facts about the functions $\cF_1$ and $\cF_2$, which appear in several trace formulas throughout this article. 
These two functions are defined on $(0,1)\cup(1,\infty)$ by 
\begin{align}
   \cF_1(s) & = \frac{2s}{s^2 - 1} + \log\biggl|\frac{s-1}{s+1}\biggr|, & \cF_2(s) & = \frac{2s^3 + 2s}{(s^2 - 1)^2} + \log\biggl|\frac{s-1}{s+1}\biggr|.
\end{align}
After computing the respective derivatives 
\begin{align}\label{eq:F2monot}
 \cF_1'(s) & = -\frac{4}{(s^2-1)^2}, & \cF_2'(s) & = -\frac{16s^2}{(s^2-1)^3}, 
\end{align}
we see that $\cF_1$ is strictly decreasing on $(0,1)$ and on $(1,\infty)$, whereas $\cF_2$ is strictly increasing on $(0,1)$ and strictly decreasing on $(1,\infty)$.
Because $\cF_2(s)$ tends to zero as $s\rightarrow0$ and as $s\rightarrow\infty$, it follows that $\cF_2$ takes only positive values. 

 In order to estimate the function $\cF_1$, we first write it as 
 \begin{align}
   \cF_1(s) = \int_s^\infty \frac{4}{(r^2-1)^2} dr
 \end{align}
 for $s\in(1,\infty)$, and then estimate the integrand to obtain 
 \begin{align}
   \frac{4}{3 s^3} = \int_s^\infty \frac{4}{r^4} dr < \cF_1(s) < \int_s^\infty \frac{4r}{(r^2-1)^{\nicefrac{5}{2}}} dr = \frac{4}{3(s^2-1)^{\nicefrac{3}{2}}}. 
 \end{align}
 Of course, this readily gives the asymptotics 
  \begin{align}
    \cF_1(s) = \frac{4}{3 s^3} + \oo(s^{-3})
  \end{align}
  as $s\rightarrow\infty$.
 On the other side, one has the estimate  
\begin{align}
  -4s - \frac{s^2}{2} =  - \int_0^s 4 + r\, dr < \cF_1(s) = \int_0^s -\frac{4}{(1-r^2)^2}dr  < -4s 
\end{align}
for $s$ close enough to zero, which yields the asymptotics $-4s/\cF_1(s) \rightarrow 1$ as $s\rightarrow0$.

In a similar way as for $\cF_1$, we write the function $\cF_2$ as
 \begin{align}
   \cF_2(s) = \int_s^\infty \frac{16r^2}{(r^2-1)^3} dr
 \end{align}
 for $s\in(1,\infty)$, so that we can estimate 
\begin{align}\label{eq:F2est01}
  \frac{16}{3(s^2-1)^{\nicefrac{3}{2}}} = \int_s^\infty \frac{16r}{(r^2-1)^{\nicefrac{5}{2}}} dr < \cF_2(s).  
\end{align}
Using the symmetry property $\cF_2(s^{-1}) = \cF_2(s)$, this readily gives the estimate 
\begin{align}\label{eq:F2est03}
   \frac{16s^{3}}{3(1-s^2)^{\nicefrac{3}{2}}} < \cF_2(s)    
\end{align}
for $s\in (0,1)$. 
Moreover, since one has 
\begin{align}
  \cF_2(s) = \int_0^s \frac{16r^2}{(1-r^2)^3}dr < \int_0^s 16r^2 + 8r^3\, dr = \frac{16s^3}{3} + 2s^4
\end{align}
for $s$ close enough to zero, we get the asymptotic behavior  
\begin{align}\label{eqnF2asymzero}
  \cF_2(s) & = \frac{16s^3}{3(1-s^2)^{\nicefrac{3}{2}}} + \oo(s^3), & s&\rightarrow0, \\ 
  \label{eqnF2asyminfty} \cF_2(s) & = \frac{16}{3(s^2-1)^{\nicefrac{3}{2}}} + \oo(s^{-3}), & s&\rightarrow\infty.
\end{align}

\section{Canonical first order systems}\label{app:D}

 The aim of this appendix is to briefly review some facts about {\em canonical systems} as far as they are needed in Section~\ref{secCO}. 
 For more details  we only refer the reader to the books~\cite{dB68}, \cite{rem} and~\cite{rom}. 
 In this article, we are going to use the following notion of a {\em Hamiltonian}. 
   
  \begin{definition} \label{def:Hamilt}
    A {\em Hamiltonian} is a locally integrable, real, symmetric and non-negative definite $2\times2$ matrix function
    \begin{align}
     h = \begin{pmatrix} h_{11} & h_{12} \\ h_{12} & h_{22} \end{pmatrix}
    \end{align}
    on $[0,\infty)$ such that the following properties hold:
    \begin{enumerate}[label=(\roman*), ref=(\roman*), leftmargin=*, widest=iii]
 \item The limit-point case prevails at infinity, meaning that 
 \begin{align}\label{appD:lpcase}
 \int_0^\infty \tr\, h(x)dx =  \int_0^\infty h_{11}(x)+ h_{22}(x)\,dx=\infty. 
 \end{align}
 \item The function $h$ does not vanish on subsets of positive Lebesgue measure.
\item  The function $h_{22}$ is not equal to zero almost everywhere on $[0,\infty)$. 
 \end{enumerate}
 If the Hamiltonian $h$ additionally satisfies  
 \begin{align}\label{eqnTraceNormal}
\tr\, h = h_{11} + h_{22} = 1
\end{align}
almost everywhere on $[0,\infty)$, then it is said to be {\em trace normalized}.
 \end{definition}
 
With each Hamiltonian $h$, one associates the canonical first order system  
 \begin{align}\label{appD:CS}
\begin{pmatrix} 0 & 1\\ -1 & 0 \end{pmatrix}f' = z h f, 
 \end{align} 
 where $z$ is a complex spectral parameter. 
 The fundamental matrix solution $U$ of the canonical system~\eqref{appD:CS} is the unique solution of the integral equation 
 \begin{align}
  U(z,x) = \begin{pmatrix} 1 & 0 \\ 0 & 1 \end{pmatrix} - z\int_0^x \begin{pmatrix} 0 & 1 \\ -1 & 0 \end{pmatrix} h(s) U(z,s)ds. 
 \end{align} 
 Using this solution, the Weyl--Titchmarsh function $m$ of the canonical system~\eqref{appD:CS} with Hamiltonian $h$ can be defined on $\C\backslash\R$ by 
 \begin{align}\label{appD:mWT-CS}
  m(z)=  \lim_{x\rightarrow\infty} \frac{U_{11}(z,x)}{U_{12}(z,x)}.
 \end{align}
 It is then possible to show that the solution $\psi(z,\redot)$ of the system~\eqref{appD:CS} given by 
 \begin{align}\label{eqnCSWeylSol}
    \psi(z,x) = U(z,x)  \begin{pmatrix} 1 \\ -m(z) \end{pmatrix}
 \end{align}
 for $z\in\C\backslash\R$ is square integrable with respect to $h$, meaning that 
 \begin{align}
   \int_0^\infty \psi(z,x)^\ast h(x) \psi(z,x) dx < \infty. 
 \end{align}
 We note that with these solutions, the Weyl--Titchmarsh function $m$ can in turn be expressed as the quotient 
 \begin{align}
   m(z) = - \frac{\psi_2(z,0)}{\psi_1(z,0)}. 
 \end{align}
 Finally, as a Herglotz--Nevanlinna function, the Weyl--Titchmarsh function $m$ of course admits an integral representation of the form~\eqref{eqnWTmIntRep}.
 
 \begin{remark}
The meaning of the coefficients $c_1$, $c_2$ and $L$ in the integral representation~\eqref{eqnWTmIntRep} are different for canonical systems when compared to generalized indefinite strings (see Remark~\ref{rem:uniquerho}). 
 For example, the constant $c_1$ in the linear term can be read off the Hamiltonian $h$ directly since 
 \begin{align}\label{eqnC1fromH}
  c_1 & = \int_0^{x_0} \tr\,h(x)dx,
\end{align}
where the point $x_0\in[0,\infty)$ is given by 
\begin{align}
  x_0 & = \sup\lbrace x\in[0,\infty) \,|\, h_{22} = 0 \text{ almost everywhere on }[0,x)\rbrace. 
\end{align}
\end{remark}

\begin{remark}\label{remTraceNormal}
Canonical systems with different Hamiltonians may have the same Weyl--Titchmarsh function. 
Indeed, introducing the function $\gx:[0,\infty)\rightarrow[0,\infty)$ by 
\begin{align}\label{eqnTraceH}
\gx(x) = \int_0^x \tr\, h(s)ds,
\end{align}
 which is strictly increasing and bijective, it is straightforward to check that  the canonical system with the new Hamiltonian defined by 
\begin{align}\label{eqnTracedH}
\tilde{h}(x) = \frac{1}{\tr\, h(\gx^{-1}(x))}h(\gx^{-1}(x))
\end{align}  
has the same Weyl--Titchmarsh function. 
In order to avoid such an ambiguity, it is customary to assume that the Hamiltonian is trace normalized. 
Notice that in this case, the function defined in~\eqref{eqnTraceH} is simply $\gx(x)=x$ for all $x\in[0,\infty)$.
\end{remark} 

 It is a fundamental result of L.~de Branges (see~\cite{rem,rom}) that the Weyl--Titchmarsh function determines the Hamiltonian almost everywhere up to a reparametrization (in particular, it uniquely determines trace normalized Hamiltonians) and that indeed all Herglotz--Nevanlinna functions arise as the Weyl--Titchmarsh function of some Hamiltonian. 
 Moreover, the correspondence $h\mapsto m$ on all trace normalized Hamiltonians, widely known as the {\em Krein--de Branges correspondence}, becomes homeomorphic with respect to suitable topologies. 
 
We are now going to briefly summarize a connection between trace normalized Hamiltonians and generalized indefinite strings; see~\cite[Section~6]{IndefiniteString}. 
For a given generalized indefinite string $(L,\omega,\dip)$, introduce the function $\varsigma\colon [0,L]\to [0,\infty]$ by 
\begin{align}\label{eqnVarsigma}
\varsigma(x) = x + \int_0^x \Wr(s)^2ds + \int_{[0,x)}d\dip
\end{align}
and define $\xi$ on $[0,\infty)$ as the generalized inverse of $\varsigma$ via 
\begin{align}\label{eqnXiasInverse}
\xi(x) = \sup\{s\in [0,L)\,| \, \varsigma(s)\le x\}.
\end{align}
Notice that $\xi$ is locally absolutely continuous with $0\le \xi'\le 1$ almost everywhere on $[0,\infty)$. 
Moreover, one has $\xi \circ\varsigma(x) = x$ for all $x\in[0,L)$ and 
\begin{align}\label{eqnXiSigma}
    \varsigma\circ \xi(x) & = \begin{cases} x, & x\in \ran{\varsigma},\\ \sup\lbrace s\in\ran{\varsigma}\,|\, s\leq x\rbrace, & x\not\in\ran{\varsigma}. \end{cases} 
  \end{align}
 It then follows that the function $h$ defined by (see~\cite[Equation~(6.10)]{IndefiniteString})
\begin{align}\label{eqnCSfromGIS}
h(x) = \begin{pmatrix} 1 - \xi'(x) & \xi'(x)\Wr(\xi(x)) \\ \xi'(x)\Wr(\xi(x)) & \xi'(x) \end{pmatrix}
\end{align}
is a trace normalized Hamiltonian. 
Furthermore, the corresponding fundamental matrix solution $U$ of the canonical system~\eqref{appD:CS} is given by 
\begin{align}\label{eqnFSSforCSviaGIS}
U(z,x) = \begin{pmatrix} 1 & 0 \\ z(x-\varsigma\circ \xi(x))) & 1 \end{pmatrix} 
             \begin{pmatrix} \theta(z,\xi(x)) & -z\phi(z,\xi(x)) \\ -z^{-1}\theta^\qd(z,\xi(x)) & \phi^\qd(z,\xi(x)) \end{pmatrix},
\end{align}
 which implies that the Weyl-Titchmarsh function of the Hamiltonian $h$ coincides with the Weyl--Titchmarsh function of the generalized indefinite string $(L,\omega,\dip)$.

 Conversely, in order to show how to transform a canonical system to a generalized indefinite string, let a trace normalized Hamiltonian $h$ be given. 
 We introduce the locally absolutely continuous, non-decreasing function $\xi$ on $[0,\infty)$ by  
  \begin{align}\label{eqnIPSurDefxi}
   \xi(x) = \int_0^x h_{22}(s)ds, 
  \end{align} 
  as well as its generalized inverse $\varsigma:[0,L]\rightarrow[0,\infty]$ via  
  \begin{align}\label{appD:xiinverse}
   \varsigma(x) = \sup\lbrace s\in[0,\infty)\,|\, \xi(s)< x\rbrace\cup\lbrace 0\rbrace. 
  \end{align}
  Here, the quantity $L\in(0,\infty]$ denotes the limit of $\xi(s)$ as $s\rightarrow \infty$ which is non-zero indeed due to our assumption on $h_{22}$.
  The function $\varsigma$ is readily verified to be strictly increasing and to satisfy $\xi \circ\varsigma(x) = x$ for all $x\in[0,L)$ as well as~\eqref{eqnXiSigma}.
  One can then show that the real-valued function $\Wr$ defined on $[0,L)$ by
  \begin{align}\label{appD:WrviaCS}
   \Wr(x) = \begin{cases} \frac{h_{12}}{h_{22}}\circ\varsigma(x), & \text{if }h_{22}\circ\varsigma(x)\not=0, \\ 0, & \text{if }h_{22}\circ \varsigma(x)=0, \end{cases}
  \end{align}
  belongs to $L^2_{\loc}[0,L)$ and thus gives rise to a real distribution $\omega$ in $H^{-1}_{\loc}[0,L)$. 
  Furthermore, we may define a positive Borel measure $\dip$ on $[0,L)$ via its distribution function by 
  \begin{align}\label{appD:dipviaCS}
   \int_{[0,x)} d\dip =  \int_0^{\varsigma(x)} 1 - \xi'(s) - \xi'(s)\Wr(\xi(s))^2\,ds. 
  \end{align}
  It is then possible to verify again that the Weyl--Titchmarsh function of the generalized indefinite string $(L,\omega,\dip)$ coincides with the Weyl--Titchmarsh function of the Hamiltonian $h$.  

\begin{remark}\label{rem:GISasCS}
 The considerations above reveal that generalized indefinite strings can be viewed as another instance of suitably normalized canonical systems. 
  Namely, reparametrizing the independent variable using~\eqref{eqnIPSurDefxi}, the Hamiltonian $h$ is transformed into a new one (possibly on a finite interval) with the bottom-right entry normalized to one. 
 However, with this choice of normalization, the top-left entry of the new Hamiltonian becomes a measure in general.
  The entries of this new Hamiltonian can be identified with the coefficients of a generalized indefinite string. 
\end{remark}
 
With the explicit connection between canonical systems and generalized indefinite strings described above, we are now in position to verify the expression~\eqref{eqnthetaphiExpTyp} for the exponential type of solutions stated in Section~\ref{sec:prelim}. 
 
 \begin{proof}[Proof of Theorem~\ref{thm:FSScartwright}]
  Let $h$ be the trace normalized Hamiltonian defined by~\eqref{eqnCSfromGIS}. 
  In view of relation~\eqref{eqnFSSforCSviaGIS}, except for the expression~\eqref{eqnthetaphiExpTyp} for the exponential type, all claims follow immediately from the corresponding facts for canonical systems; see~\cite[Section~4.4]{rem} or~\cite[Section~1]{krla14}. 
  However, we are also able to infer from~\eqref{eqnFSSforCSviaGIS} and the Krein--de Branges type formula (see~\cite[Theorem~4.26]{rem} or \cite[Theorem~11]{rom}), that the exponential type of the functions in~\eqref{eqnFSSstring} is given by
  \begin{align*}
   \int_0^{\varsigma(x)} \sqrt{\det h(s)}ds = \int_0^{\varsigma(x)} \sqrt{\xi'(s) - \xi'(s)^2 - \xi'(s)^2 \Wr(\xi(s))^2}ds.
  \end{align*}
  It therefore suffices to show that one has 
  \begin{align}\label{eqnxixi}
   \xi'(s) = \xi'(s)^2 + \xi'(s)^2 \Wr(\xi(s))^2 + \xi'(s)^2 \varrho(\xi(s))^2
  \end{align}
  for almost all $s\in[0,\infty)$, since the expression for the exponential type above then simplifies to   
  \begin{align*}
      \int_0^{\varsigma(x)} \xi'(s)\varrho(\xi(s))ds,
  \end{align*}
  which yields the integral in~\eqref{eqnthetaphiExpTyp} after a simple substitution (use~\cite[Corollary~5.4.4]{bo07} for example). 
  In order to verify~\eqref{eqnxixi}, let $D$ be the set of all $x\in(0,L)$ such that $\varsigma$ is differentiable at $x$ with 
    \begin{align*}
     \varsigma'(x) = 1 + \Wr(x)^2 + \varrho(x)^2.
  \end{align*}
  As $\varsigma$ is strictly increasing and given by~\eqref{eqnVarsigma}, the complement of $D$ has Lebesgue measure zero. 
  By performing another substitution, one then gets 
  \begin{align*}
   \int_0^x \id_{[0,L)\backslash D}(\xi(s))\xi'(s)ds = \int_0^{\xi(x)} \id_{[0,L)\backslash D}(s)ds = 0
  \end{align*}
  for every $x\in[0,\infty)$, which implies that 
  \begin{align*}
  \id_{[0,L)\backslash D}(\xi(s))\xi'(s) = 0
  \end{align*}
  for almost all $s\in[0,\infty)$. 
  For such an $s$, one clearly has~\eqref{eqnxixi} whenever $\xi'(s)=0$. 
  Otherwise, if $\xi'(s)\not=0$, then (we do not need to consider the case when $\xi(s)=L$ as this happens at most for one $s$ with $\xi'(s)\not=0$) $\xi(s)$ belongs to $D$ and one has  
  \begin{align*}
   \varsigma'(\xi(s)) = 1 + \Wr(\xi(s))^2 + \varrho(\xi(s))^2.
  \end{align*} 
  Since $\varsigma$ is continuous at $\xi(s)$, we may infer from~\eqref{eqnXiSigma} that $s$ necessarily belongs to $\ran{\varsigma}$. 
  Due to left-continuity of the function $\varsigma$, it then follows that there is a strictly increasing sequence $s_n$ in $\ran{\varsigma}$ that converges to $s$ (we may assume that $s$ is not zero of course). 
  We thus get from~\eqref{eqnXiSigma} that
   \begin{align*}
    \frac{\varsigma(\xi(s_n))-\varsigma(\xi(s))}{\xi(s_n)-\xi(s)} \frac{\xi(s_n)-\xi(s)}{s_n-s} = 1
   \end{align*}
  and it remains to take the limit as $n\rightarrow\infty$ to obtain    
  \begin{align*}
   \bigl(1+ \Wr(\xi(s))^2 + \varrho(\xi(s))^2\bigr) \xi'(s) =1,
  \end{align*} 
  which readily yields~\eqref{eqnxixi} again.  
 \end{proof}

\section*{Acknowledgments} 
We are grateful to the anonymous referee for carefully reading our manuscript and providing numerous helpful comments.


\end{document}